\documentclass{amsart}

\usepackage{graphicx}
\usepackage[utf8]{inputenc}
\usepackage{nicefrac}
\usepackage{mathabx}
\usepackage{amsrefs}
\usepackage{esint}
\usepackage{mathtools}
\mathtoolsset{showonlyrefs}
\usepackage{wrapfig}
\usepackage{xcolor}

\newtheorem{theorem}{Theorem}[section]
\newtheorem{lemma}[theorem]{Lemma}

\theoremstyle{definition}
\newtheorem{definition}[theorem]{Definition}

\newtheorem{prop}[theorem]{Proposition}
\newtheorem{cor}[theorem]{Corollary}
\newtheorem{ass}{Assumption}

\newtheorem{prob}{Problem}

\allowdisplaybreaks
\theoremstyle{remark}
\newtheorem{remark}[theorem]{Remark}

\numberwithin{equation}{section}

\allowdisplaybreaks

\newcommand{\dx}{\;\mathrm{d}x}
\newcommand{\dxi}{\;\mathrm{d}\xi}
\newcommand{\dy}{\;\mathrm{d}y}

\newcommand{\dt}{\;\mathrm{d}t}
\newcommand{\ds}{\;\mathrm{d}s}
\newcommand{\dz}{\;\mathrm{d}z}

\usepackage{amsmath}
\usepackage{braket} 
\newcommand{\R}{\mathbb{R}}
\newcommand{\N}{\mathbb{N}}

\newcommand{\Msym}{M_{\rm sym}}

\def\changeKen{\marginpar{\textcolor{teal}{changed by Kensuke}}
\textcolor{teal}}

\def\changeMar{\marginpar{\textcolor{orange}{changed by Marius}}
\textcolor{orange}}

\def\changeAnna{\marginpar{\textcolor{blue}{changed by Anna}}\textcolor{blue}}

\makeatletter
\@namedef{subjclassname@2020}{%
  \textup{2020} Mathematics Subject Classification}
\makeatother

\begin{document}

\title[An obstacle problem for the $p$-elastic energy]{An obstacle problem for the $p-$elastic energy}

\author{Anna Dall'Acqua}
\address{Universität Ulm, Institut für angewandte Analysis, 89081 Ulm}
\email{anna.dallacqua@uni-ulm.de}
\author{Marius Müller}
\address{Universität Augsburg, Institut für Mathematik, 86159 Augsburg}

\email{marius1.mueller@uni-a.de}
\author{Shinya Okabe}
\address{Mathematical Institute, Tohoku University, Sendai, 980-8578, Japan}
\email{shinya.okabe@tohoku.ac.jp}
\author{Kensuke Yoshizawa}
\address{Mathematical Institute, Tohoku University, Aoba, Sendai 980-8578, Japan; present address: Faculty of Education, Nagasaki University, 1-14 Bunkyo-machi, Nagasaki, 852-8521, Japan}
\email{k-yoshizaw@nagasaki-u.ac.jp
}

\date{\today}


\begin{abstract}
In this paper we consider an obstacle problem for 
a generalization of the $p$-elastic energy among graphical curves with fixed ends. Taking into account that the Euler--Lagrange equation has a degeneracy, 
we address the question whether solutions have a flat part, i.e. an open interval where the curvature vanishes.
We also investigate which is the main cause of the loss of regularity, the obstacle or the degeneracy. Moreover, we give several conditions on the obstacle  that assure existence and nonexistence of solutions. The analysis can be refined in the special case of the $p$-elastica functional, where we obtain sharp existence results and uniqueness for symmetric minimizers.
\end{abstract}
\maketitle

\tableofcontents

\section{Introduction} \label{section:intro}
This paper is concerned with an obstacle problem for the $p$-elastic energy defined by 
\begin{equation} \label{eq:p-elastic_energy}
\mathcal{E}_p(u) := \int^1_0 |\kappa_u|^p \sqrt{1+(u')^2} \dx, \quad p>1, 
\end{equation}
where $\kappa_u$ denotes the curvature of ${\rm graph}(u):=\{ (x, u(x)) \, \colon \, x \in [0,1]\}$. 
More precisely, we consider the variational problem 
\begin{equation} \label{obstacle_problem:1.1}
\inf_{v \in M(\psi)} \mathcal{E}_p(v) 
\end{equation}
with 
\begin{equation} \label{eq:admset}
M(\psi) := \{ v \in W^{2,p}(0,1) \cap W^{1,p}_0(0,1) \mid v \ge \psi \mbox{ in } [0, 1] \}. 
\end{equation}
Here, the function $\psi$ denotes the obstacle and satisfies the following assumption.
\begin{ass}[Assumption on the obstacle $\psi$] \label{ass:psi}
We require that $\psi \in C^0([0,1])$ with $\psi(0), \psi(1)< 0$ and  
that there exists some $x_0 \in (0,1)$ such that $\psi(x_0)> 0$. 
\end{ass}

Obstacle problems have been attracting attention from a lot of researchers. 
Moreover, the $p$-elastica functional has recently raised a lot of interest, see e.g.\ \cite{AGM,Watanabe,SW20,LP22,TatsuyaKensuke1,TatsuyaKensuke2,TatsuyaKensuke3,TatsuyaKensuke4} for the study of critical points and see e.g.\ \cite{NP20, OPW20, BHV, Poz22, OW23} for the study of gradient flows.
In this article we  study
the regularity of minimizers. In cases where the absence of obstacle implies smoothness of minimizers, i.e., any solution of the Euler--Lagrange equation for the corresponding energy functional is smooth,
one can generally observe
that the presence of obstacle induces a loss of regularity of minimizers. 
The collapse of the regularity occurs at the coincidence set between minimizers and obstacle. 
On the other hand, if the Euler--Lagrange equation for the corresponding energy has a degeneracy, it is significant to ask the following mathematical question: 
\begin{prob} \label{prob:1.1}
Which is the main cause of the loss of regularity of minimizers, the presence of obstacle or the degeneracy of the Euler--Lagrange equation?
\end{prob}

The critical points of the elastic energy 
\[
\mathcal{E}_2(\gamma) = \int_\gamma \kappa(s)^2 \ds
\]
with the length constraint are called elastica and have been studied well in the mathematical literature.
Here, $\gamma$~is a planar curve, and $\kappa$ and $s$ denote the curvature and the arc length parameter of~$\gamma$, respectively.  
Since the curvature of any critical point of $\mathcal{E}_2$ with no constraint satisfies the Euler--Lagrange equation 
\[
2 \kappa_{ss} + \kappa^3 = 0,  
\]
it is well known that any critical point of $\mathcal{E}_2$ is analytic, cf. \cite{LangerSinger}. 
On the other hand, the Euler--Lagrange equation for the $p$-elastic energy 
\[
\mathcal{E}_p(\gamma) = \int_\gamma |\kappa|^p \ds
\]
is formally written as 
\begin{equation} \label{EL-eq:1.2}
p (|\kappa|^{p-2} \kappa)_{ss} + (p-1) |\kappa|^p \kappa = 0.  
\end{equation}
The degeneracy of the Euler--Lagrange equation \eqref{EL-eq:1.2} induces a loss of regularity of critical points of $\mathcal{E}_p$. 
One of the reasons for the loss of regularity is the emergence of {\it flat core solutions}, which are critical points of $\mathcal{E}_p$ 
and contain flat parts, i.e., parts with~$\kappa \equiv 0$ (see \cite{Watanabe,TatsuyaKensuke1}). 
In \cite{Watanabe}, it was proved that the flat core solutions do not belong to~$C^3$. 
Recently, a complete classification of solutions of \eqref{EL-eq:1.2} has been obtained, as well as the optimal regularity (see \cite{TatsuyaKensuke1}). 

One of the purposes of this paper is to give an answer to Problem \ref{prob:1.1} for the obstacle problem \eqref{obstacle_problem:1.1}.
Following this motivation, we consider the more general obstacle problem
\begin{equation} \label{obstacle_problem:general}
\inf_{v \in M(\psi)} \mathcal{E}(v)
\end{equation}
with 
\begin{equation} \label{def-E-intro}
    \mathcal{E}(u)  := \int_0^1 |G(u'(x))'|^p \dx, 
\end{equation}
called  the \emph{generalized p-elastica functional}, where $\psi$ and $G$ satisfy Assumptions \ref{ass:psi} and \ref{ass:G}, respectively: 
\begin{ass}[Assumptions on $G$]\label{ass:G}
We assume that $G \in C^\infty( \mathbb{R})$. For the derivatives of $G$ we use the notation $\dot{G}, \Ddot{G},\ldots$. We require 
\begin{enumerate}
    \item $G$ is \emph{odd}, i.e. $G(-s) = - G(s)$,
    \item $\dot{G} > 0$, $(s \mapsto s \dot{G}(s)) \in L^1(\mathbb{R})$.
\end{enumerate}
\end{ass}
\noindent
The $p$-elastic energy \eqref{eq:p-elastic_energy} can be recovered when considering
\begin{equation}\label{eq:EU_p}
G(z) := EU_p(z) := \int_0^z \frac{1}{(1+s^2)^{\frac{3}{2}- \frac{1}{2p}}} \ds. 
\end{equation}
In particular the case $p= 2$ and $G = EU_2$ resembles the classical elastic energy of $\mathrm{graph}(u)$. Generalized functionals like in \eqref{def-E-intro} have applications in image restoration, cf. \cite{DalMaso}.

To study the optimal regularity of minimizers of problem \eqref{obstacle_problem:general}, we introduce a notion of {\em critical points} of problem \eqref{obstacle_problem:general}.
Inspired by the work \cite[p.\ 248]{Euler} of Euler  on the elastic energy, we consider first the following substitution. 
\begin{definition}\label{def:w}
Let $u \in W^{2,p}(0,1)$. Then the quantity $w \in L^\frac{p}{p-1}(0,1)$ given by 
\begin{equation} \label{gEs}
    w(x) := w_u(x) := -p\dot{G}(u'(x))^{p-1} |u''(x)|^{p-2} u''(x) \quad a.e. \; x \in (0,1)
\end{equation}
is called \emph{Euler's substitution}.
\end{definition}
Then critical points for problem \eqref{obstacle_problem:general} are defined as follows: 
\begin{definition} \label{def:critical}
We say that $u \in M(\psi)$ is a {\em critical point} of $\mathcal{E}$ with obstacle constraint if there exists a nonnegative Radon measure $\mu$ supported on $\{ u = \psi \}$ such that Euler's substitution $w=w_u$ satisfies 
\begin{equation}\label{eq:solELeq}
-\int^1_0 w(x) \dfrac{d}{dx}\Bigl( \dot{G}(u'(x)) \phi'(x) \Bigr) \dx = \int^1_0 \phi(x) \, {\rm d}\mu(x) 
\end{equation} 
for all $\phi \in W^{2,p}(0,1) \cap W^{1,p}_0(0,1)$. 
\end{definition}
Our first main result Theorem \ref{Theorem:1} is concerned with a qualitative property and the optimal regularity of critical points.
\begin{theorem} \label{Theorem:1}
Let $p \in (1,\infty)$ and $u \in M(\psi)$ be a critical point in the sense of Definition~{\rm \ref{def:critical}}. 
Then $u \in C^2([0,1])$ and satisfies 
\begin{equation} \label{q-property}
u''(x) < 0 \quad \text{for all} \quad x \in (0,1), \quad \text{and} \quad u''(0)=u''(1)=0.  
\end{equation}
In particular, $u$ is concave. Moreover, the following hold$\colon$
\begin{enumerate}
\item[{\rm (i)}] If $1 < p \le 2$, then 
\[
u \in W^{3,\infty}(0,1), \ u'''\in BV(0,1) \mbox{ but, in general, } u \not \in C^3(0,1). 
\]
\item[{\rm (ii)}] If $p > 2$, then 
\[
u \in W^{3,q}(0,1) \quad \text{for all} \quad q \in \bigl[1, \tfrac{p-1}{p-2} \bigr) \quad \text{and} \quad u \not\in W^{3, \tfrac{p-1}{p-2}}(0,1). 
\]
\end{enumerate}
\end{theorem}
One of the novelties of this paper is the generalized Euler's substitution~\eqref{gEs} for the Euler--Lagrange equation. 
For each critical point $u$, the substitution allows us to apply the maximum principle to the (second order) elliptic PDE satisfied by $w_u$ and, in turn, to prove several qualitative properties as well as the optimal regularity of $w_u$.  
We prove Theorem~\ref{Theorem:1} by translating the qualitative properties and the optimal regularity on $w_u$ into those of $u$.  

We prove that any minimizer of problem \eqref{obstacle_problem:general} is a critical point in the sense of Definition~\ref{def:critical}. 
Thus any minimizer of problem~\eqref{obstacle_problem:general} satisfies the property~\eqref{q-property} which implies that any minimizer of problem~\eqref{obstacle_problem:general} 
is concave and has no flat core part. 
Then we observe that the collapse of regularity of minimizers may occur only at the boundary and the coincidence set. 
Moreover, we infer from the optimal regularity~(i) and~(ii) in Theorem~\ref{Theorem:1} that for $p \in (2, \infty)$ the degeneracy of the Euler--Lagrange equation 
plays a more important role on the loss of regularity, whereas for $p \in (1, 2]$ it is the effect of obstacle constraint. 

\begin{remark}
On the existence of minimizers of problem~\eqref{obstacle_problem:general}, we prove the existence for small obstacles and non-existence for large obstacles 
(see Theorem~\ref{thm:exinotopti} and Proposition~\ref{Lemma:non-existence}, respectively). 
Moreover, with the aid of a nonlinear Talenti inequality, 
we prove existence of a symmetric minimizer 
for suitably small symmetric obstacles $\psi$ (see Proposition \ref{prop:simsmall} and Corollary \ref{Cor:symmetry_minimizer}). 
\end{remark}

The second purpose of this paper is to give the complete classification of solvability of problem \eqref{obstacle_problem:1.1} in terms of the `size' of obstacle $\psi$, where $\psi$ denotes a specified obstacle, the so-called cone obstacle. 
Recently, obstacle problems for the elastic energy have been attracting attention and have been studied well in the mathematical literature, e.g., 
problem \eqref{obstacle_problem:1.1} with $p=2$ (\cites{Anna,Moist,MariusExistence,Y2021,Novaga}), the obstacle problem for $\mathcal{E}_2$ with the clamped boundary condition (\cite{GrunauOkabe}), 
and gradient flows corresponding to problem \eqref{obstacle_problem:1.1} with $p=2$ (\cites{Marius_2020,MariusFlow,OY}).  
Following our second purpose, we focus on the case of symmetric cone obstacle, 
given by functions $\psi \in C^0([0,1])$ such that $\psi|_{[0, 1/2]}$ and $\psi|_{[1/2,1]}$ are affine linear, and $\psi(x)=\psi(1-x)$ for $x \in [0,1]$. 
More precisely, we consider the obstacle problem 
\begin{equation} \label{obstacle_problem:1.2}
\inf_{v \in M_{\rm sym}(\psi)} \mathcal{E}_p(v) 
\end{equation}
with 
\begin{equation}\label{eq:msym}
M_{\rm sym}(\psi) := \{ v \in M(\psi) \, \colon \, v(x)=v(1-x) \,\, \text{for} \,\, x \in [0,1] \}, 
\end{equation}
where $\psi$ is a symmetric cone satisfying Assumption \ref{ass:psi}. 
For problem \eqref{obstacle_problem:1.2} with $p=2$ it was proved by \cites{Anna,Moist,MariusExistence,Y2021} that there exists a threshold $h>0$ satisfying the following: 
if $\psi(\tfrac{1}{2}) < h$, then there exists a unique minimizer $u \in M_{\rm sym}(\psi)$ of~$\mathcal{E}_2$ in~$M_{\rm sym}(\psi)$; 
if $\psi(\tfrac{1}{2}) \ge h$, then there is no minimizer of~$\mathcal{E}_2$ in~$M(\psi)$. 
Theorem~\ref{Theorem:1.2} gives an extension of the result to the case $p \neq 2$: 
\begin{theorem} \label{Theorem:1.2}
Let $ p > 1$. There exists a constant $h_*>0$ satisfying the following$\colon$ 
Let $\psi$ be a symmetric cone and satisfy Assumption {\rm \ref{ass:psi}}. 
\begin{enumerate}
\item[{\rm (i)}] If $\psi(\frac{1}{2})<h_*$, then there is a unique minimizer $u \in M_{\rm sym}(\psi)$ of $\mathcal{E}_p$ in $M_{\rm sym}(\psi)$.
\item[{\rm (ii)}] If $\psi(\frac{1}{2}) \geq h_*$, then there is no minimizer of $\mathcal{E}_p$ in $M(\psi)$.  
\end{enumerate}
\end{theorem}
Here we note that the value $h_*$ can be explicitly given (see \eqref{eq:h_*} below).
To prove Theorem \ref{Theorem:1.2}, we employ a strategy, given by \cite{Moist}, making use of the polar tangential angle. 
To this aim, using a substitution $\omega:=|\kappa|^{p-2} \kappa$, first we derive an explicit formula of a solution of the Euler--Lagrange equation \eqref{EL-eq:1.2}.  
The explicit formula gives us a special 
curve, the so called $p$-rectangular. 
By way of the $p$-rectangular we extend the polar tangential strategy into the case $p \neq 2$. 
Both the substitution $\omega$ and the Euler's substitution $w_u$ reduce the second-order ODE in terms of $\kappa$ to 
a second-order ODE in terms of $u$. However, the reduced equations are different, and each reduced equation independently provides properties of critical points
(more precisely see Remark~\ref{rem:6.3toberefereedtointheintro} below).

This paper is organized as follows: 
In Section~\ref{section:EL-analysis}, using the Euler's substitution, we motivate the definition of critical points of $\mathcal{E}$ in $M(\psi)$. 
In Section~\ref{chap:criticalpt}, we study the optimal regularity and qualitative properties of critical points which are introduced in Section \ref{section:EL-analysis}. 
More precisely, first we prove the regularity of the Euler's substitution in Section~ \ref{section:reg-E-sub}. 
By way of the results in Section~\ref{section:reg-E-sub} we show the following for critical points of $\mathcal{E}$ in $M(\psi)$: Nondegeneracy (Section~\ref{section:nondege-cp}); properties of the coincidence set (Section~\ref{sec:coinset}); optimal regularity (Section~\ref{section:Optimal-reg-cp}). 
In Section~\ref{section:existence-sym-mini} we prove the existence of minimizers of $\mathcal{E}$ in $M(\psi)$ (Section~\ref{section:existence-mini}) and symmetry of minimizers of $\mathcal{E}$ in $M(\psi)$ under a smallness condition on symmetric obstacles (Section~\ref{sec:symmetry}). 
The nonexistence of minimizers for $\mathcal{E}$ in $M(\psi)$ for large cone obstacles $\psi$ is proved in Section~ \ref{sec:nonexistence}. 
Finally we prove Theorem~\ref{Theorem:1.2} in Section~\ref{section:4}. 

\smallskip

\noindent
{\bf Acknowledgements.}
The first author has been funded partially by the DFG (German Research Foundation)- Projectnumber: 404870139.
The third author was supported in part by JSPS KAKENHI Grant Number JP19H05599, 20KK0057, and 21H00990. 
The fourth author was supported by JSPS KAKENHI Grant Number JP19J20749. The authors would like to thank the anonymous referee for helpful suggestions. 
\\ {\bf Data availability statement.}  Data sharing not applicable to this article as no datasets were generated or analysed during the current study.
\section{Euler--Lagrange Analysis} \label{section:EL-analysis}
One of the main novelties in this article is the detailed study of the \emph{Euler--Lagrange equation}. It relies on an observation that was already aware to Euler (see~\cite{Euler}), who understood that for the classical elastic energy the substitution 
\begin{equation}
    v(x) := \frac{u''(x)}{(1+u'(x)^2)^\frac{5}{4}}
\end{equation}
transforms the Euler--Lagrange equation into a \emph{second-order elliptic equation} for~$v$. This \emph{order reduction} allows usage of the maximum principle and will have extensive consequences for the study of the obstacle problem, cf. \cite{MariusExistence ,GrunauOkabe}. We can  generalize Euler's substitution to our generalized $p$-elastica functional with a refined substitution.  In the following we always consider $p \in (1,\infty)$ as fixed if not stated otherwise. 

To this end, let $p \in (1,\infty)$ and notice that $\mathcal{E}: W^{2,p}(0,1) \rightarrow \mathbb{R}$ is Fréchet differentiable. Its Fréchet derivative reads (for each $\phi \in W^{2,p}(0,1)$ and using Assumption~\ref{ass:G})
\begin{align}
\begin{split} \label{eq:FD}
    D\mathcal{E}(u) (\phi) &= \frac{d}{dt}\bigg\vert_{t= 0}  \mathcal{E}(u + t \phi )  \\
    &= p  \int_0^1  \left( \dot{G}(u')^p |u''|^{p-2} u'' \phi'' +  \dot{G}(u')^{p-1} \Ddot{G}(u') \phi' |u''|^p \right) \dx.  
\end{split}
\end{align}
Now Euler's trick can be resembled as follows. In the second summand write $|u''|^p = ( |u''|^{p-2} u'' ) u''$. Then one can define $w \in L^{\frac{p}{p-1}}(0,1)$ as Euler's substitution (see Definition \ref{def:w}) 
and obtain by factoring out that 
\begin{equation}
    D\mathcal{E}(u)(\phi) = - \int_0^1 w \left( \dot{G}(u') \phi'' + \Ddot{G}(u') u'' \phi' \right) \dx.
\end{equation}
One can now observe by the product rule that the expression in parentheses is nothing but $\frac{d}{dx} (\dot{G}(u') \phi')$ and thus one obtains
\begin{equation}\label{eq:doubleu}
    D\mathcal{E}(u)(\phi) = - \int_0^1 w \frac{d}{dx} \left(\dot{G}(u') \phi'\right) \dx .
\end{equation}
In particular, the equation $D \mathcal{E}(u) = 0$ is a very weak form of an \emph{elliptic equation of second order} for $w$, namely of
\begin{equation}
   - \frac{d}{dx} \left( \dot{G}(u'(x)) \frac{d}{dx} w(x)  \right) = 0. 
\end{equation}
Notice that this equation is elliptic since $\dot{G} > 0$ by Assumption \ref{ass:G}. 
It becomes obvious that the quantity $w$ plays an important role in our analysis.

\subsection{A measure-valued Euler--Lagrange equation}

The basis of our measure-valued approach is given by the study of \emph{variational inequalities}, which is a theory of critical points for  functionals defined on  a convex subset of a Banach space, cf. \cite{Kinderlehrer} for details. For obstacle problems it is usual that a notion of critical points can be established by means of a variational inequality. Such inequality can be reformulated as a measure-valued Euler--Lagrange equation. The measure that appears is always supported on the \emph{coincidence set} $\{ u = \psi\}$. Since this measure-valued equation is (second-order) elliptic in terms of $w_u$, we can proceed our analysis using methods for measure-valued elliptic PDEs. We refer to \cite{Ponce} for a very detailed introduction to these types of PDEs.

We will next derive that the critical point equation in Defintion \ref{def:critical} is a necessary criterion for a minimizer.
\begin{prop}[Euler--Lagrange equation]\label{prop:ELeq}
Let $u \in M(\psi)$ be a minimizer. Then there exists a finite positive Radon measure $\mu$ supported on $\{ u= \psi \}$ such that 
\begin{equation}\label{eq:minELeq}
   -  \int  w(x) \frac{d}{dx} \left( \dot{G}(u'(x))  \phi'(x)  \right) = \int \phi \; \mathrm{d}\mu  \quad \forall \phi \in W^{2,p}(0,1) \cap W_0^{1,p}(0,1), 
\end{equation}
with $w =w_u$ being Euler's substitution, see \eqref{gEs}.  We call $\mu$ \emph{coincidence measure}. 
\end{prop}
\begin{proof}
Let $u \in M(\psi)$ be a minimizer. Since $M(\psi)$ (see \eqref{eq:admset}) is a convex subset of $W^{2,p}(0,1)\cap W_0^{1,p}(0,1)$, for each $v \in M(\psi)$ one has that $u + t (v-u)  \in M(\psi)$ for all $t \in [0,1]$. In particular, minimality of $u$ implies  
\begin{equation}\label{eq:varpos}
    0 \leq \lim_{t \rightarrow 0+} \frac{ \mathcal{E}(u+ t(v-u) ) - \mathcal{E}(u) }{t} = \frac{d}{dt} \bigg\vert_{t = 0} \mathcal{E}(u+t(v-u)) = D \mathcal{E}(u) (v-u). 
\end{equation}
Next let $\phi \in C_0^\infty(0,1)$ be such that $\phi \geq 0$. Considering $v = u + \phi \in M(\psi)$ in the previous equation we find with \eqref{eq:doubleu}
\begin{equation}
    0 \leq D\mathcal{E}(u)(\phi) = -  \int_0^1 w \frac{d}{dx} \left( \dot{G}(u') \phi' \right) \dx.  
\end{equation}
This and the Riesz--Markow--Kakutani Theorem (cf. \cite[Theorem 1.39]{EvGar}) yield that there exists a positive (but not necessarily finite) Radon-measure $\mu$ on $(0,1)$ such that 
\begin{equation}
  - \int_0^1 w \frac{d}{dx} \left(\dot{G}(u') \phi' \right) \dx =  \int \phi \; \mathrm{d}\mu \quad \forall \phi \in C_0^\infty(0,1).  
\end{equation}
We show next that $\mu$ is supported on $\{ u = \psi \}$.
Indeed, fix some arbitrary $\phi \in C_0^\infty( (0,1) \cap  \{ u > \psi\})$. One readily checks that there exists $\epsilon_0 > 0$ such that for all $\epsilon \in (-\epsilon_0,\epsilon_0)$ one has $u + \epsilon \phi \in M(\psi)$. Considering  \eqref{eq:varpos} with $v = u + \epsilon \phi$ and using~ \eqref{eq:doubleu} we obtain
\begin{equation}
    0 \leq - \epsilon \int_0^1 w \frac{d}{dx} \left( \dot{G}(u') \phi' \right) \dx = \epsilon \int \phi \; \mathrm{d}\mu.
\end{equation}
Looking at positive and negative values of $\epsilon$ we conclude that $\int \phi \; \mathrm{d}\mu = 0$ and since $\phi \in C_0^\infty( (0,1) \cap \{ u > \psi \})$ was arbitrary we infer that the support $\mathrm{spt}(\mu)$ satisfies
\[
\mathrm{spt}(\mu) \cap ((0,1) \cap \{ u > \psi \}) = \emptyset,
\]
i.e. $\mathrm{spt}(\mu) \subset \{u = \psi \}.$
Since $\{u = \psi\}$ is a compact subset of $(0,1)$ (cf. Remark~\ref{rem:coinci}), we infer that $\mu$ is also finite.
We have shown that 
\begin{equation}\label{eq:ELeq}
    - \int_0^1 w \frac{d}{dx} \left( \dot{G}( u') \phi' \right) \; \mathrm{d}x  = \int \phi \; \mathrm{d}\mu \quad \forall \phi \in C_0^\infty(0,1).
\end{equation}
It remains to show that this equation holds actually true for all $\phi \in W^{2,p}(0,1) \cap W_0^{1,p}(0,1)$. Notice first that \eqref{eq:ELeq} holds true for all $\phi \in W_0^{2,p}(0,1)$ by taking the $W^{2,p}$-closure of $C_0^\infty$. Now choose $\xi \in C_0^\infty(0,1)$ such that $\xi \equiv 1$ on some open neighborhood $O \subset \subset  (0,1)$ of $\{ u = \psi \}.$ Such $\xi$ exists since $\{ u= \psi \}$ is a compact subset of $(0,1)$. Let $\phi \in W^{2,p}(0,1) \cap W_0^{1,p}(0,1)$ be arbitrary. Then $\phi = (\phi \xi) +  \phi  (1- \xi) $ and $\phi \xi \in W_0^{2,p}(0,1)$. 
Therefore 
\begin{align}
    & - \int_0^1 w \frac{d}{dx} \left( \dot{G}( u') \phi' \right) \; \mathrm{d}x  \\ =& -\int_0^1 w \frac{d}{dx} \left( \dot{G}( u') (\phi \xi)' \right) \; \mathrm{d}x  - \int_0^1 w \frac{d}{dx} \left( \dot{G}( u') (\phi (1-\xi))' \right) \; \mathrm{d}x.
\end{align}
Now 
\begin{equation}\label{eq:w2p0}
    -\int_0^1 w \frac{d}{dx} \left( \dot{G}( u') (\phi \xi)' \right) \; \mathrm{d}x = \int (\phi \xi) \; \mathrm{d}\mu  = \int \phi \; \mathrm{d}\mu,
\end{equation}
since $\xi \equiv 1$ on $\{u = \psi\} \supset \mathrm{spt}(\mu)$. 
For the second summand we obtain by \eqref{eq:doubleu}
\begin{equation}
    - \int_0^1 w \frac{d}{dx} \left( \dot{G}( u') (\phi (1-\xi))' \right) \; \mathrm{d}x = D \mathcal{E}(u) ( \phi( 1- \xi)).
\end{equation}
Since $\phi ( 1- \xi)$ has support compactly contained in $\{ u > \psi \} \subset [0,1]$ we infer that $ u + \epsilon \phi ( 1- \xi) \in M(\psi)$ for all $\epsilon \in (-\epsilon_0,\epsilon_0)$ for some $\epsilon_0 > 0$. This implies by considering~\eqref{eq:varpos} with $v= u + \epsilon \phi ( 1- \xi) $ that  
\begin{equation}
     0 \leq D \mathcal{E}(u) ( \epsilon \phi ( 1- \xi) ) = \epsilon D \mathcal{E}(u) ( \phi( 1-\xi)). 
\end{equation}
Again looking at positive and negative values of $\epsilon$ we find that $D \mathcal{E}(u) ( \phi ( 1-\xi)) = 0$ and therefore 
\begin{equation}
     - \int_0^1 w \frac{d}{dx} \left( \dot{G}( u') (\phi (1-\xi))' \right) \; \mathrm{d}x = 0.  
\end{equation}
This and \eqref{eq:w2p0} imply \eqref{eq:minELeq}.
\end{proof}
\begin{remark}\label{rem:coinci}
By Assumption \ref{ass:psi}, 
for each $u \in M(\psi)$ the \emph {coincidence set} \begin{equation}
    \{ u = \psi \} := \{ x \in [0,1] : u(x) = \psi(x) \} 
\end{equation}
is a compact subset of $(0,1)$, since $\psi(0) < 0 = u(0), \psi(1) < 0 = u(1)$.
\end{remark}

In the sequel we will study (optimal) regularity properties of critical points. Doing so we will discover certain useful geometric properties that will help to characterize existence and nonexistence of minimizers.

\section{Optimal Regularity of critical points}\label{chap:criticalpt}

Now that a convenient critical point equation is established, elliptic regularity can be applied to examine its solutions $u \in M(\psi)$. A problem is that this elliptic regularity discussion first only applies to \emph{Euler's substitution} $w_u$ and not to $u$ itself. In order to translate regularity results for $w_u$ into regularity results for $u$, a \emph{degeneracy phenomenon} must be taken into account. Such phenomenon happens  --- if $p > 2$ ---  exactly on the set $\{ u'' = 0 \}$, as we shall see.  The elliptic maximum principle will enable us to overcome this problem. Indeed, one can show that $\{ u'' = 0 \} = \{ 0, 1 \}$. In particular, the mentioned degeneracy phenomenon can only occur at the boundary $\partial (0,1)$. 

\subsection{Regularity for Euler's substitution} \label{section:reg-E-sub}
In the sequel we denote by $\mathcal{L}^1$ the one-dimensional \emph{Lebesgue measure}. 
With this notation fixed we can use some results for measure-valued PDEs to obtain regularity of Euler's substitution  $w$. The following proposition shows $W^{1,\infty}$-regularity of $w$ and, as an immediate consequenence of \eqref{gEs}, also (non-optimal) $C^2$-regularity of $u$.
\begin{prop}[Regularity and boundary conditions for $w$]\label{prop:reguw}
Let $u \in M(\psi)$ be a critical point in the sense of Definition \ref{def:critical} and $w$ 
be Euler's substitution. 

Then $w \in W^{1,\infty}(0,1)$ and $w(0) = w(1) = 0$. In particular, $w$ solves the \emph{measure-valued elliptic equation} 
\begin{equation}\label{eq:elliptic}
    \int a(x) w'(x)  \phi'(x) \, dx= \int \phi \; \mathrm{d}\mu \quad \forall \phi \in W^{2,p}(0,1) \cap W_0^{1,p}(0,1),
\end{equation}
where $a(x):= \dot{G}(u'(x))$ for all $x \in (0,1)$.
Moreover, there exists a nonincreasing bounded function $m : (0,1) \rightarrow \mathbb{R}$, which is locally constant on $\{ u > \psi \}$ such that 
\begin{equation}\label{eq:aw'=m}
    a(x) w'(x) = m(x) \quad \textrm{for a.e.} \; x \in (0,1).
\end{equation}
Furthermore, $u \in C^2([0,1])$ and $u''(0) = u''(1) = 0.$ 
\end{prop}
\begin{proof} 
By \eqref{eq:solELeq} computing $\frac{d}{dx} ( \dot{G}(u') \phi' )$  with the chain rule we infer
 \begin{equation}\label{eq:deriv}
 - \int_0^1 w \dot{G}(u') \phi'' \dx = \int_0^1 \Ddot{G}(u') u''  w \phi' \dx + \int \phi \; \mathrm{d}\mu ,
 \end{equation}
for all $\phi \in W^{2,p}(0,1) \cap W_0^{1,p}(0,1)$. Since $u \in W^{2,p} \hookrightarrow C^1([0,1])$, $w \in L^{\frac{p}{p-1}}$ and $G \in C^{\infty}$, the function $f := \Ddot{G} (u') u'' w$ belongs to $L^1(0,1)$. 
With the introduced notation we infer 
 \begin{equation}
     -\int_0^1 w \dot{G}(u')  \phi'' \dx  = \int_0^1 f \phi' \dx + \int \phi \; \mathrm{d}\mu \, ,
 \end{equation}
 for all $\phi \in W^{2,p}(0,1) \cap W_0^{1,p}(0,1)$. 
 We can now use Fubini's theorem to rewrite the measure term in a convenient way. For  $\phi \in W^{2,p}(0,1) \cap W_0^{1,p}(0,1)$ one has
 \begin{equation}\label{eq:mufubini}
     \int_0^1 \phi \; \mathrm{d}\mu = \int_0^1 \int_0^x \phi' \; \mathrm{d}\mathcal{L}^1 \; \mathrm{d}\mu = \int_0^1 \mu((x,1)) \phi'(x) \; \mathrm{d}\mathcal{L}^1(x) = \int_0^1 g \phi' \dx,
 \end{equation}
where $g(x)  = \mu((x,1))$ a.e., i.e. $g \in L^\infty(0,1)$. Using this one finds 
\begin{equation}\label{eq:arbitest}
    - \int_0^1 w \dot{G}(u')  \phi'' \dx = \int_0^1 (f+g) \phi' \dx \,  \quad \textrm{ for all $\phi \in W^{2,p}(0,1) \cap W_0^{1,p}(0,1)$. }
\end{equation}
 Using test functions of the form 
 \begin{equation}
     \phi(x) = \int_0^x \xi(s) \; \mathrm{d}s - x \int_0^1 \xi(s) \; \mathrm{d}s,
 \end{equation}
 for arbitrary $\xi \in C_0^\infty((0,1))$ one infers that there exists $C \in \mathbb{R}$ such that for each $\xi \in C_0^\infty((0,1))$ one has 
 \begin{equation}
     - \int_0^1 (w \dot{G}(u')) \xi' \dx  = \int_0^1 (f+g+C) \xi \dx,
 \end{equation}
 whereupon one concludes that $w \dot{G}(u') \in W^{1,1}(0,1)$ and $(w \dot{G}(u'))' = f + g+ C$. By the fact that in one dimension $W^{1,1} \subset C^0$ we find that $w \dot{G}(u') \in C^0([0,1])$. In particular, since $\frac{1}{\dot{G}\circ u'}$ is continuous (as $\dot{G}> 0$) one has that $w$ is continuous. 
 From \eqref{gEs} and the fact that $z \mapsto |z|^{p-2}z$ has a continuous inverse function it follows that also $u''$ is continuous, i.e. $u \in C^2([0,1])$.
 Recalling the definition of $f$ we find that $f \in C^0([0,1])$. With this knowledge we infer from $(w \dot{G}(u'))' = f+ g+ C$ that $w \dot{G}(u') \in W^{1,\infty}(0,1)$. Using that $\frac{1}{\dot{G} \circ u'} \in W^{1,\infty}$ (as $\dot{G} > 0$ and $u'' \in L^\infty$) we infer from the Banach algebra property of $W^{1,\infty}$ that $w \in W^{1,\infty}$.
 This information given, we may again look at \eqref{eq:arbitest} for arbitrary $\phi \in W^{2,p}(0,1) \cap W_0^{1,p}(0,1)$ and integrate by parts to find \begin{equation}
    - \left[ w \dot{G}(u') \phi' \right]_0^1 + \int_0^1 (w\dot{G}(u'))' \phi' \dx  = \int_0^1 (f+g) \phi' \dx.
 \end{equation}
 Using again $(w\dot{G}(u'))' = f + g + C$ we find
 \begin{equation}
      - \left[ w \dot{G}(u') \phi' \right]_0^1 = \int_0^1 C \phi' \dx = 0 \quad \forall \phi \in W^{2,p}(0,1)\cap W_0^{1,p}(0,1). 
 \end{equation}
 Choosing suitable test functions $\phi \in W^{2,p}(0,1) \cap W_0^{1,p}(0,1)$ such that $\phi'(0) \neq 0, \phi'(1) = 0$ and respectively $\phi'(1) \neq 0,\phi'(0) = 0$ we infer that $w(0) \dot{G}(u'(0)) = w(1) \dot{G}(u'(1)) = 0$, which implies by Assumption \ref{ass:G} that $w(0) = w(1) = 0$. Thereupon \eqref{gEs} implies that $u''(0) = u''(1)= 0$. Now integration by parts in \eqref{eq:solELeq} and \eqref{eq:mufubini} imply (with $a(x):= \dot{G}(u'(x))$)
 \begin{equation}
     \int_0^1 (a(x) w'(x) - \mu((x,1))) \phi'(x) \dx = 0 \, \quad  \textrm{for all $\phi \in C_0^\infty(0,1).$}
 \end{equation}
   By the fundamental lemma of calculus of variations we infer that there exists a constant $C \in \mathbb{R}$ such that 
 \begin{equation}\label{eq:ding}
     a(x) w'(x) - \mu((x,1)) = C \quad \mbox{ a.e. in } (0,1). 
  \end{equation}
 Next we define $m(x) := \mu((x,1)) + C$ for $x \in (0,1)$. This is clearly bounded since $\mu$ is a finite measure and also nonincreasing. Moreover, since $\mu( \{ u > \psi \}) = 0$, it is also locally constant on $\{ u > \psi \}$. Rearranging \eqref{eq:ding} implies that $a w' = m$ almost everywhere. 
\end{proof}
\begin{remark}\label{rem:BV}
Equation \eqref{eq:aw'=m} implies that $w' = \frac{1}{a} m \in BV(0,1)$, since $m$ is $BV$ as monotone bounded function and  $\frac{1}{a} = \frac{1}{\dot{G} \circ u'} \in C^1([0,1])$.
\end{remark}
From these regularity properties of $w$ one can already infer some basic properties of critical points. 
\begin{cor}[Concavity]\label{cor:concavity}
Let $u \in M(\psi)$ be a critical point in the sense of Definition \ref{def:critical}. Then $u$ is concave.
\end{cor}
\begin{proof}
Since $\mathrm{sgn}(u'') = -\mathrm{sgn}(w)$ it suffices to show that 
$w \geq 0$  on $(0,1)$ with $w=w_u$ being Euler's substitution. Let $a,m$ be as in  the previous proposition. Then 
\begin{equation}
    a(x) w'(x) = m(x) \quad \mbox{ a.e. in }(0,1).
\end{equation}
Since $a > 0$ one has $\mathrm{sgn}(m)= \mathrm{sgn}(w')$ a.e.. Since $m$ is nonincreasing there must exist a (not necessarily unique) $c \in [0,1]$ such that $w' \geq 0$ a.e. on $(0,c]$ and $w' \leq 0$ a.e. on $(c,1)$. In particular $w$ is nondecreasing on $[0,c]$ and nonincreasing on $[c,1]$. This monotonicity behavior and $w(0) = w(1) = 0$ imply immediately that $w \geq 0$. 
\end{proof}
\begin{remark}
An important consequence of the concavity of $u$ is that  Euler's substitution $w$ 
can be rewritten as 
\begin{equation}\label{eq:defww}
    w(x) = p \dot{G}(u'(x))^{p-1} (-u''(x))^{p-1}. 
\end{equation}
\end{remark}

\begin{cor}[Regularity away from the obstacle]\label{cor:regunonc}
Let $u \in M(\psi)$ be a critical point in the sense of Definition \ref{def:critical} and let $a= \dot{G} \circ u'$ as in Proposition \ref{prop:reguw}. Then $w \in C^2(\{ u > \psi \})$ and there holds (in the classical sense)
\begin{equation}\label{eq:elli}
    \frac{d}{dx}  \left( a(x) \frac{d}{dx} w(x) \right) = 0 \quad \forall x \in \{ u > \psi \}.
\end{equation}
\end{cor}
\begin{proof}
By the previous corollary and the chain rule $a \in C^1([0,1])$ and $a > 0$. Moreover Proposition \ref{prop:reguw} yields that $w' = \frac{m}{a}$ where $m$ is locally constant on $\{ u > \psi \}$. In particular, the chain rule then implies that $w' \in C^1(\{ u > \psi \})$ and hence  
$w \in C^2(\{ u > \psi \})$.  Equation \eqref{eq:elli} follows directly from \eqref{eq:elliptic} and the fact that $\mathrm{spt}(\mu) \cap \{ u >  \psi \} = \emptyset$. 
\end{proof}

We have now already inferred an optimal(!) regularity result for $w$, which is $w \in W^{1,\infty}(0,1)$. (Optimality can for example be seen with \cite[Remark 3.7]{Moist} for $p=2$).  
The rest of this section is devoted to the optimal regularity of $u$. Notice that the regularity of $u$ and the regularity of $w$ are two different questions. Indeed, $u''$ and $w$ are related via \eqref{eq:defww}, which means that  a $p-1$'st root has to be taken in order to retrieve $u''$ from $w$. This $p-1$'st root imposes regularity issues on $\{ u''=0 \}$, at least when $p > 2$ (since then the $p-1$'st root is not locally Lipschitz in $\R$). This explains that in order to find the optimal regularity, we have to pay special attention to the set $\{ u'' = 0\}$ (a \emph{flat core}), where a \emph{degeneracy phenomenon} might happen. 
\subsection{Nondegeneracy in the interior} \label{section:nondege-cp}
An obstruction to the regularity of the problem is imposed by 
the phenomenon of \emph{flat cores}. This means that the solution may contain \emph{line segments}. 

While flat cores will  generally occur for critical points with fixed length, cf. \cite{Watanabe}, we  will argue in this section that they will not occur for our critical points \emph{without fixed length assumption}. Understanding (non)degeneracy of minimizers with fixed length will be subject to future research. 

The basis of our argument is again the 
critical point equation satisfied by  
Euler's substitution.

\begin{prop}[Nondegeneracy] \label{prop:deg}
Let $u \in M(\psi)$ be a critical point in the sense of Definition \ref{def:critical}. Then $u$ is nondegenerate, i.e. $\{ u'' = 0 \} = \{ 0 ,1 \}$.  In particular $u'' < 0$ on $(0,1)$. 
\end{prop}
\begin{proof}
Let $u \in M(\psi)$ be a critical point and $w$ be its Euler's substitution. It is sufficient to prove that $\{w=0\}=\{0,1\}$. Let $m : (0,1) \rightarrow \mathbb{R}$ be the function from Proposition \ref{prop:reguw}. Arguing as in the proof of Corollary \ref{cor:concavity} 
we see that there must exist 
$c \in [0,1]$ such that 
$w' \geq 0$ a.e. on $[0,c]$ and $w' \leq 0$ a.e. on $[c,1]$. 
In particular, $w$ can be decomposed into two monotone parts. Since $w(0) = w(1) = 0$ we conclude from this that $\{ w = 0 \} = [0,\alpha_1] \cup [\alpha_2,1]$ for some $\alpha_1 \geq 0$ and $\alpha_2 \leq 1$. This is due to the elementary fact that monotone functions may attain a certain value only at a point or an interval. It remains to show that $\alpha_1 = 0$ and $\alpha_2 = 1$. For a contradiction assume that $\alpha_1 > 0$. Then in particular $w \equiv 0$ on $(0, \alpha_1)$ and therefore one also has $w' \equiv 0$ on $(0,\alpha_1)$. Notice that then 
\begin{equation}
    m(x) = a(x) w'(x) =  0 \mbox{ in } (0,\alpha_1). 
\end{equation}
Since $m$ is nonincreasing on $(0,1)$ one infers from this that $m \leq 0$ on $(0,1)$. In particular one has 
\begin{equation}
    w'(x) = \frac{m(x)}{a(x)} \leq 0 \quad a.e. \; x \in (0,1), 
\end{equation}
since $a= \dot{G} \circ u'>0$. Hence $w$ is nonincreasing. This and $w(0) = w(1) = 0$ however imply that $w \equiv 0$ on $[0,1]$. As a consequence one has $u'' \equiv 0$ on $[0,1]$. A contradiction to Assumption \ref{ass:psi} since this and $u(0) = u(1) = 0$ imply $u \equiv 0$. We infer that $\alpha_1 = 0$. Analogously one can show that $\alpha_2 = 1$ (the only difference is that in the case of $\alpha_2 < 1$, $w$ turns out to be nondecreasing instead of nonincreasing). 
\end{proof}

This nondegeneracy will allow us to take the mentioned $p-1$'st root of $w$ in $(0,1)$ without losing any differentiablity.  This can be used to prove many regularity statements and also further structural results, e.g. results about the \emph{coincidence set}.

An immediate consequence is e.g. the following local regularity result away from the obstacle.

\begin{cor}[Interior regularity away from the obstacle]\label{cor:reguaway}
Let $u \in M(\psi)$ be a critical point in the sense of Definition \ref{def:critical}. Then $u \in C^\infty((0,1) \cap \{ u > \psi \})$.
\end{cor}
\begin{proof}
By Corollary \ref{cor:regunonc} we have (cf. \eqref{eq:defww}) that $w = p \dot{G}(u')^{p-1}(-u'')^{p-1} \in C^2(\{ u > \psi \})$. Using that by Proposition \ref{prop:deg} $-u'' > 0$ on $(0,1)$
and taking the $p-1$'st root we find 
\begin{equation}
    \dot{G}(u') u'' =- \big(  \tfrac{1}{p} w\big)^\frac{1}{p-1} \in C^2((0,1) \cap \{u > \psi\}).
\end{equation}
We conclude $\frac{d}{dx} (G \circ u') \in C^2((0,1) \cap \{ u > \psi \})$ and therefore $G \circ u' \in C^3((0,1) \cap \{ u > \psi \})$. Since $G \in C^\infty$ and $\dot{G}>0$ we infer that $G$ is a local $C^3$-diffeomorphism and thus $u' \in C^3((0,1) \cap \{ u > \psi \})$. From this follows immediately that $u \in C^4((0,1) \cap \{ u > \psi \})$. This (and the nondegeneracy) at hand allow us to bootstrap further. For each $x \in \{u > \psi\}$ one has (cf. Proposition \ref{prop:reguw}) that $(\dot{G}\circ u') w' = \mathrm{const}.$ in a neighborhood of $x$. We use this information to derive an equation for the fourth order derivatives of $u$ that can be written as
\begin{equation}
    u''''(x) = \frac{F(u'(x),u''(x),u'''(x))}{\dot{G}(u'(x))^\alpha (-u''(x))^{\beta}}, 
\end{equation}
where $\alpha,\beta \in \mathbb{R}$ and $F : \mathbb{R}^3 \rightarrow \mathbb{R}$ is a smooth function. This formula and the fact that $\dot{G} > 0$ and $-u''> 0$ (cf. Proposition \ref{prop:deg}) make a bootstrapping argument possible, whereupon we conclude that $u \in C^\infty((0,1) \cap \{ u > \psi \})$. 
\end{proof}

\subsection{The coincidence set}\label{sec:coinset}

In this section we discuss some properties of the \emph{coincidence set} $\{ u= \psi \}$. Understanding the coincidence set is important for the regularity discussion, since it is a set where regularity can potentially be lost. In light of the interior nondegeneracy result of the previous section, coincidence with $\psi$ is actually the only phenomenon that can obstruct interior regularity.
\begin{prop}[Nonempty coincidence set]\label{prop:nonempty} Let $u \in M(\psi)$ be a critical point in the sense of Definition \ref{def:critical}. Then
$\{ u = \psi \} \neq \emptyset$. 
\end{prop}
\begin{proof}
Assume that $\{u = \psi\} = \emptyset$, i.e. $\{ u > \psi \} = [0,1]$. We infer by \eqref{eq:elli} and Proposition \ref{prop:reguw} that $w \in C^2((0,1) ) \cap C^0([0,1])$ is a classical solution of 
\begin{equation}\label{eq:boundaryproblemforw}
    \begin{cases}
        \frac{d}{dx} ( a(x) w'(x) ) = 0 & \textrm{ in } (0,1) \\
        w =  0 & \textrm{ on } \partial(0,1) = \{ 0, 1 \},
    \end{cases}
\end{equation}
where $a= \dot{G} \circ u'$. 
The elliptic maximum principle yields that $w \vert_{[0,1]} \equiv 0$. This however implies by \eqref{eq:defww} that $u'' = 0$ on $[0,1]$, i.e. $u$ is a line. Since $u(0)= u(1)= 0$ it follows immediately that $u \equiv 0$. This contradicts Assumption \ref{ass:psi}.
\end{proof}

\begin{remark}
The previous proof points out that the unique critical point \emph{without obstacle} (i.e. with $\mu = 0$) is given by $u \equiv 0$. This observation reveals a uniqueness result for (weak) solutions for the \emph{Navier problem}, i.e. the boundary value problem
\begin{equation}
    \begin{cases}
     \nabla_{L^2} \mathcal{E}(u) = 0 & \textrm{on}  \; (0,1) \\ u(0) = u(1) = 0 & \\
     u''(0) = u''(1) = 0 
    \end{cases}
\end{equation}
is uniquely solved by $u \equiv 0$. Applying this insight to the case of $\mathcal{E}(u) := \mathcal{E}_p(\mathrm{graph}(u))$, i.e. $G = EU_p$, yields 
that the straight line $u \equiv 0$ is the only graphical $p$-elastica that satisfies the Navier boundary conditions $u(0) = u(1) = 0$ and $\kappa_u(0) = \kappa_u(1) = 0$.
The landscape of critical points becomes much richer once one abandons the graph 
condition. Indeed, the equivalent Navier problem for curves
\begin{equation}
\begin{cases}
    \nabla_{L^2(ds)} \mathcal{E}_p(\gamma) = 0  & \textrm{on $(0,L_\gamma)$} \\ \gamma(0) = (0,0)^T , \gamma(L_\gamma) = (0,1)^T  \\ \kappa(0) = \kappa(L_\gamma) = 0 
    \end{cases}
\end{equation}
 has infinitely many solutions, cf. \cite[Theorem 1,(ii),(iv)]{Watanabe}, already for $p=2$ (cf. \cite{KensukeNavier}). This shows that the special geometry of graph curves is advantageous for the Navier problem.
\end{remark}
Another consequence of the nondegeneracy and the $C^2$-regularity is that obstacles can only be touched at \emph{concave points}, in the following sense. 

\begin{prop}[Noncoincidence at convex points]
Let $u \in M(\psi)$ be a critical point in the sense of Definition \ref{def:critical}. Let $x_0 \in (0,1)$ be such that $\psi$ is twice continuously differentiable in a neighborhood of $x_0$ and $\psi''(x_0) \geq 0$. Then $x_0 \not \in \{u = \psi\}$. 
\end{prop}
\begin{proof}
Assume that there exists some $x_0 \in (0,1)$ as in the statement and $u(x_0) = \psi(x_0)$. Since $(u-\psi) \geq 0$ one infers that $u- \psi$ attains a (local) minimum at $x_0$. Thereupon and since $(u-\psi) \in C^2$ in a neighborhood of $x_0$, it follows that $(u-\psi)'(x_0) = 0$ and $(u-\psi)''(x_0) \geq 0$. Hence we obtain 
\begin{equation}
    u''(x_0) \geq \psi''(x_0) \geq 0. 
\end{equation}
However, Proposition \ref{prop:deg} implies that $u'' < 0$ on $(0,1)$, which is a contradiction. 
\end{proof}
\begin{remark} \label{rem:symcone}
Obstacles of particular interest are \emph{cone obstacles}, given by functions $\psi \in C^0([0,1])$ such that there exists a $\theta \in (0,1)$ s.t. $\psi\vert_{[0,\theta]}$ and $\psi\vert_{[\theta,1]}$  are \emph{affine linear}. For such obstacle $\psi$ let $u \in M(\psi)$ be a critical point in the sense of Definition~\ref{def:critical}.
By the previous result $\{u = \psi \} \cap (0,1) \setminus \{ \theta \} = \emptyset$, as $\psi''(x_0) \geq 0$ for all $x_0 \in (0,1) \setminus \{\theta \}$. However since $\{ u = \psi \} \neq \emptyset$ (see Proposition \ref{prop:nonempty}) we infer that $\{ u = \psi \} = \{ \theta \}$. 
\end{remark}

\subsection{Optimal regularity for critical points} \label{section:Optimal-reg-cp}

In this section we study the optimal (global) regularity of critical points. Recall that we have two major obstructions to the regularity. The first one is the \emph{obstacle constraint} --- hitting the obstacle will affect the regularity. The second one is the \emph{degeneracy} of the equation when $u'' = 0$. While the latter phenomenon can (by the previous section) only occur at the boundary $\partial (0,1)$, the first phenomenon can only occur in the interior (as $\psi(0),\psi(1) <0$). In this section we investigate which of the two phenomena imposes a harsher restriction on the regularity. The answer to this question will depend on $p$. Indeed, for $p \in (2,\infty)$ the degeneracy plays a more important role, whereas for $p \in (1,2]$
it is the obstacle constraint. 
We remark that in both cases we will experience a substantial loss of regularity compared to the $C^\infty$-regularity  that holds where none of the two obstructions are present, cf. Corollary \ref{cor:reguaway}, that shows regularity away from the obstacle and away from the boundary (where the degeneracy occurs).

\begin{prop}[Case 1: $p\leq 2$]\label{prop:regu1}
Let $u \in M(\psi)$ be a critical point in the sense of Definition \ref{def:critical} and $p \leq 2$. Then $u \in W^{3,\infty}(0,1)$ and $u''' \in BV(0,1)$. This regularity is optimal in the sense that there exist obstacles $\psi_*$ satisfying Assumption \ref{ass:psi}  such that critical points (and even minimizers) in $M(\psi_*)$ can not lie in $C^3(0,1)$.
\end{prop}
\begin{proof}
\textbf{Step 1.} We show $W^{3,\infty}$-regularity.
First notice that by Proposition \ref{prop:reguw} one has that Euler's substitution $w = p \dot{G}(u')^{p-1} (-u'')^{p-1}$ (cf. \eqref{eq:defww}) lies in $W^{1,\infty}(0,1)$. Since $u' \in C^1([0,1])$ one has that $\frac{1}{\dot{G}(u')^{p-1}}$ also lies in $W^{1,\infty}(0,1)$. We infer from the product rule that $v(x) :=  (-u''(x))^{p-1}$ lies in $W^{1,\infty}(0,1)$. Since $F(z) := |z|^{\frac{1}{p-1}}$ is locally Lipschitz in $\mathbb{R}$ (as $p\leq 2$) and $u''$ is bounded we infer that $-u'' = F \circ v$ lies in  $W^{1,\infty}(0,1)$. We infer $u'' \in W^{1,\infty}(0,1)$ and therefore $u \in W^{3,\infty}(0,1)$.\\ 
\textbf{Step 2.} We show the $BV$-regularity of $u'''$.
Proposition~\ref{prop:reguw} and the chain rule imply for $a.e.$ $x \in (0,1)$
\begin{align}
     \frac{m(x)}{\dot{G}(u'(x))}  & = w'(x)= -c_p w(x)^\frac{p-2}{p-1} \left( \Ddot{G}(u'(x)) u''(x)^2 +  \dot{G}(u'(x)) u'''(x) \right),  
\end{align}
with $c_p$ a constant depending only on $p$. 
Rearranging and using that $w > 0$ we find for a.e. $x \in(0,1)$
\begin{equation}\label{eq:u'''}
    u'''(x) =  \frac{1}{\dot{G}(u'(x))} \Bigl( - \frac{m(x)}{c_p \dot{G}(u'(x))} \frac{1}{w(x)^{\frac{p-2}{p-1}}} - \Ddot{G}(u'(x)) u''(x)^2 \Bigr). 
\end{equation}
We conclude from the previous formula and the chain rule in $BV$ that $u''' \in BV(0,1)$ if and only if $m \cdot w^{\frac{2-p}{p-1}} \in BV(0,1)$. To prove the latter, we show that there exists $M> 0$  such that \begin{equation}
    \left\vert \int_0^1 m(x) w(x)^\frac{2-p}{p-1} \phi'(x) \; \mathrm{d}x  \right\vert \leq M ||\phi||_{\infty} \quad \forall \phi \in C_0^\infty(0,1).
\end{equation}
For $p=2$ this is trivial by the $BV$-property of $m$ (as decreasing bounded function). For $p \in (1,2)$ fix $\phi \in C_0^\infty(0,1)$. Since $w > 0$ on $(0,1)$ and $w \in W^{1,\infty}$ we find
\begin{align}
    & \left\vert  \int_0^1 m(x) w(x)^\frac{2-p}{p-1} \phi'(x) \; \mathrm{d}x  \right\vert  \\ & = \left\vert \int_0^1 m(x) \frac{d}{dx}\left( w(x)^\frac{2-p}{p-1} \phi(x) \right) \; \mathrm{d}x  - \int_0^1 \tfrac{2-p}{p-1}m(x)w(x)^\frac{3-2p}{p-1} w'(x) \phi(x) \; \mathrm{d}x  \right\vert.
\end{align}
Since $m(x) =  \mu ((x,1))+C$ (see \eqref{eq:ding}) with a finite measure $\mu$, proceeding as in \eqref{eq:mufubini} in the first integral, we obtain
\begin{equation}
\left\vert \int_0^1 m(x) w(x)^\frac{2-p}{p-1} \phi'(x) \; \mathrm{d}x  \right\vert \leq  \left\vert \int w^{\frac{2-p}{p-1}}  \phi \; \mathrm{d}\mu \right\vert + \left\vert \int_0^1  \frac{2-p}{p-1} \frac{m(x)^2}{\dot{G}(u'(x))}w(x)^\frac{3-2p}{p-1}  \phi(x) \; \mathrm{d}x \right\vert,
\end{equation}
where we used again $w' = \frac{m}{\dot{G}(u')}$ 
from Proposition \ref{prop:reguw}. By boundedness of $m$, $\mu(0,1)< \infty$ and $w \in L^\infty(0,1)$ one infers \begin{equation}\label{eq:BVregu}
    \left\vert \int_0^1 m(x) w(x)^\frac{2-p}{p-1} \phi'(x) \; \mathrm{d}x \right\vert \leq \tilde{M} \left( 1 + \int_0^1 |w(x)|^{\frac{3-2p}{p-1}} \; \mathrm{d}x \right) ||\phi||_\infty,
\end{equation}
for some finite $\tilde{M}> 0$.
It remains to show that the integral on the right hand side is finite. If $p\leq \frac{3}{2}$ this is trivial.
For other values of $p$ we  
let $\delta> 0$ be chosen such that $\{ u = \psi \} \cap [0,\delta) \cup (1-\delta,1]  = \emptyset$.
Since by Proposition \ref{prop:deg} $w> 0$ on $(0,1)$ one has that $|w|^\frac{3-2p}{p-1} \in L^\infty((\delta,1-\delta))$. Therefore it only remains to show integrability on $(0,\delta)$ and on $(1-\delta,1)$. Next we show the integrability on $(0,\delta)$.
 By the local constancy of $m$ on $\{ u > \psi \}$ (cf. Proposition \ref{prop:reguw})  there exists a constant $B_0 \in \mathbb{R}$ such that $m \equiv B_0$ on $(0,\delta)$. We claim that $B_0 \neq 0$. Indeed,  $B_0= 0$ would imply $w' = 0$ on $(0,\delta)$, which in turn would yield $w(x) = w(0) = 0$ for all  $x \in (0,\delta)$, contradicting the nondegeneracy statement in Proposition \ref{prop:deg}. We infer that $|\dot{G}\circ u'| \; |w'| =|B_0| > 0 $ on $(0,\delta)$. Using that $|\dot{G} \circ u'|$ is uniformly bounded above and below one finds $c,C > 0$ such that $c \leq |w'|\leq C$ on $(0,\delta)$. This, $w(0) = 0$ and $w\vert_{(0,1)}> 0$ implies that
\begin{equation}\label{eq:wcontrolDelta}
    cx \leq w(x) \leq Cx \quad \forall x \in (0,\delta).
\end{equation}
Using that $x \mapsto x^\frac{3-2p}{p-1}$ is integrable on $(0,\delta)$ (as $\frac{3-2p}{p-1} = - 2 + \frac{1}{p-1} > -1$) we obtain that $|w|^\frac{3-2p}{p-1}$ is also integrable on $(0,\delta)$.
Integrability on $(1-\delta,1)$ follows the same lines.  We conclude that the integral on the right hand side of \eqref{eq:BVregu} is finite. The claimed $BV$-regularity follows.
\\
\textbf{Step 3.} We construct an obstacle $\psi_*$ as in the statement.  
We actually show that for cone obstacles (cf. Remark \ref{rem:symcone}) 
 each critical point (and therefore each minimizer) does not lie in $C^3(0,1)$.  
To this end let $\psi_*$ be an arbitrary cone obstacle satisfying Assumption \ref{ass:psi}, say  $\psi_*\vert_{[0,\theta]}$ and $\psi_*\vert_{[\theta,1]}$  are affine linear for some $\theta \in (0,1)$. 
By Remark \ref{rem:symcone},
each critical point $u_* \in M(\psi_*)$ must have the tip of the cone as point of coincidence, i.e.
$\{ u_* = \psi_* \} = \{ \theta \}$.  
We fix now any critical point $u_*$ and for a contradiction assume now that $u_* \in C^3(0,1)$.
Let $w_* := w_{u_*}$ be Euler's substitution, see 
Definition \ref{def:w}. 
By the nondegeneracy in $(0,1)$, cf. Proposition ~\ref{prop:deg}, and Proposition ~\ref{prop:reguw}, it follows that $w_* \in C^1(0,1)$. Moreover, by Proposition~ \ref{prop:reguw} there exists $m_* :(0,1) \rightarrow \mathbb{R}$ nonincreasing and locally constant on $\{ u_* > \psi_* \}$ and $a_* \in C^1(0,1), a_*> 0$ such that $a_*(x) w_*'(x) = m_*(x)$ for a.e. $x \in (0,1)$. Since $\{ u_* > \psi_* \}= (0,1) \setminus \{ \theta \}$, the only point of nonconstancy of $m_*$ can be at $\theta$ and thus one has 
\begin{equation}
 a_*(x) w_*'(x) =  m_*(x) = \begin{cases}
   c_1 & x \in (0, \theta) \\ c_2 & x \in (\theta,1) 
  \end{cases} ,
\end{equation}
for some real constants $c_1 \geq c_2$. However, since $a_*,w_*' \in C^0(0,1)$ one infers that $c_1 = c_2$ and therefore there exists some $c= c_1 =c_2 \in \mathbb{R}$ such that \begin{equation}
    a_*(x) w_*'(x) = c \quad \forall x \in (0,1). 
\end{equation}
In particular one has $\frac{d}{dx}(a_*(x) w_*'(x)) = 0$ for all $x \in (0,1)$. Once this is established one can follow the lines of the proof of Proposition \ref{prop:nonempty} (after \eqref{eq:boundaryproblemforw}) to infer the contradiction $u \equiv 0$. 
\end{proof}

\begin{remark}
The previous proof reveals that all critical points whose coincidence set is a singleton may not lie in $C^3(0,1)$. Actually, the same argument would reveal that whenever the coincidence set of a critical point is \emph{discrete} then it cannot lie in $C^3(0,1)$. 
If, on the contrary, the coincidence set has a \emph{nondiscrete part}, more regularity might be possible. 
Indeed, by \eqref{eq:elliptic} it becomes obvious that the regularity depends strongly on the \emph{coincidence measure} $\mu$. If, say, $\mu$ is absolutely continuous with respect to the Lebesgue measure and has a smooth density, then standard elliptic regularity theory implies higher regularity. Notice however that  a  necessary condition for such absolute continuity is $\mathcal{L}^1(\mathrm{spt}(\mu)) > 0$, which implies $\mathcal{L}^1(\{ u= \psi \}) \geq \mathcal{L}^1(\mathrm{spt}(\mu)) > 0$, i.e. coincidence on a very large set. 
\end{remark}
\begin{remark}
Finally, we clarify that in the case of $p \leq 2$, the major reason for the loss of regularity is the obstacle. Indeed, for these values of $p$ the degeneracy phenomenon at the boundary does not affect the regularity since the $p-1$'st root is locally Lipschitz on $[0,\infty)$ when $p \leq 2$. This is why the Sobolev regularity of $u''$ and $w$ must coincide, cf. Proposition \ref{prop:regu1}.  That the major regularity problem is imposed by the obstacle becomes now obvious when recalling from Proposition \ref{prop:reguw} that
$  a(x) w'(x) = m(x)  
$ for a.e. $x \in (0,1).$
 Since $m$ is locally constant on $\{ u > \psi \}$ (and therefore smooth on $\{ u > \psi \}$) and $a \in C^1([0,1])$, $w'$ (and hence also $u'''$) may only lose its $C^0$-regularity on $\{ u = \psi \}$, i.e. because of coincidence with the obstacle. We have seen in Proposition \ref{prop:regu1} that this will actually happen. 
\end{remark}

We have now seen that touching the obstacle can potentially shrink regularity to $W^{3,\infty}$. 
In the case of $p > 2$ we will actually lose even more regularity due to the 
\emph{degeneracy} at the boundary. 

\begin{prop}[Case 2: $p > 2$]\label{prop:regu2}
Let $u \in M(\psi)$ be a critical point in the sense of Definition \ref{def:critical} and $p > 2$. Then $u \in W^{3,q}(0,1)$ for all $q \in [1, \frac{p-1}{p-2})$ but
$u \not \in W^{3, \frac{p-1}{p-2}}(0,1)$ . 
\end{prop}
\begin{proof}
Proposition~\ref{prop:reguw} yields 
that Euler's substitution $w = p \dot{G}(u')^{p-1} (- u'')^{p-1}$ lies in $W^{1,\infty}(0,1)$. By the nondegeneracy in the interior (cf. Proposition~\ref{prop:deg}) one has that $-u''>0$ on $(0,1)$, or alternatively $w > 0$ on $(0,1)$. Therefore, since $\frac{1}{\dot{G}(u')^{p-1}}$ is $C^1$, we infer that $ v := (-u'')^{p-1} \in W^{1,\infty}$. Moreover, since $-u'' > 0$ on $(0,1)$ (cf. Proposition~\ref{prop:deg}) and $F(z) := z^{\frac{1}{p-1}}$ is locally Lipschitz on $(0,\infty)$ ($0$ excluded!) we infer that $-u'' = F \circ v \in W^{1,\infty}_{loc}(0,1)$, i.e. $u''$ is locally Lipschitz and in particular weakly differentiable. 

\textbf{Step 1.} We show that $u \not \in W^{3, \frac{p-1}{p-2}}(0,1)$. 
Reasoning as in the proof of Proposition~ \ref{prop:regu1} (below \eqref{eq:BVregu})
there exists a constant $B_0 \in \R$, $B_0 \ne 0$ and $\delta>0$ s.t. 
$m = B_0$ on $(0,\delta)$.
In particular, from \eqref{eq:u'''} for a.e. $x \in (0, \delta)$ we have
\begin{equation}
    u'''(x) =  \frac{1}{\dot{G}(u'(x))} \Bigl( - \frac{B_0}{c_p \dot{G}(u'(x))} w(x)^{\frac{2-p}{p-1}} - \Ddot{G}(u'(x)) u''(x)^2 \Bigr) 
\end{equation}
with $B_0 \ne 0$.
We infer that on $(0,\delta) $, $|u'''|$ has the same integrability as $w^\frac{2-p}{p-1}$. Indeed, the prefactor $\frac{1}{\dot{G}\circ u'}$ can safely be disregarded since it is bounded from above and below. The summand $(\Ddot{G} \circ u' ) u''^2$ lies in $C([0,1]) \subset L^\infty(0,1)$, cf. Proposition~\ref{prop:reguw}, and hence can also be disregarded for the integrability. 
To finish the proof we investigate the integrability of $w^\frac{2-p}{p-1}$ on $(0,\delta)$.
Similar to the discussion in \eqref{eq:wcontrolDelta} in Proposition \ref{prop:regu1} we can derive that there exists $c,C> 0$ such that $cx \leq w(x) \leq Cx$ for all $x \in (0,\delta)$.
This in particular yields that $w^\frac{2-p}{p-1}$ has the same integrability as $x^\frac{2-p}{p-1}$. This on contrary implies by the discussion above that $|u'''|$ has the same integrability as $x^{\frac{2-p}{p-1}}$ on $(0,\delta)$. In particular, $|u'''| \not \in L^\frac{p-1}{p-2}(0,\delta)$, which implies that $u \not \in W^{3, \frac{p-1}{p-2}}(0,1).$ \\
\textbf{Step 2.}  We show that $u''' \in L^q(0,1)$ for all $q \in [1,\frac{p-1}{p-2})$. 
Note that \eqref{eq:u'''} and $\dot{G} \circ u' \geq a_0$ on $(0,1)$
imply that for a.e. $x \in (0,1)$ one has 
\begin{equation}
    |u'''(x)| \leq  \frac{||m||_\infty}{c_pa_0^2} w(x)^\frac{2-p}{p-1}  + \frac{1}{a_0}||\Ddot{G} \circ u'||_{L^\infty} ||u''||_{L^\infty}^2. 
\end{equation}
This and $w> 0$ implies again that for $\delta > 0$ as in Step 1 one has $w \in L^\infty(\delta,1-\delta)$. Hence it only remains to show $L^q$-integrability on $(0,\delta) \cup (1-\delta,1)$. 
Concerning $(0,\delta)$, let $c, C > 0$ be chosen as in Step 1.
 Then one has for $x \in (0,\delta)$
\begin{equation}
    |u'''(x)| \leq \frac{||m||_\infty}{c_p a_0^2} (cx)^{\frac{2-p}{p-1}} + \frac{1}{a_0}||\Ddot{G} \circ u'||_{L^\infty} ||u''||_{L^\infty}^2 \, .
\end{equation}
The last summand is finite by Proposition \ref{prop:reguw} and the first summand lies clearly in $L^q(0,\delta)$ for all $q \in [1, \frac{p-1}{p-2})$. Along the same lines, one can show the same integrability on $(1-\delta,1)$.
From the above discussion we conclude that 
\begin{equation}
    u''' \in L^\infty(\delta, 1- \delta) \cap L^q((0,\delta) \cup ( 1- \delta,1)) \subset L^q(0,1). \qedhere
\end{equation}
\end{proof}
Finally we combine the results in this section to obtain Theorem~\ref{Theorem:1}.
\begin{proof}[Proof of Theorem \ref{Theorem:1}]
The theorem follows combining the claims of Proposition~\ref{prop:reguw} and of Propositions~\ref{prop:deg}, \ref{prop:regu1} and \ref{prop:regu2}.
\end{proof}

\begin{remark}
We remark that in the case of $p>2$ the $BV$-property of $u'''$ can not be true. This is due to the fact that in one dimension $BV(0,1) \subset L^\infty(0,1)$.  However a close examination of the proof of Proposition \ref{prop:regu1} (Step 2, in particular \eqref{eq:BVregu}) still yields that $u''' \in BV_{loc}(0,1)$. This is due to the fact that $|w|^\frac{3-2p}{p-1}$ is integrable on $(\delta,1-\delta)$ by nondegeneracy. 
\end{remark}

\section{Existence and symmetry of minimizers} \label{section:existence-sym-mini}
\subsection{Existence of minimizers} \label{section:existence-mini} 
 In this section we will show an existence result for small obstacles for any $p \in (1,\infty)$. The smallness condition on the obstacle is expressed in terms of an upper bound on $\inf_{u \in M(\psi)} \mathcal{E}(u)$. In Section \ref{sec:nonexistence} we prove a non-existence result 
 for large obstacles, i.e.  smallness requirements on $\psi$ are really necessary. The bounds that we formulate in this section are likely not optimal.
 However in Section \ref{section:4} we will obtain a sharp existence result when minimizing among \emph{symmetric functions} only. 
  Notice that it is delicate to show symmetry of minimizers for a higher order problem due to the lack of a maximum principle. 
  Nevertheless we will present a symmetry result for small obstacles in Section \ref{sec:symmetry}. 
 
 In order to quantify a sufficient upper bound on the infimum we will define the constant 
\begin{equation}\label{eq:defcp}
    c_p(G) := 2 \lim_{s \rightarrow \infty} G(s), 
\end{equation}
which always exists in $[0,\infty]$ since $G$ is monotone and bounded by Assumption \ref{ass:G}.

 We also define some carefully chosen comparison functions, inspired by \cite{Anna}. For $c \in (0,c_p(G))$ we set
 \begin{equation}\label{eq:uc}
     u_c :[0,1] \to [0,\infty), \quad u_c(x) := \frac{1}{c} \int_{G^{-1} ( \frac{c}{2}- cx) }^{G^{-1}( \frac{c}{2} ) } s \dot{G}(s) \ds . 
 \end{equation}
 Clearly, one has $u_c \in C^\infty([0,1])$ and 
 \begin{equation}\label{eq:ucprime}
     u_c'(x) = G^{-1} \left( \frac{c}{2} - cx \right). 
 \end{equation}
 This implies that 
 \begin{equation}\label{eq:energyuc}
     \mathcal{E}(u_c) = \int_0^1 |G(u_c')'|^p \, \dx = c^p.
  \end{equation}
  We also define 
  \begin{equation} \label{def:U_0^G}
      U_0^G(x) := \lim_{c \rightarrow c_p(G)} u_c(x). 
  \end{equation}
  \begin{remark}\label{rem:increasingness}
      The map $c \mapsto u_c(x)$ is actually increasing for each $x \in [0,1]$, which is why the limit in \eqref{def:U_0^G} is a pointwise supremum. Indeed, one can calculate with \eqref{eq:uc} and \eqref{eq:ucprime} for fixed $x \in [0,\frac{1}{2}]$
      \begin{align}
          \partial_c u_c(x)  & = - \tfrac{1}{c} u_c(x) + \tfrac{1}{c} \left( \tfrac{1}{2}G^{-1}(\tfrac{c}{2}) - (\tfrac{1}{2}-x)G^{-1}(\tfrac{c}{2}-cx) \right)  \\ & = \tfrac{1}{c} (x u_c'(0)- u_c(x)) + \tfrac{1}{c}(\tfrac{1}{2}-x) (G^{-1}(\tfrac{c}{2})- G^{-1}(\tfrac{c}{2}-cx) ). 
      \end{align}
      Using the concavity of $u_c$ (cf. \eqref{eq:ucprime}) and $x \in [0,\frac{1}{2}]$ one infers that $\partial_c u_c(x)$ is nonnegative as a sum of nonnegative terms. For $x \in [\frac{1}{2},1]$ one obtains the same result by symmetry.
  \end{remark}
  It is easy to see (using that $(s \mapsto s \cdot \dot{G}(s)) \in L^1$ by Assumption \ref{ass:G}) that $U_0^G \in C^\infty(0,1) \cap C([0,1])$, $U_0^G(0)= U_0^G(1)= 0$ and 
  \begin{equation}
      (U_0^G)'(x) = G^{-1} \left( \frac{c_p(G)}{2}- c_p(G) x \right) \quad \forall x \in (0,1). 
  \end{equation}
  
  \begin{figure}[h!]
  \begin{minipage}{0.45\textwidth}
      \includegraphics[scale=0.6]{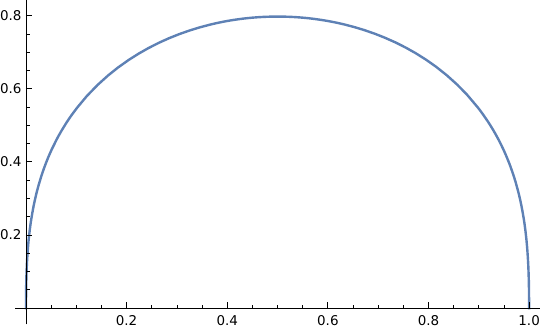}
  \end{minipage}
  \begin{minipage}{0.45\textwidth}
      \includegraphics[scale=0.6]{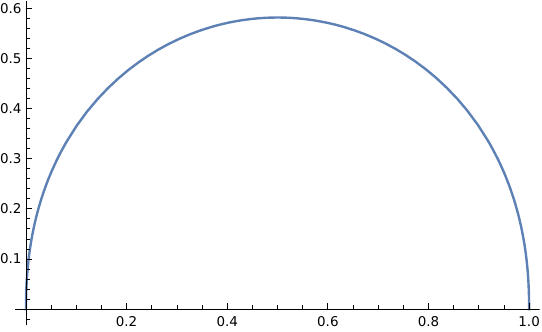}
  \end{minipage}
  \caption{$U_0^G$ for $G= EU_2$ (left) and $G=EU_5$ (right).
  }
  \label{fig1}
  \end{figure}
  This is why we sometimes write by abuse of notation $U_0^G = u_{c_p(G)}$. 
Next we give a universal upper bound on $\inf_{u \in M(\psi)} \mathcal{E}(u)$ independent of the obstacle. We expect this bound to be sharp for large obstacles since it is 
sharp for the minimization among symmetric functions, as we shall see in Section~\ref{section:4}. 
  
 \begin{prop}[A universal bound]\label{prop:univbound}
 Let $\psi \in C([0,1])$ satisfy Assumption~\ref{ass:psi}. Then
\begin{equation}
    \inf_{u \in M(\psi)} \mathcal{E}(u) \leq c_p(G)^p. 
\end{equation}
\end{prop}

\begin{proof}
We provide a comparison function with techniques similar to \cite[Lemma~2.4]{Anna}. For $\delta > 0$ and $c \in (0,c_p(G))$ look at 
\begin{equation}
    u_{\delta,c}(x) := \begin{cases}
     u_c'(0) x & x \in [0,\delta], \\ u_c'(0) \delta + (1-2\delta) u_c \left( \frac{x-\delta}{1-2\delta} \right) & x \in (\delta, 1- \delta),
     \\  u_c'(0) \delta + u_c'(1) (x-(1-\delta)) & x \in [1-\delta, 1].
    \end{cases} 
\end{equation}
Since $u_{\delta,c}(0) = 0$ and $u_{\delta,c}(1) = 0$ (since $u_{c}'(1) = - u_c'(0)$ by \eqref{eq:ucprime}) one easily sees that $u_{\delta,c} \in W^{2,p}(0,1) \cap W_0^{1,p}(0,1)$. We choose now $\delta$ and $c$ such that $u_{\delta,c} \in M(\psi)$. 
Let $\delta > 0$ be arbitrary but fixed and small enough such that $\psi < 0$ on $[0,\delta] \cup [1-\delta, 1]$. Next choose $c =c(\delta) \in (0,c_p(G))$ such that $u_c'(0) \delta > ||\psi||_\infty.$ This can be done since $u_c'(0) = G^{-1}(\frac{c}{2}) \rightarrow \infty$ as $c \uparrow c_p(G)$. Clearly, then $u_{\delta,c} \in M(\psi)$. 

Next we compute using \eqref{eq:ucprime}
\begin{equation}
    G(u_{c,\delta}') = \begin{cases}
     G(u_c'(0)) & x \in [0,\delta] \\ G( u_c'( \frac{x-\delta}{1-2\delta} ) ) & x \in ( \delta, 1-\delta) \\ G(u_c'(1) ) & x \in [1-\delta,1]
    \end{cases} = \begin{cases}
     \frac{c}{2} & x \in [0,\delta], \\ \frac{c}{2}- c \frac{x-\delta}{1-2\delta} & x \in (\delta,1-\delta), \\ -\frac{c}{2} & x \in [1-\delta, 1].
    \end{cases}
\end{equation}
In particular we conclude that $G(u_{c,\delta}')' =- \frac{c}{1-2\delta} \chi_{(\delta,1-\delta)}$. Hence a straightforward computation implies
\begin{equation}
    \mathcal{E}(u_{c,\delta}) = (1-2\delta) \left( \frac{c}{1-2\delta} \right)^p  \leq \frac{c_p(G)^p}{(1-2\delta)^{p-1}}. 
\end{equation}
Since $\delta > 0$ can be chosen arbitrarily small we obtain the claim. 
\end{proof}

Finally we can formulate an existence theorem for small obstacles under a more restrictive energy bound. 

\begin{theorem}[Existence for small obstacles]\label{thm:exinotopti}
Suppose that 
\begin{equation}\label{eq:inficondi}
    \inf_{u \in M(\psi)} \mathcal{E}(u) < \frac{1}{2^p} c_p(G)^p. 
\end{equation}
Then there exists $u_* \in M(\psi)$ such that $\mathcal{E}(u_*) = \inf_{u \in M(\psi)} \mathcal{E}(u)$. Moreover $u_*$ is a critical point in the sense of Definition \ref{def:critical} (and hence satisfies all the regularity properties discussed in Section~\ref{chap:criticalpt}). 
\end{theorem}
\begin{proof}
Let $(u_n)_{n \in \mathbb{N}} \subset M(\psi)$ be a minimizing sequence. We first claim that $(u_n)_{n \in \mathbb{N}}$ is bounded in $W^{1,\infty}(0,1)$. Indeed, since $u_n(0) = u_n(1)= 0$ there must exist $\xi_n \in (0,1)$ such that $u_n'(\xi_n) = 0$. Further let $x_n \in [0,1]$ be such that $|u_n'(x_n) |=||u_n'||_{L^\infty}$. 
Now there exists $\delta > 0$ such that for $n$ large enough one has by Jensen's inequality
\begin{align}\label{eq:estimatedelta}
    \frac{1}{2^p} c_p(G)^p - \delta  & = \int_0^1 |G(u_n')'|^p \dx \geq \Big( \int_0^1 |G(u_n')'| \dx  \Big)^p\\ & \geq \Bigl\vert \int_{\xi_n}^{x_n} G(u_n')' \dx  \Bigm\vert^p  
    =  |G(u_n'(x_n))|^p = G(|u_n'(x_n)|)^p = G(||u_n'||_\infty)^p. 
\end{align}
 From this it directly follows that there exists $C> 0 $ such that $||u_n'||_{L^\infty} \leq C$. 

Therefore we obtain 
\begin{equation}
    \frac{1}{2^p} c_p(G)^p > \int_0^1 |G(u_n')'|^p  \dx = \int_0^1 \dot{G}(u_n')^p |u_n''|^p \dx \geq \inf_{z \in [-C,C]} \dot{G}(z)^p \int_0^1 |u_n''|^p \dx. 
\end{equation}
Since $\inf_{z \in [-C,C]} \dot{G}(z) > 0$ one has that $||u_n''||_{L^p}$ is bounded. Since $||u_n||_{L^p} \leq C ||u_n'||_{L^\infty}$ by the zero boundary conditions we infer that $(u_n)_{n \in \mathbb{N}}$ is bounded in $W^{2,p}(0,1)$, implying that (up to a subsequence) $u_n \rightharpoonup u_*$ in $W^{2,p}(0,1)$. One readily checks (since then $u_n \rightarrow u_*$ in $C([0,1])$) that $u_* \in M(\psi).$ It only remains to show that $\mathcal{E}$ is lower semicontinuous in $u_*$. To see this we use that $u_n '' \rightharpoonup u_*''$ weakly in $L^p(0,1)$ and therefore (since $u_n \rightarrow u_*$ in $C^1$) 
also $\dot{G}(u_n') u_n'' \rightharpoonup \dot{G}(u_*') u_*''$ weakly in $L^p(0,1)$. Using this and $\dot{G}(u_*')^{p-1}|u_*''|^{p-2} u_*'' \in L^{\frac{p}{p-1}}(0,1)$ 
we find 
\begin{align}
    \mathcal{E}(u_*) & = \int_0^1 \dot{G}(u_*')^{p-1} |u_*''|^{p-2} u_*'' \dot{G}(u_*') u_*'' \dx  \\ & = \lim_{n \rightarrow \infty }\int_0^1 \dot{G}(u_*')^{p-1} |u_*''|^{p-2} u_*'' \dot{G}(u_n') u_n'' \dx
    \\ & \leq \liminf_{n\rightarrow \infty} 
    \Bigl(\int_0^1 |\dot{G}(u_*')|^p |u_*''|^p \dx \Bigr)^\frac{p-1}{p}\Bigl( \int_0^1 |\dot{G}(u_n') u_n''|^p \dx \Bigr)^\frac{1}{p}
   \\ &  = 
   \mathcal{E}(u_*)^\frac{p-1}{p} \liminf_{n \rightarrow \infty} \mathcal{E}(u_n)^\frac{1}{p}.
\end{align}
The lower semicontinuity follows by standard rearrangements.
\end{proof}
\begin{remark}\label{rem:plasticexistence}
The energy bound \eqref{eq:inficondi} in the previous theorem is somewhat difficult to translate in a geometric condition on the obstacle. However one can formulate a sufficient condition that is very easy to handle: If there exists $c \in (0, \frac{c_p(G)}{2})$ such that $\psi(x) \leq u_c(x)$ for all $x \in (0,1)$  then \eqref{eq:inficondi} holds true. This becomes obvious when using $u_c$ as a test function and looking at  \eqref{eq:energyuc}.
\end{remark}
\begin{remark}
We will show in the coming sections that without the energy bound \eqref{eq:inficondi} nonexistence may occur --- even though the energy infimum is always finite (cf. Proposition \ref{prop:univbound}). Numerics for $p=2$ and $G =EU_2$ (cf. \cite{Anna}) suggest that minimizers develop vertical slope at $x=0,1$ once the obstacles reach a certain height. Vertical slopes are however not allowed in our class $M(\psi)$. This makes it natural to look at a \emph{relaxation} of the problem. In the case of $p \in (1,\infty)$ and $G= EU_p$ one would want to find a relaxation that maintains the geometric meaning of the functional. This suggests looking at the $p$-elastic energy of curves in the larger class of \emph{pseudographs}, i.e. curves that can consist of graph parts and vertical parts, cf. \cite{MariusExistence}.
In \cite{MariusExistence} it is observed that for $p=2$ the elastic energy possesses a pseudograph minimizer, regardless of the height of the obstacle. We are confident that the analysis carries over to the situation of arbitrary $p$. 
\end{remark}
  
In the coming section we will look at symmetry of minimizers. One should also point out another geometric property which we already derived now --- the \emph{strict concavity}, i.e. $u''<0$, which we have already shown for every critical point in Proposition \ref{prop:deg}. 

\subsection{Symmetry}\label{sec:symmetry}
In this section we use a  \emph{nonlinear Talenti inequality} to prove symmetry of minimizers $u \in M(\psi)$ for suitably small symmetric obstacles $\psi$. While this symmetry property is also expected to hold for general symmetric obstacles, it is somewhat delicate to prove, since the Euler--Lagrange equation is fourth order and hence lacks a maximum principle. Notice also that curvature functionals generally exihibit nonsymmetric critical points despite symmetric boundary data, cf. \cite{Mandel}.
The smallness assumption is needed since Talenti's inequality (see \cite{Talenti}) is --- originally --- a tool to examine linear equations. Thus a control of the nonlinearity is needed. 

\begin{prop}[Symmetry under smallness condition]\label{prop:simsmall}
Suppose that $\psi$ is symmetric and $\psi \vert_{[0,\frac{1}{2}]}$ is increasing. 
Suppose further that there exists some minimizer $u \in M(\psi)$ that satisfies $||u'||_{L^\infty} \leq C_0$, where $C_0 > 0$ is such that 
$
    2\dot{G}(x) + x  \Ddot{G}(x) > 0$  for all $x \in [0,C_0] 
$.
Then one can find a  minimizer $v \in M(\psi)$ which is additionally symmetric.
\end{prop}
\begin{proof}
Let $u \in M(\psi)$ be as in the statement. Define $f := G(u')'$. Then $f \in L^p(0,1)$. Let $f_* \in L^p(0,1)$ be the \emph{symmetric decreasing rearrangement} of $f$, cf. \cite[Section~ 7.1]{MariusFlow}. Note in particular that $||f_*||_{L^p(0,1)} = ||f||_{L^p(0,1)}$.\\
\textbf{Intermediate claim.} There exists a 
solution $v \in W^{2,p}(0,1)$  of 
\begin{equation}
    \begin{cases}
        G(v')' = f_* & \textrm{ in } (0,1), \\ v(0) = v(1) = 0,
    \end{cases}
\end{equation}
which additionally satisfies 
\begin{itemize}
    \item $v = v(1-\cdot)$ (i.e. it is 
    symmetric),
    \item $v \geq u_*$, where $u_*$ is the symmetric decreasing rearrangement of $u$.
\end{itemize}
Notice that such $v$ lies in $M(\psi)$ since $v \geq u_* \geq \psi_* = \psi$.  Here we have used that the symmetric decreasing rearrangement is order-preserving (cf.  \cite[Section~3.3]{LiebLoss}) and $\psi_* = \psi$ since $\psi$ is symmetric and radially decreasing (due to the fact that $\psi\vert_{[0,\frac{1}{2}]}$ is increasing). Now 
\begin{equation}
    \mathcal{E}(v) = ||G(v')'||_{L^p(0,1)}^p = ||f_*||_{L^p(0,1)}^p = ||f||_{L^p(0,1)}^p = ||G(u')'||_{L^p(0,1)}^p = \mathcal{E}(u). 
\end{equation}
Therefore $v$ is a symmetric minimizer and thus the proof is finished once the intermediate claim is shown. The proof of the intermediate claim is an application of \cite[Lemma 6.3]{MariusFlow}, which can straightforwardly be generalized to the case $p \neq 2$. To check the prerequisites of \cite[Lemma 6.3]{MariusFlow} we need to ensure that $\frac{1}{2}||f||_{L^p(0,1)} \leq ||G||_\infty$ and that $\frac{1}{G^{-1}}$ is convex on $[0, G(||u'||_\infty)]$. \\
\textbf{Step 1}. We show first that $\frac{1}{2}||f||_{L^p(0,1)} \leq ||G||_\infty$. Indeed, Proposition \ref{prop:univbound} yields
\begin{equation}
    ||f||_{L^p(0,1)}^p = \mathcal{E}(u) \leq c_p(G)^p, 
\end{equation}
i.e., by Assumption \ref{ass:G}, $||f||_{L^p(0,1)} \leq c_p(G) = 2 \lim_{z \rightarrow \infty} G(z)= 2 ||G||_\infty$ and thus $\frac{1}{2}||f||_{L^p(0,1)} \leq ||G||_\infty$. \\
\textbf{Step 2}. We show convexity of $\frac{1}{G^{-1}}$ on $[0,G(||u'||_{L^\infty})]$. Notice that due to the assumption this is a subset of $[0, G(C_0)]$. 
We compute 
 \begin{equation}
     \left( \frac{1}{G^{-1}} \right)''(s) = \frac{1}{G^{-1}(s)^3 \dot{G}(G^{-1}(s))^3} ( 2 \dot{G}(G^{-1}(s)) + G^{-1}(s) \Ddot{G}(G^{-1}(s))) .
 \end{equation}
This expression is positive on $[0,G(C_0)]$ iff 
$
     (2 \dot{G} + x \cdot  \Ddot{G}) (G^{-1}(s)) > 0 
$
 for all $s \in [0,G(C_0)]$. Monotonicity of $G,G^{-1}$ implies the claim.
\end{proof}

\begin{remark}\label{rem:when2dotG+xDdotG>0}
 We remark that for 
 $G = EU_p$ one has $2 \dot{G}(x) +x \Ddot{G}(x) > 0$ for all $x \in \left[0, \sqrt{\frac{2p}{p-1}}\right).$
 \end{remark}
 
 \begin{cor} \label{Cor:symmetry_minimizer}
 Suppose that $\psi$ is symmetric and $\psi \vert_{[0,\frac{1}{2}]}$ is increasing.  Moreover, suppose that $\psi \leq u_c$ for some $c \in (0,\frac12 c_p(G))$ which is such that  $2 \dot{G}(x) + x \Ddot{G}(x) > 0$ on $[0,\sqrt[p]{G^{-1}(c^p)}]$. Then there exists a symmetric  minimizer in $M(\psi)$.
 \end{cor}
 \begin{proof}
 It follows from Remark \ref{rem:plasticexistence} that there exists some minimizer $u \in M(\psi).$ In order to show that there exists also a symmetric minimizer, we apply Proposition~ \ref{prop:simsmall}. In order to do so it remains to ensure that $||u'||_{L^\infty}\leq C_0 := \sqrt[p]{G^{-1}(c^p)}$, where $c$ is as in the statement. Let $\mathcal{E}^* := \mathcal{E}(u)$. The fact that $u$ is minimizing and $\psi \leq u_c$ yields that $\mathcal{E}^* \leq c^p$. Now Jensen's inequality yields
 \begin{equation}\label{eq:estiJensen}
    c^p \geq \mathcal{E}^* =  \mathcal{E}(u)   \geq [G(u'(1))- G(u'(0))]^p \geq G(\max\{|u'(0)|,|u'(1)|\})^p, 
 \end{equation}
 as $G$ is odd and $u'(0),u'(1)$ have opposite signs. By concavity (cf. Corollary~ \ref{cor:concavity}), one has $||u'||_{L^\infty} = \max \{ |u'(0)|, |u'(1)| \}$ and thus a simple rearrangement of  \eqref{eq:estiJensen} yields
 \begin{equation}
     ||u'||_{L^\infty} = \max \{ |u'(0)|, |u'(1)| \} \leq \sqrt[p]{G^{-1}(c^p)}. \qedhere
 \end{equation}
 \end{proof}
\begin{remark}
 The previous corollary and  Remark \ref{rem:when2dotG+xDdotG>0} 
 yield that if $G = EU_p$, $\psi$ is symmetric, $\psi\vert_{[0,\frac{1}{2}]}$ is increasing and $\psi < u_{\hat{c}(p)}$ for $$\hat{c}(p) :=  \min \left\lbrace \frac{1}{2}c_p(EU_p), \sqrt[p]{EU_p((\tfrac{2p}{p-1})^{\frac{p}{2}})} \right\rbrace$$ then one finds a symmetric minimizer in $M(\psi)$. Due to Remark \ref{rem:increasingness}, $u_{\hat{c}(p)}$ is also the optimal choice 
 among all functions in $\{ u_c : c \in ( 0, c_p(G)) \}$  satisfying the conditions in Corollary \ref{Cor:symmetry_minimizer}. 
 \end{remark}

\begin{figure}[h!]
    \centering
\includegraphics[scale=0.7]{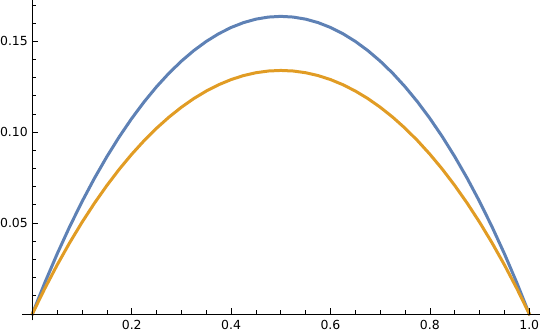}
    \caption{The
    function $u_{\hat{c}(p)}$ for $p = \frac{3}{2}$ (blue), $p = 17$ (yellow).}
    \label{fig:enter-label}
\end{figure}

This symmetry result shows that under suitable conditions on $G$ and $\psi$ minimization in $M(\psi)$ and minimization in the smaller class $M_{\rm sym}(\psi)$ yield the same result. In Section \ref{section:4} we will study minimization in $M_{\rm sym}(\psi)$ in detail and show further properties, such as uniqueness and explicit parametrization of minimizers, if $\psi$ is a \emph{cone obstacle}.

\section{Nonexistence for large cone obstacles}\label{sec:nonexistence}
In this section we will use the properties of minimizers that we studied to obtain a nonexistence result for suitably large symmetric cone obstacles (see Remark \ref{rem:symcone}). For this nonexistence result we do not have to require symmetry of minimizers, only the obstacle $\psi$ is required to satisfy $\psi= \psi(1-\cdot)$. Hence the tip of the cone is at $x=\frac12$ and $\theta $ in Remark \ref{rem:symcone} is equal to $\frac12$. We will derive a sharp bound on the $L^\infty$ norm of minimizers for general $G$, provided that  the supremum in the following lemma is finite. 

\begin{prop}[Nonexistence for  solutions for large obstacles]
\label{Lemma:non-existence}
Let $\psi$ be a symmetric cone obstacle. 
Suppose that $u \in M(\psi)$ is a minimizer. Then
\begin{equation}\label{eq:supfinite}
    ||u||_{L^\infty} \leq \frac{1}{2}\sup_{A  > 0} \frac{\int_0^A \frac{s\dot{G}(s)}{(A-s)^\frac{1}{p}}  \ds }{ \int_0^A \frac{\dot{G}(s)}{(A-s)^\frac{1}{p}} \ds }.
\end{equation}
In particular, if $||\psi^+||_{L^\infty}$ is larger than the right hand side in \eqref{eq:supfinite}, then there does not exist a minimizer $u \in M(\psi)$.
\end{prop}
\begin{proof}
Let $x_0 \in (0,1)$ be a point where $u'(x_0)= 0$. Such $x_0$ exists due to the boundary conditions $u(0)= u(1) = 0$. Without loss of generality we may assume $x_0 \in (0,\frac{1}{2}]$, otherwise we may consider $u(1-\cdot)$, which is also a minimizer. By concavity (cf. Corollary~ \ref{cor:concavity}) one concludes that $u'( \frac{1}{2}) \leq 0$. 
Recall from Propositions~\ref{prop:regu1} and \ref{prop:regu2} that $u \in W^{3,q}(0,1)$ for some $q >1$ and that since $u> \psi$ on $[0,1] \setminus \{\frac12\}$
\begin{equation}
    \dot{G}(u'(x)) w'(x) \equiv \mathrm{const.} \quad \textrm{on $(0,\tfrac{1}{2})$},
\end{equation}
where $w$ is Euler's substitution. 
This in particular yields some $B \in \mathbb{R}$ such that 
\begin{equation}
 w'(x) = \frac{B}{\dot{G}(u'(x))}  \qquad \forall x  \in (0 , \tfrac{1}{2}).
\end{equation}

We can now multiply both sides with $w(x)^\frac{1}{p-1} = p^{\frac{1}{p-1}}\dot{G}(u'(x)) (-u''(x))$ to get
\begin{equation}
    w'(x) w^\frac{1}{p-1}(x) =-  p^{\frac{1}{p-1}} B u''(x)  . 
\end{equation}
Integrating once and using $w(0) = 0$ yields 
\begin{equation}
  \frac{p-1}{p}  w(x)^\frac{p}{p-1} = -p^\frac{1}{p-1}B (u'(x) - u'(0)) = p^\frac{1}{p-1}B (u'(0) - u'(x)). \label{eq:w_explicit} 
\end{equation}
Note that $u'(x) < u'(0)$ for all $x> 0$ since by Proposition \ref{prop:deg} one has $u''(x)< 0$ for all $x \in (0, 1)$. This fact also implies that $B> 0$.
Rewriting now in terms of $u$ we find
\begin{equation}
 (p-1)(-u''(x))^p \dot{G}(u'(x))^p = B (u'(0) - u'(x)).
\end{equation}
Taking the $p$-th root we infer 
\begin{equation}
    \frac{-u''(x) \dot{G}(u'(x))}{(u'(0)- u'(x))^\frac{1}{p}} = 
    \Bigl(\frac{1}{p-1}B \Bigr)^\frac{1}{p} \quad \forall x \in (0, \tfrac{1}{2}) 
\end{equation}
and hence 
\begin{equation}\label{eq:Fcircu}
    \int_{u'(x)}^{u'(0)} \frac{\dot{G}(s)}{(u'(0)- s)^\frac{1}{p}} \ds = 
    \Bigl(\frac{1}{p-1}B \Bigr)^\frac{1}{p} x
    \quad \forall x \in [0, \tfrac{1}{2}].  
\end{equation}
Let now $F : (-\infty, u'(0)] \rightarrow \mathbb{R}$ be given by
\begin{equation}
    F(y) := 
     \Bigl(\frac{p-1}{B} \Bigr)^{\frac{1}{p}}
    \int_{y}^{u'(0)} \frac{\dot{G}(s)}{(u'(0)- s)^\frac{1}{p}} \ds .
\end{equation}
Then $F$ is invertible as $F' < 0$ on $(-\infty,u'(0))$. 
Further,  $F(u'(0)) = 0$ since 
\begin{equation}
    F(y) \leq C \frac{1}{B} ||\dot{G}||_\infty  (u'(0)-y)^{1- \frac{1}{p}}. 
 \end{equation}
 Moreover, $u'(x) = F^{-1}(x)$ for all $x \in [0, \frac{1}{2}]$, as \eqref{eq:Fcircu} shows. Next we determine the value of $\big(\frac{p-1}{B} \big)^{\frac{1}{p}}$. 
 To this end we look at 
 \begin{equation}
     \Bigl(\frac{p-1}{B} \Bigr)^{\frac{1}{p}}\int_{0}^{u'(0)} \frac{\dot{G}(s)}{(u'(0)- s)^\frac{1}{p}} \ds  = F(0) = F(u'(x_0)) = x_0 \, .
 \end{equation}
 This yields 
 \begin{align}
     u(x)&   = \int_0^x F^{-1}(y) \dy = \int_{F^{-1}(0)}^{F^{-1}(x)} y F'(y) \dy = \Bigl(\frac{p-1}{B} \Bigr)^{\frac{1}{p}} 
     \int_{u'(x)}^{u'(0)} \frac{y \dot{G}(y)}{(u'(0)-y)^\frac{1}{p}} \dy  
    \\  &  = x_0 \frac{1}{\int_{0}^{u'(0)} \frac{\dot{G}(s)}{(u'(0)- s)^\frac{1}{p}} \ds  } \int_{u'(x)}^{u'(0)} \frac{y \dot{G}(y)}{(u'(0)- y)^\frac{1}{p}} \dy .
 \end{align}
 Since $u$ is concave, $u$ attains a global maximum value at each point $z \in (0,1)$ with $u'(z)= 0$. In particular (since $x_0 \leq \frac{1}{2})$
 \begin{equation}
     ||u||_\infty = u(x_0) =   \frac{x_0}{\int_{0}^{u'(0)} \frac{\dot{G}(s)}{(u'(0)- s)^\frac{1}{p}} \ds  } \int_{0}^{u'(0)} \frac{y \dot{G}(y)}{(u'(0)- y)^\frac{1}{p}} \dy \leq \frac{1}{2} \frac{\int_{0}^{u'(0)} \frac{s \dot{G}(s)}{(u'(0)- y)^\frac{1}{p}} \; \mathrm{d}s }{ \int_0^{u'(0)} \frac{\dot{G}(s)}{(u'(0)- s)^\frac{1}{p}} \ds}.
 \end{equation}
 Thereupon \eqref{eq:supfinite} follows. For the last sentence of the statement observe that if $u \in M(\psi)$ is a minimizer then $||u||_{L^\infty} \geq ||\psi^+||_{L^\infty}$. If now $||\psi^+||_{L^\infty}$ is larger than the right hand side in \eqref{eq:supfinite} this yields a contradiction. 
\end{proof}

The bound in \eqref{eq:supfinite} is only nontrivial provided that the supremum in \eqref{eq:supfinite} is finite. This is the content in the following lemma whose proof we postpone to the appendix. 

\begin{lemma}[Boundedness of \eqref{eq:supfinite}, Proof in Appendix \ref{ref:app}]\label{lem:6.2}
Let $G$ satisfy Assumption \ref{ass:G} and assume additionally that  there exists $\epsilon >0$ such that 
\begin{equation}\label{eq:asymptG}
    \lim_{z \rightarrow \infty}  z^{2+\epsilon} \dot{G}(z) = 0 .
\end{equation}
Then the function 
\begin{equation}
  H : (0,\infty) \rightarrow \mathbb{R}, \quad   H(A) := \frac{\int_0^A \frac{s\dot{G}(s)}{(A-s)^\frac{1}{p}}  \ds }{ \int_0^A \frac{\dot{G}(s)}{(A-s)^\frac{1}{p}} \ds }
\end{equation}
is continuous, bounded on $(0,\infty)$ and satisfies 
\begin{equation}
    \lim_{A \rightarrow 0} H(A) = 0, \qquad \lim_{A \rightarrow \infty} H(A) = \frac{2}{c_p(G)}{\int_0^\infty s \dot{G}(s) \; \mathrm{d}s}. 
 \end{equation}
\end{lemma}

\begin{remark}
We remark that for each $p \in (1,\infty)$ the $p$-elastica functional $G= EU_p$ satisfies \eqref{eq:asymptG} for any $\epsilon \in (0, 1- \frac{1}{p})$. Hence this nonexistence result applies to the $p$-elastica functional.
\end{remark}
\begin{remark}
Under additional assumptions, 
$H$ can be shown to be (\emph{eventually}) \emph{increasing}, e.g. in the case of $p = 2$ and $G = EU_2$, cf. \cite[Lemma~7.6]{MariusFlow}. If this monotonicity can be shown, then the previous lemma yields actually an \emph{explicit} bound for nonexistence, namely  $\frac{1}{2}H(\infty) = \frac{1}{c_p(G)} ||x \cdot \dot{G}||_{L^1(0,\infty)}$, cf. \eqref{eq:supfinite} and Lemma~\ref{lem:6.2}. In the case of $G = EU_p$, the same sharp bounds can also be obtained using the geometry of $p$-elastic curves, which we pursue in Section \ref{section:4}. 
\end{remark}

\section{Uniqueness of symmetric minimizers} \label{section:4}
In this section we consider problem \eqref{obstacle_problem:1.1}, i.e., the minimization problem 
\[
\inf_{v \in M_{\rm sym}(\psi)} \mathcal{E}_p(v), 
\]
where the functional $\mathcal{E}_p$ and the admissible set $M_{\rm sym}(\psi)$ are respectively defined by \eqref{eq:p-elastic_energy} and \eqref{eq:msym}, and $\psi$ denotes a cone obstacle. 
Here we recall that the $p$-elastic energy $\mathcal{E}_p$ is given by the functional $\mathcal{E}$ with $G=EU_p$,  which is defined by \eqref{eq:EU_p}. 

\subsection{Existence of minimizers} \label{section:6.1}
Following the strategy in Section \ref{section:existence-mini}
we first prove the existence of minimizers of problem \eqref{obstacle_problem:1.1}. 
Taking $G=EU_p$ in \eqref{eq:defcp}, we define the constant $c_p$ by 
\begin{equation} \label{def_cp_sec6}
c_p := c_p(EU_p) = 2 \int^\infty_0 \dfrac{1}{(1+t^2)^{\frac{3}{2} - \frac{1}{2p}}} \dt.  
\end{equation}

\begin{theorem} \label{thm:sym-existence}
Let $\psi : [0,1] \to \mathbb{R}$ satisfy Assumption~{\rm \ref{ass:psi}} and 
suppose that 
\begin{equation} \label{eq:smallness_220115}
\inf_{v \in M_{\rm sym}(\psi)}\mathcal{E}_p(v)<c_p^p. 
\end{equation}
Then there exists a minimizer $u \in M_{\rm sym}(\psi)$ of $\mathcal{E}_p$ in $M_{\rm sym}(\psi)$.
\end{theorem}
\begin{proof}
Let $\{ u_k \}_{k \in \mathbb{N}} \subset M_{\rm sym}(\psi)$ be a minimizing sequence for $\mathcal{E}_p$.
By assumption~\eqref{eq:smallness_220115} we may assume that  
\begin{equation}\label{eq:0907-1}
\mathcal{E}_p(u_k) \leq (c_p-2\delta)^p  \quad  \text{for all} \quad k \in \mathbb{N}
\end{equation}
for some $\delta>0$. 
Since $EU_p$ is an odd function, we have 
\[
EU_p(u'_k(x))=EU_p(-u'_k(1-x))=-EU_p(u'_k(1-x))
\]
for all $x \in [0,1]$ and $k \in \N$, 
and then 
\begin{align*}
2EU_p(|u'_k(x)|)&=2|EU_p(u'_k(x))| \\
&= | EU_p(u'_k(x))-EU_p(u'_k(1-x)) | \\
&=\Bigl| \int_x^{1-x} EU'_p(u_k'(\xi)) u_k''(\xi) \dxi \Bigr| \\
&\leq \Bigl( \int_0^{1} |EU'_p(u_k'(\xi))|^p |u_k''(\xi)|^p \dxi \Bigr)^{\frac{1}{p}} 
= \mathcal{E}_p(u_k)^{\frac{1}{p}}
\end{align*}
for all $k \in \N$. 
This together with \eqref{eq:0907-1} yields a uniform estimate
\begin{equation} \label{eq:0907-2}
\|u_k'\|_{L^{\infty}(0,1)} \leq EU_p^{-1}(\tfrac{c_p}{2}-\delta)<\infty \quad \text{for all} \quad k \in \N.
\end{equation}
Then, along the same lines as in the proof of Theorem~ \ref{thm:exinotopti}, we complete the proof of Theorem~ \ref{thm:sym-existence}. 
\end{proof}

Let $u \in M_{\rm sym}(\psi)$ be a minimizer of $\mathcal{E}_p$ in $M_{\rm sym}(\psi)$, which is obtained by Theorem~\ref{thm:sym-existence}.  
Similar to \eqref{eq:varpos}, we see that
\begin{align}\label{eq:pre_sym_ineq}
D\mathcal{E}_p(u)(v-u) \geq0\quad \text{for any} \quad v\in \Msym(\psi).
\end{align}
While the admissible set is restrictive, employing the same argument as in \cite[Lemma 2.1]{Y2021}, we improve \eqref{eq:pre_sym_ineq} to  
\begin{align}\label{eq:sym_ineq}
D\mathcal{E}(u)(v-u) \geq0\quad \text{for any} \quad v\in M(\psi).
\end{align}
(For a proof of \eqref{eq:sym_ineq}, see Appendix \ref{appendix6}.)
Recalling that $\mathcal{E}_p$ is given by $\mathcal{E}$ with $G=EU_p$, we observe from \eqref{eq:sym_ineq} and Theorem~\ref{thm:exinotopti} that $u$ is a critical point of $\mathcal{E}_p$ in $M(\psi)$ 
in the sense of Definition~ \ref{def:critical}. 
Then, Corollary \ref{cor:concavity} implies the concavity of $u$, and Propositions \ref{prop:regu1} and \ref{prop:regu2} give us the optimal regularity of $u$. 
Moreover, we notice that the (measure-valued) Euler--Lagrange equation \eqref{eq:solELeq} becomes 
\begin{equation}\label{eq:solELeq-EU_p}
    -\int_0^1  \Bigl[ p  \frac{ |\kappa_u|^{p-2}\kappa_u}{1+(u')^2} \phi'' + (1-3p) \frac{|\kappa_u |^{p} u'}{\sqrt{1+(u')^2}} \phi' \Bigr] \dx = \int \phi \; \mathrm{d}\mu. 
\end{equation}
As we mentioned in Section \ref{section:intro}, if $\gamma$ is a critical point of $\mathcal{E}_p$ and parameterized by arc length, then the curvature $\kappa$ satisfies \eqref{EL-eq:1.2} in the sense of distributions. 
Thus it is quite natural that we have: 

\begin{lemma}[Proof in Appendix \ref{ref:app}]\label{p-EL}
Let $u \in W^{3,q}(0,1)$ for some $q>1$ and $\tilde{\kappa}_u$ be the curvature of ${\rm graph}(u)$ parameterized by arc length, i.e., 
\[
\tilde{\kappa}_u(\mathbf{s}):= \kappa_u(x(\mathbf{s})),
\]
where $x(\mathbf{s})$ is the inverse of $\mathbf{s}(x) = \int_0^x \sqrt{1+u'(\xi)^2}\dxi$. 

If $\kappa_u$ satisfies on $E:=(x_1, x_2) \subset (0,1)$
\begin{equation} 
\label{eq:0709-2}
\int^{x_2}_{x_1}  \Bigl[ p  \frac{ |\kappa_u|^{p-2}\kappa_u}{1+(u')^2} \phi'' + (1-3p) \frac{|\kappa_u |^{p} u'}{\sqrt{1+(u')^2}} \phi' \Bigr] \dx =0 
\end{equation}
for all $\phi \in C^{\infty}_0(x_1, x_2)$, then $\tilde{\kappa}_u$ satisfies 
\begin{equation} 
\label{eq:0710-4}
\int^{\mathbf{s}(x_2)}_{\mathbf{s}(x_1)}  \bigl[ p |\tilde{\kappa}_u|^{p-2} \tilde{\kappa}_u \varphi'' + (p-1) |\tilde{\kappa}_u|^{p}\tilde{\kappa}_u \varphi \bigr]\, \mathrm{d}\mathbf{s} =0
\end{equation}
for all $ \varphi \in C^{\infty}_0(\mathbf{s}(x_1), \mathbf{s}(x_2))$.
\end{lemma}

\subsection{Explicit formulae of free $p$-elastica}
Let $u$ be a minimizer of $\mathcal{E}_p$ in $\Msym(\psi)$.
Since Assumption~\ref{ass:psi} implies $u > \psi$ near 
$x=0$,
we infer from Lemma~\ref{p-EL} that $u$ near $x=0$ can be characterized as
a $C^2$-planar curve whose curvature $\kappa\in C([0,L])$ vanishes at $x=0$ and satisfies 
\begin{align}\label{eq:p-free}
p\big(|\kappa|^{p-2}\kappa \big)'' + (p-1)|\kappa|^p\kappa =0
\end{align}
in the sense of distributions.

Keeping this fact in mind, 
let us consider $C^2$-curves $\gamma:[0,L] \to \R^2$ whose curvature $\kappa\in C([0,L])$ satisfies $\kappa(0)=0$ and 
\begin{align}\label{eq:planar-EL}
\int_0^L \bigl[ p |\kappa|^{p-2} \kappa \varphi'' + (p-1) |\kappa|^{p}\kappa \varphi \bigr]\, \ds =0
\quad \text{for all}\quad \varphi \in C^{\infty}_0(0,L),
\end{align}
where $s$ denotes the arc length parameter and $L$ denotes the length of $\gamma$.
Set 
\begin{equation}
\label{eq:w-kappa}
\omega(s) := |\kappa(s)|^{p-2} \kappa(s), 
\end{equation}
which is equivalent to $\kappa(s)=|\omega(s)|^{\frac{2-p}{p-1}} \omega(s)$. 

\begin{remark}\label{rem:6.3toberefereedtointheintro}
Euler's substitution $w_u$ (see Definition~\ref{def:w}) and $\omega$ introduced in \eqref{eq:w-kappa} are different from each other and play different roles.
In fact, with the choice of $G=EU_p$ in \eqref{gEs}, one notices that 
\[
w_u(x)=-p|\kappa_u(x)|^{p-2} \kappa_u(x)\big(1+u'(x)^2 \big)^{\frac{p-1}{2p}}.
\]
On the other hand, for ${\rm graph}(u):=\Set{(x, u(x)) | x \in [0, 1]}$ and its arc  length $\mathbf{s}(x)=\int_0^x\sqrt{1+u'(y)^2}\dy$, $\omega$ becomes
\[
\omega(\mathbf{s}(x))=|\kappa_u(x)|^{p-2}\kappa_u(x).
\]
Notice that in this sense $\omega$ and $w_u$ have in common that they involve the term $|u''|^{p-2}u''$ as the term of highest order.
Their roles are as follows:
Euler's substitution allows $w_u$ to satisfy a second order differential equation without zeroth-order term; 
the transformation $\omega$ changes the quasilinear Euler--Lagrange equation into a semilinear equation. 
\end{remark} 
When $\kappa \in C([0,L])$, $\omega$ also belongs to $C([0,L])$.
Inserting this into \eqref{eq:planar-EL}, 
we can reduce \eqref{eq:planar-EL} to
\begin{equation}
\label{eq:0606-2}
\int^L_0 \Bigl[ p \omega \varphi'' + (p-1) |\omega|^{\frac{2}{p-1}} \omega \varphi \Bigr] \, \ds =0 
\quad \text{for all} \quad \varphi \in C^{\infty}_0(0,L).
\end{equation}
Therefore, $\omega\in C([0,L])$ becomes a weak solution of semilinear equation, and then adopting ideas from \cite[Proof of Theorem 3.9]{DDG08} and \cite[Proof of Proposition 3.2]{Anna}, we obtain the following regularity result of $\omega$.
\begin{lemma}\label{lem:regularity-omega}
If $\omega\in C([0,L])$ satisfies \eqref{eq:0606-2}, then it follows that $\omega \in C^2([0,L])$ and \begin{equation}
\label{eq:w-equation}
p \omega''(s) +(p-1)|\omega(s)|^{\frac{2}{p-1}}\omega(s)=0, \quad s\in(0,L).
\end{equation}
\end{lemma}
The proof of Lemma~\ref{lem:regularity-omega} is similar to that of \cite[Proof of Proposition 3.2]{Anna}, and given in the Appendix.

We consider the Cauchy problem for equation \eqref{eq:w-equation}: 
\begin{align}
\label{eq:0606-3}
\begin{cases}
p \omega''(s) +(p-1)|\omega(s)|^{\frac{2}{p-1}}\omega(s)=0, \\
\omega(0)=0, \\
\omega'(0) = \lambda, 
\end{cases}
\end{align}
where $\lambda \in \mathbb{R}$. 
For each $\lambda \in \mathbb{R}$, problem \eqref{eq:0606-3} possesses a unique solution in a neighborhood of $s=0$
since $F(\xi)=|\xi|^{\frac{2}{p-1}}\xi$ is locally Lipschitz on $\R$. 
In order to find an explicit form of the solution to \eqref{eq:0606-3}, we introduce a generalized trigonometric function.
For each $q,r \in (0, \infty)$, we define $\sin_{q,r} x$ via its 
inverse function 
\begin{equation}
\label{eq:sin_pq}
\sin_{q, r}^{-1}x := \int_0^x \frac{1}{(1-t^r)^{\frac{1}{q}}} \dt, \quad 0 \leq x \leq 1,
\end{equation}
and $\pi_{q,r}$ by 
\[
\pi_{q, r}:= 2\sin_{q,r}^{-1}1 = 2\int_0^1 \frac{1}{(1-t^r)^{\frac{1}{q}}}\dt = \frac{2}{r}\, \mathrm{B}\bigl( \tfrac{1}{q'}, \tfrac{1}{r} \bigr),
\]
where $q':= q / (q-1)$ and $\mathrm{B}$ denotes the beta function, i.e.,
\[
\mathrm{B}(x,y) = \mathrm{B}(y,x) = \int_0^{\infty}\frac{t^{x-1}}{(1+t)^{x+y}}\dt 
= \int_0^{1}t^{x-1} (1-t)^{y-1} \dt, 
\quad x, y>0.
\]
First the function $\sin_{q,r} x$ is defined on $[0, \pi_{q,r}/2]$, and then it is extended symmetrically to 
the interval $[0, \pi_{q,r}]$, more precisely $\sin_{q,r} x : = \sin_{q,r} (\pi_{q,r} - x)$ for $x \in [ \pi_{q,r}/2 , \pi_{q,r}]$. Further it can be extended as a $C^1$-function on $\R$ first extending it to $[-\pi_{q,r}, \pi_{q,r}]$ as an odd function and then to all of $\R$ as a $2 \pi_{q,r}$-periodic function. Notice that $\sin_{q,r} x$ is strictly increasing on $[0, \pi_{q,r}/2]$.
The function  $\cos_{q,r}:\mathbb{R} \to \mathbb{R}$ is defined by  
\begin{equation}
\label{cos_pq}
 \cos_{q,r}x:=\frac{d}{dx} \sin_{q,r} x, \quad x \in \mathbb{R}. 
\end{equation}
It is known that
\begin{equation}
\label{eq:0711-4}
|\cos_{q,r} x|^q + |\sin_{q,r} x|^r =1, \quad x \in \mathbb{R},
\end{equation}
and 
$y=\sin_{q,r} x$ satisfies 
\begin{equation} 
\label{eq:0606-4}
\bigl( |y'|^{q-2} y' \bigr)' + \frac{r}{q'}|y|^{r-2}y=0,  
\end{equation}
see e.g. \cite[(2.7)]{EGL} for \eqref{eq:0711-4} and \cite[p.1510]{KT} for \eqref{eq:0606-4}.

\begin{lemma} \label{lem:0606-1}
For each $\lambda \in \mathbb{R}$, the solution of \eqref{eq:0606-3} is given by 
\[
 \omega(s) = (p')^{\frac{1}{p'}} \lambda|\lambda|^{-\frac{1}{p}} \sin_{2,2p'} \bigl( (p')^{-\frac{1}{p'}}|\lambda|^{\frac{1}{p}} s \bigr),  
\]
where $p'=p/(p-1)$. 
\end{lemma}

\begin{proof}
Let $y(x):= \sin_{2,r}{x}$. 
Set $\omega(x):= A y(\alpha x)$ and choose $(A, \alpha, r) \in \mathbb{R} \times \mathbb{R}_{>0} \times \mathbb{R}_{>0}$ suitably later. 
It follows from the definition of $\sin_{q,r}$ that $y(0)=0$, and then $\omega(0)=0$.
Since now $y(x)$ satisfies 
\[
y''(x) + \frac{r}{2}|y(x)|^{r-2}y(x)=0,
\]
we see that 
\[
\omega''(x)= A \alpha^2 y''(\alpha x) = - \dfrac{A \alpha^2 r}{2} |y(\alpha x)|^{r-2} y(\alpha x) = - \dfrac{r}{2} |A|^{2-r} \alpha^2 |\omega(x)|^{r-2} \omega(x). 
\]
Thus, setting 
\[
\dfrac{r}{2} |A|^{2-r} \alpha^2= \dfrac{p-1}{p}, \quad r-2 = \dfrac{2}{p-1}, 
\]
i.e., 
\[
r= 2 p', \quad \alpha = \dfrac{1}{p'}|A|^{\frac{1}{p-1}}, 
\]
we see that 
\[
\omega(x) = A y((p')^{-1}|A|^{\frac{1}{p-1}} x)= A \sin_{2,2p'}{\bigl( (p')^{-1}|A|^{\frac{1}{p-1}} x \bigr)} 
\]
satisfies the differential equation in \eqref{eq:0606-3}. 
Moreover, since 
\[
(\sin_{2,q} x)' = \cos_{2,q}x \quad \text{and} \quad \cos_{2,q} 0 =1, 
\]
one obtains that 
$\omega'(0)= (p')^{-1} A|A|^{\frac{1}{p-1}}$.
Thus, setting 
\[
A= (p')^{\frac{1}{p'}} \lambda|\lambda|^{-\frac{1}{p}} \quad \text{and} \quad \alpha=(p')^{-1} |A|^{\frac{1}{p-1}}= (p')^{-\frac{1}{p'}}|\lambda|^{\frac{1}{p}}, 
\]
we see that $\omega$ satisfies $\omega'(0)=\lambda$.  
Therefore Lemma \ref{lem:0606-1} follows. 
\end{proof}

For $\lambda \ge 0$, we define 
\begin{equation} 
\label{eq:0711-3}
\omega_{\lambda}(s) :=- (\lambda p')^{\frac{1}{p'}} \sin_{2,2p'} \bigl((p')^{-\frac{1}{p'}}\lambda^{\frac{1}{p}} s \bigr). 
\end{equation}
It follows from Lemma \ref{lem:0606-1} that $\omega_\lambda$ is the unique solution of \eqref{eq:w-equation} with $\omega(0)=0$ and $\omega'(0)=-\lambda$.
Set 
\begin{equation} 
\label{eq:0710-6}
k_{\lambda}(s):=|\omega_{\lambda}(s)|^{\frac{2-p}{p-1}} \omega_{\lambda}(s).   
\end{equation}
For each $\lambda>0$, let us introduce the curve $(X_{\lambda}, Y_{\lambda})$ (that is unique modulo Euclidean transformations) parameterized by 
arc length 
and 
such that its curvature is $k_{\lambda}$.

\begin{lemma} \label{lemma:pre-G}
Let $\lambda>0$ and set for $s \in [0,2L_{\lambda}]$
\begin{equation} 
\label{eq:0711-2}
X_{\lambda}(s) = \frac{1}{\lambda p'} \int_0^s |k_{\lambda}(t)|^{p} \dt,
\quad 
Y_{\lambda}(s) =-\frac{p-1}{\lambda}\int_0^s |k_{\lambda}(t)|^{p-2}k_{\lambda}'(t)\dt, 
\end{equation}
where $p'=p/(p-1)$. Then 
\begin{align}
& X_\lambda(0)=Y_\lambda(0)=0, \quad X'_\lambda(0)=0, \quad Y'_\lambda(0)=1, \label{eq:G-tangent} \\
& X'_\lambda(s)^2 + Y'_\lambda(s)^2 =1 \quad \text{for all} \quad s \ge 0, \label{eq:arc-length-para} \\
& X'_\lambda(s) Y''_\lambda(s) - X''_\lambda(s) Y'_\lambda(s) = k_\lambda(s) \quad \text{for all} \quad s \ge 0. \label{eq:G-curvature=k_l} 
\end{align}
\end{lemma}

\begin{proof}
First we prove \eqref{eq:G-tangent}. 
It clearly follows from \eqref{eq:0711-2} that $X_\lambda(0)=Y_\lambda(0)=0$. 
By \eqref{eq:0710-6} we see that $k_{\lambda}(s)$ is differentiable and 
\begin{align} \label{eq:0712-1}
k_{\lambda}'(s)&= \frac{1}{p-1} |\omega_{\lambda}(s)|^{\frac{2-p}{p-1}} \omega_{\lambda}'(s), \quad s > 0.
\end{align}
It follows from \eqref{cos_pq} and \eqref{eq:0711-3} that
\begin{align} \label{eq:0711-1}
\omega_{\lambda}'(s) = -\lambda \cos_{2,2p'}{\bigl((p')^{-\frac{1}{p'}}\lambda^{\frac{1}{p}} s \bigr)}.
\end{align}
Since $|k_{\lambda}(s)|=|\omega_{\lambda}(s)|^{\frac{1}{p-1}}$, we deduce from \eqref{eq:0711-2} that 
\begin{equation}
\label{eq:X'Y'}
X'_{\lambda}(s) 
= \frac{1}{\lambda p'} |\omega_{\lambda}(s)|^{p'}, 
\quad
Y'_{\lambda}(s) = - \frac{1}{\lambda} \omega_{\lambda}'(s). 
\end{equation}
Recalling that $\sin_{q,r}(0)=0$ and $\cos_{q,r}(0)=1$, we deduce from \eqref{eq:0711-3}, \eqref{eq:0711-1} and \eqref{eq:X'Y'} that $X'_\lambda(0)=0$ and $Y'_\lambda(0)=1$. 
Thus \eqref{eq:G-tangent} follows. 

We turn to \eqref{eq:arc-length-para}.
Combining \eqref{eq:X'Y'} with \eqref{eq:0711-3} and \eqref{eq:0711-1}, we see that 
\begin{align*}
X'_{\lambda}(s)^2 + Y'_{\lambda}(s)^2
&= \frac{1}{(\lambda p')^2} \Bigl( (p')^{\frac{1}{p'}} \lambda^{1-\frac{1}{p}} \sin_{2,2p'} \bigl((p')^{-\frac{1}{p'}}\lambda^{\frac{1}{p}} s \bigr) \Bigr)^{2p'} \\
&\qquad+ \frac{1}{\lambda^2} \Bigl( \lambda \cos_{2,2p'}{\bigl((p')^{-\frac{1}{p'}}\lambda^{\frac{1}{p}} s \bigr)} \Bigr)^2 \\
&=\sin_{2,2p'}^{2p'}{\bigl((p')^{-\frac{1}{p'}}\lambda^{\frac{1}{p}} s \bigr)} 
+\cos_{2,2p'}^2{\bigl((p')^{-\frac{1}{p'}}\lambda^{\frac{1}{p}} s \bigr)}
=1,
\end{align*}
where we used \eqref{eq:0711-4} in the last equality.
Thus \eqref{eq:arc-length-para} follows. 

Finally we prove \eqref{eq:G-curvature=k_l}. 
Since $w_{\lambda}$ is a solution of \eqref{eq:w-equation}, we have 
\[
Y''_{\lambda}(s)= -\frac{1}{\lambda}\omega''_{\lambda}(s)
= \frac{1}{\lambda p'} |\omega_{\lambda}(s)|^{\frac{2}{p-1}}\omega_{\lambda}(s).
\]
Combining this with
\[
X''_{\lambda}(s)= \frac{1}{\lambda}|\omega_{\lambda}(s)|^{\frac{2-p}{p-1}}\omega_{\lambda}(s)\omega'_{\lambda}(s),
\]
we observe from \eqref{eq:0711-4}, \eqref{eq:0711-3}, \eqref{eq:0710-6} and \eqref{eq:0711-1} that   
\begin{align*}
& X'_{\lambda}(s)Y''_{\lambda}(s) - X''_{\lambda}(s)Y'_{\lambda}(s) \\
& \qquad =\frac{1}{(\lambda p')^2} |\omega_{\lambda}(s)|^{\frac{p+2}{p-1}}\omega_{\lambda}(s) + \frac{1}{\lambda^2} |\omega_{\lambda}(s)|^{\frac{2-p}{p-1}}\omega_{\lambda}(s) \omega'_{\lambda}(s)^2 \\
& \qquad =|\omega_{\lambda}(s)|^{\frac{2-p}{p-1}}\omega_{\lambda}(s) \Bigl[ \frac{1}{(\lambda p')^2}|\omega_{\lambda}(s)|^{2p'} + \frac{1}{\lambda^2} \omega'_{\lambda}(s)^2 \Bigr] \\
& \qquad =|\omega_{\lambda}(s)|^{\frac{2-p}{p-1}}\omega_{\lambda}(s)
= k_\lambda(s). 
\end{align*}
Therefore Lemma \ref{lemma:pre-G} follows. 
\end{proof}
\begin{definition}
For each $\lambda>0$, we define $\Gamma_\lambda : [0, 2 L_\lambda] \to \mathbb{R}^2$ by 
\begin{equation}
\label{eq:def-Gm}
\Gamma_\lambda(s):=(X_\lambda(s), Y_\lambda(s)),
\quad s\in [0, L_{\lambda}]
\end{equation}
where $X_\lambda$ and $Y_\lambda$ are defined by \eqref{eq:0711-2} and  
\[
L_{\lambda}:=\frac{1}{2} (p')^{\frac{1}{p'}} \lambda^{-\frac{1}{p} } \pi_{2, 2p'}.
\]
\end{definition}
\begin{remark}
The planar curve $\Gamma_\lambda : [0, 2 L_\lambda] \to \mathbb{R}^2$ is symmetric with respect to $s=L_{\lambda}$.
More precisely, it follows that 
\begin{align} \label{eq:0718-4}
X_{\lambda}(2L_{\lambda} -s) = -X_{\lambda}(s) + X_{\lambda} (2L_{\lambda}), 
\;
Y_{\lambda}(2L_{\lambda} -s) = Y_{\lambda}(s), \; s\in[0,L_{\lambda}].
\end{align}
Although it is straightforward to obtain \eqref{eq:0718-4}, we give a rigorous derivation in Appendix~\ref{appendix6}.
\end{remark} 
On properties of $\Gamma_\lambda$, we have: 
\begin{lemma} \label{scaling}
Let $\lambda>0$. 
For $\Gamma_{\lambda}:[0,2L_{\lambda}] \to \mathbb{R}^2$, the following hold$\colon$  
\begin{enumerate}
\item[{\rm (i)}]  $\Gamma_{\lambda}|_{(0, 2L_{\lambda})}$ is the graph of a function$;$ 
\item[{\rm (ii)}] $\Gamma_{\lambda}'(0)=(0,1)$, $\Gamma'_\lambda(L_\lambda)=(1,0)$, $\Gamma_{\lambda}'(2L_{\lambda})=(0,-1);$ 
\item[{\rm (iii)}] $\Gamma_{\lambda}(s)=\lambda^{-\frac{1}{p}}\Gamma_1(\lambda^{\frac{1}{p}} s)$ for all $\lambda > 0$ and $s \in [0, 2 L_\lambda];$ 
\item[{\rm (iv)}] The curvature of $\Gamma_\lambda(s)$ is given by $k_\lambda(s)$. Moreover, the curvature $k_\lambda(s)$ satisfies 
$k_{\lambda}(0)=0$, $k_{\lambda}(s)<0$ for all $s \in (0, 2L_\lambda)$, and $k_{\lambda}'(s)<0$ for all $s \in (0, L_{\lambda})$;
\item[{\rm (v)}] Let $\theta_{\lambda}:[0,2L_{\lambda}]\to \R$ be the tangential angle of $\Gamma_{\lambda}$, i.e., 
\[
\Gamma'_{\lambda}(s)= 
\begin{pmatrix}
\cos \theta_{\lambda}(s) \\
\sin \theta_{\lambda}(s)
\end{pmatrix}.
\]
Then $\theta_{\lambda}$ is strictly decreasing from $\pi/2$ to $-\pi/2$.
\end{enumerate}
\end{lemma}
\begin{proof}
The assertion (i) follows from the fact that $X_{\lambda}'(s) >0$ in $(0,2L_{\lambda})$ by definition of $\sin_{2,2p'}$.  
We also deduce from  
\begin{align*}
\sin_{2,2p'}(0)& =0, \quad \sin_{2,2p'}(\pi_{2, 2p'}/2)=1, \quad \sin_{2,2p'}(\pi_{2, 2p'})=0, \\
\cos_{2,2p'}(0)& =1, \quad \cos_{2,2p'}(\pi_{2, 2p'}/2)=0, \quad \cos_{2,2p'}(\pi_{2, 2p'})=-1, 
\end{align*}
that (ii) holds true. 

We turn to (iii). 
It follows from the definition of $X_\lambda$ and $w_\lambda$ that
\begin{equation}
\label{eq:0712-2}
\begin{aligned}
X_{\lambda}(s) 
&= \frac{1}{\lambda p'}\int_0^{s}  |\omega_{\lambda}(t)|^{p'} \dt
= \int^s_0 \sin_{2,2p'}^{p'}{\bigl((p')^{-\frac{1}{p'}}\lambda^{\frac{1}{p}} t \bigr)}  \dt \\
&=\lambda^{-\frac{1}{p}} \int^{\lambda^{\frac{1}{p}} s}_0 \sin_{2,2p'}^{p'}{\bigl((p')^{-\frac{1}{p'}} \tau \bigr)} \,\mathrm{d}\tau
= \lambda^{-\frac{1}{p}} X_1(\lambda^{\frac{1}{p}} s).
\end{aligned}
\end{equation}
Similarly, we deduce from \eqref{eq:G-tangent}, \eqref{eq:X'Y'}, $\omega_{\lambda}(0)=0$ and \eqref{eq:0711-3}
that 
\begin{align}\label{eq:0712-3}
\begin{aligned}
Y_{\lambda}(s) 
= -\frac{1}{\lambda} \int^s_0 \omega_{\lambda}'(t) \dt 
= - \frac{1}{\lambda} \omega_{\lambda}(s) 
= (p')^\frac{1}{p'} \lambda^{-\frac{1}{p}} \sin_{2,2p'}{\bigl((p')^{-\frac{1}{p'}}\lambda^{\frac{1}{p}} s \bigr)}, 
\end{aligned}
\end{align}
which implies that $Y_1(s)= (p')^\frac{1}{p'} \sin_{2,2p'}{((p')^{-\frac{1}{p'}} s)}$ and 
\[
Y_{\lambda}(s) = \lambda^{-\frac{1}{p}} Y_{1}( \lambda^{\frac{1}{p}}s).
\]
Thus (iii) follows. 

Next we prove (iv). 
It follows from Lemma \ref{lemma:pre-G} that the curvature of $\Gamma_\lambda$ is given by $k_\lambda$. 
Recalling that
\[
\sin_{2,2p'} (\lambda^{\frac{1}{p}} (p')^{-1}s)>0 \quad  \text{in} \quad (0, 2L_{\lambda}),  \quad
\cos_{2,2p'} (\lambda^{\frac{1}{p}} (p')^{-1}s)>0 \quad  \text{in} \quad   (0,L_{\lambda}),
\] 
we infer from \eqref{eq:0711-3} and \eqref{eq:0710-6} that $k_{\lambda}(s)<0$ for all $s \in (0, 2{L_{\lambda}})$
and deduce from \eqref{eq:0712-1} and \eqref{eq:0711-1} that $k_{\lambda}'(s)<0$ for all $s \in (0, L_{\lambda})$. 

Finally, (v) follows immediately from (ii) and (iv) 
since $\theta_{\lambda}(s)=\int_0^{s} k_{\lambda}(t)\dt$.
Therefore Lemma \ref{scaling} follows. 
\end{proof}
\begin{remark}
When $p=2$, \eqref{eq:planar-EL} coincides with the Euler--Lagrange equation for the 
classical elastic energy, and the curve whose curvature satisfies \eqref{eq:planar-EL} is called free elastica.
Therefore, $\Gamma_{\lambda}$ yields a generalization of free elastica since the curvature of $\Gamma_{\lambda}$ satisfies the generalized Euler--Lagrange equation \eqref{eq:planar-EL}. 
Moreover, Lemma~\ref{scaling}-(iii) implies that  $\lambda$ plays a role as a scaling factor and in particular, $\Gamma_1$ determines $\Gamma_{\lambda}$ for any $\lambda>0$. 
\end{remark}
\subsection{Proof of Theorem \ref{Theorem:1.2}}
In order to prove the uniqueness of minimizers of problem \eqref{obstacle_problem:1.1}, 
in what follows we suppose that $\psi$ is a symmetric cone obstacle (see Section~\ref{sec:nonexistence}).
Then thanks to the 
properties of $\Gamma_\lambda$ such as the scale invariance and the monotonicity of the curvature, the strategy developed in \cite{Moist} can work.

First let us introduce the notion of the \emph{polar tangential angle} (e.g. see \cite[Definition 2.1]{Moist}). 
\begin{definition} \label{def:polar-tangential}
Let an arc length parameterized curve $\gamma \in C^2( [0, L] ; \R^2)$ satisfy $\gamma(0)=(0,0)$ and $\gamma(s)\neq(0,0)$ for $s\in(0,L]$. 
For $\gamma$, the \emph{polar tangential angle function} $\varpi : (0, L] \to \mathbb{R}$ is defined as a continuous function such that 
\[
R_{\varpi}\frac{\gamma(s)}{|\gamma(s)|}=\gamma_s(s) \quad \text{for all} \quad s \in (0, L],  
\]
where $R_\theta$ stands for the counterclockwise rotation matrix through angle $\theta \in \mathbb{R}$. 
\end{definition}
From now on,  $\varpi_\lambda : (0, 2L_{\lambda}] \to \R$ denotes the polar tangential angle function for $\Gamma_{\lambda}:[0, 2L_{\lambda}] \to \mathbb{R}^2$, i.e., 
\begin{align}\label{eq:0808-1}
R_{\varpi_{\lambda}} \frac{\Gamma_{\lambda}(s)}{|\Gamma_{\lambda}(s)|} = \Gamma_{\lambda}'(s) \quad \text{for all}\quad  s \in (0,  2L_{\lambda}].
\end{align}
\begin{remark}
Since $k_\lambda \in C([0,2L_\lambda])$, the curve $\Gamma_\lambda$ belongs to  $C^2([0,2L_\lambda]; \R^2)$.
In \cite{Moist}, the polar tangential angle function is defined for curves of class $C^{\infty}$ to ensure the differentiability of $\varpi$.
However, we notice that for each $C^2$-curve one can still obtain $\varpi \in C^1$.
Thus the methods developed in \cite{Moist} are applicable to $C^2$-curves, including $\Gamma_\lambda$.
\end{remark}

\begin{figure}[htbp]
\centering
\includegraphics[width=6cm]{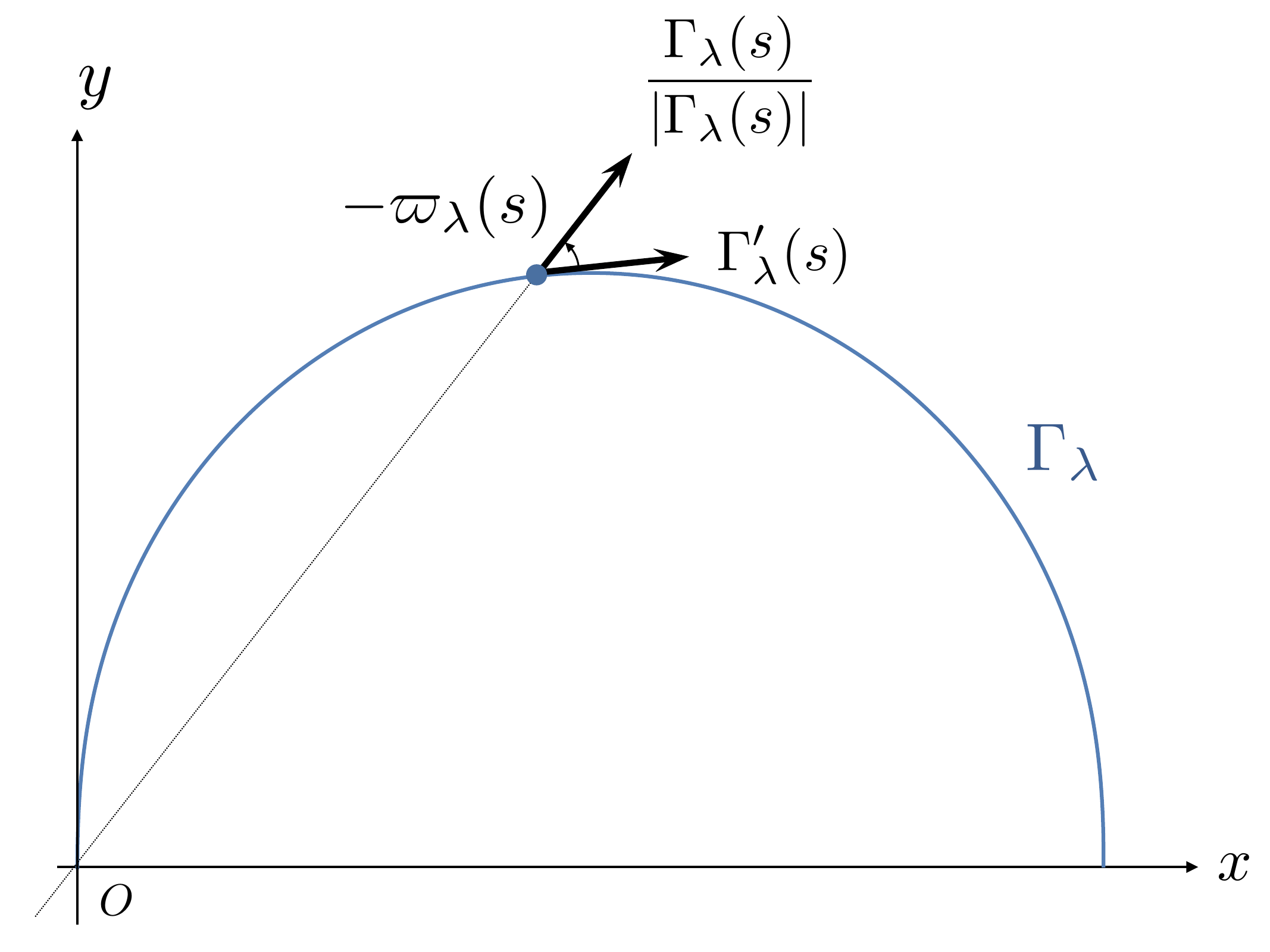} 
\includegraphics[width=6cm]{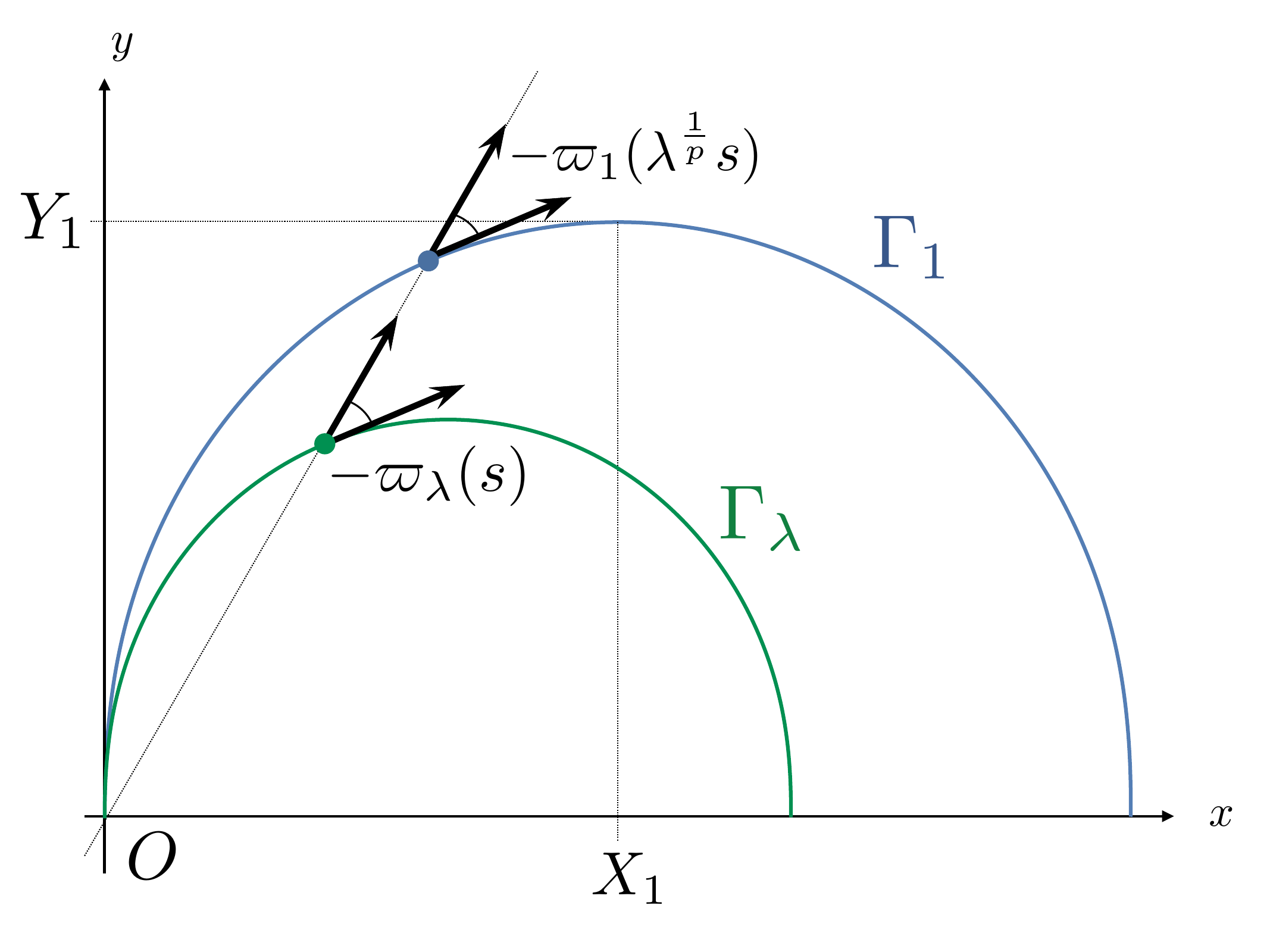}
\caption{Polar tangential angle in Definition~\ref{def:polar-tangential} (left). Scaling behavior of the polar tangential angle, cf. Lemma~\ref{polar_tangent} (right).}
\end{figure}
\begin{lemma}[Monotonicity and scale invariance of $\varpi_\lambda$] \label{polar_tangent}
For each $\lambda>0$, the polar tangential angle $\varpi_{\lambda}$ is strictly monotonically decreasing in $(0, L_\lambda)$ and satisfies the following scaling property$\colon$ 
\begin{equation}\label{eq:0808-2}
\tan\big(\varpi_{\lambda}(s)\big) = \tan\big(\varpi_1(\lambda^{\frac{1}{p}}s)\big) \quad \text{for all} \quad  s \in (0,  L_{\lambda}].
\end{equation}
Moreover, it follows that for each $\lambda>0$
\begin{equation}
\label{eq:211020-1}
\lim_{s \downarrow 0} \tan \varpi_{\lambda}(s) = 0, \quad \tan \varpi_\lambda(L_\lambda) = -\frac{Y_1(L_1)}{X_1(L_1)}, 
\end{equation}
and 
\begin{equation} \label{eq:0907-4}
X_1(L_1)=\frac{1}{2} (p')^{-1+\frac{1}{p'}} \mathrm{B}\bigl( 1-\tfrac{1}{2p}, \tfrac{1}{2} \bigr), \quad
Y_1(L_1)=(p')^{\frac{1}{p'}}.
\end{equation}
\end{lemma}
\begin{proof}
To begin with, we show the monotonicity of $\varpi_\lambda(s)$.
Lemma~\ref{scaling}-(iv) asserts that the curvature $k_\lambda$ of $\Gamma_\lambda$ satisfies $k_{\lambda}(s) k'_{\lambda}(s)>0$ for $s \in (0, L_\lambda)$, 
and hence by applying \cite[Corollary 2.7]{Moist} to $\Gamma_\lambda$, we deduce that $\varpi_{\lambda}$ is monotonically decreasing in $(0, L_\lambda)$.

Next we check \eqref{eq:211020-1}. 
Combining $\Gamma'_\lambda(L_\lambda)=(1,0)$ with the definition of $\omega_\lambda$, we obtain
\[
\tan \varpi_\lambda(L_\lambda) =-\frac{Y_\lambda(L_\lambda)}{X_\lambda(L_\lambda)} =-\frac{Y_1(L_1)}{X_1(L_1)},
\]
where we used Lemma~\ref{scaling}-(iii) in the last equality.
On the other hand, it follows from the definition of $\varpi_\lambda(s)$ that
\begin{equation} \label{eq:0718-5}
\cos \varpi_\lambda(s) = \frac{\Gamma_\lambda(s)}{|\Gamma_\lambda(s)|} \cdot \Gamma'_\lambda(s),
\end{equation}
where $\cdot$ denotes the Euclidean inner product.
By \eqref{eq:G-tangent} we have 
\[
\lim_{s \downarrow 0} \frac{ Y_{\lambda}(s) }{ X_{\lambda}(s) } 
= \lim_{s \downarrow 0} \frac{ Y'_{\lambda}(s) }{ X'_{\lambda}(s) } 
= \infty, 
\]
which implies that
\[
\frac{\Gamma_\lambda(s)}{|\Gamma_\lambda(s)|} 
 = \Bigl( \frac{X_{\lambda}(s)}{ \sqrt{ X_{\lambda}(s)^2 + Y_{\lambda}(s)^2} } , 
  \frac{Y_{\lambda}(s)}{ \sqrt{X_{\lambda}(s)^2 + Y_{\lambda}(s)^2} } \Bigr) 
  \to (0,1)
\]
as $s \downarrow 0$. 
This together with \eqref{eq:0718-5} and Lemma~\ref{scaling}-(ii) gives 
\[ 
\cos \varpi_\lambda(s) \to 1  \quad \text{as} \quad s \downarrow 0, 
\]
i.e., $\tan \varpi_\lambda(s) \to 0$ as $s \downarrow 0$. 

We now turn to prove \eqref{eq:0808-2}. The graph property of $\Gamma_{\lambda}$ ensures that $\tan \varpi_{\lambda}$ is a continuous function on $(0,L_{\lambda}]$. 
By the scaling property Lemma~\ref{scaling}-(iii), we deduce from \eqref{eq:0808-1} that 
\[
\cos{\varpi_\lambda(s)}= \dfrac{\Gamma_\lambda(s)}{|\Gamma_\lambda(s)|}\cdot \Gamma'_\lambda(s) 
 = \dfrac{\Gamma_1(\lambda^{\frac{1}{p}} s)}{|\Gamma_1(\lambda^{\frac{1}{p}} s)|}\cdot \Gamma'_1(\lambda^{\frac{1}{p}} s)
 = \cos{\varpi_1(\lambda^{\frac{1}{p}} s)}
\]
for all $s \in (0, 2 L_\lambda)$. 
By \eqref{eq:211020-1} and the monotonicity of $\varpi_\lambda$, 
$\tan(\varpi_{\lambda}(s))$ and $\tan(\varpi_1(\lambda^{\frac{1}{p}}s))$ are negative for any $s \in (0, L_{\lambda}]$ and hence equal. 
Hence we obtain \eqref{eq:0808-2}. 
It remains to show \eqref{eq:0907-4}.
We infer from \eqref{eq:0712-3} and $L_1=(p')^{\frac{1}{p'}} \pi_{2, 2p'}/2$ that
\[
Y_1( L_1) = (p')^{\frac{1}{p'}} \sin_{2,2p'}{\bigl( \frac{\pi_{2, 2p'}}{2} \bigr)} = (p')^{\frac{1}{p'}}, 
\]
and it follows from \eqref{cos_pq} and \eqref{eq:0711-4}
and  \eqref{eq:0712-2} that
\begin{align*}
X_1(L_1)
&=\int^{L_1}_0  \sin_{2,2p'}^{p'} \bigl( (p')^{-\frac{1}{p'}} t \bigr) \dt 
= (p')^{\frac{1}{p'}} \int^{ \frac{\pi_{2, 2p'}}{2}}_0  \sin_{2,2p'}^{p'}  \sigma \,\mathrm{d}\sigma \\
&=(p')^{\frac{1}{p'}} \int^1_0 \frac{\xi^{p'}}{\sqrt{1-\xi^{2p'}}} \dxi \\
&= \dfrac{1}{2} (p')^{-1+\frac{1}{p'}} \int_0^{1} t^{-\frac{1}{2p}} (1-t)^{-\frac{1}{2}} \dt
=\frac{1}{2} (p')^{-1+\frac{1}{p'}} \mathrm{B}\bigl(1-\tfrac{1}{2p},\tfrac{1}{2}\bigr).
\end{align*}
Therefore Lemma \ref{polar_tangent} follows.
\end{proof}

Since Lemma~\ref{polar_tangent} has been shown, 
we are now ready to generalize the uniqueness and the sharp (non)existence result obtained in 
\cite{Moist} for $\mathcal{E}_2$.
We here stress that the application of \cite{Moist} to our problem is not clear when $p\neq2$, due to the degeneracy and the loss of regularity.
However, with the help of Propsition~\ref{prop:reguw} and Proposition~\ref{prop:deg}, we are able to overcome the difficulties.
Consequently, similar to \cite[Eq. 3.14]{Moist}, we have the following 
\begin{lemma} \label{lem:0827-1}
Let $h>0$ be given.
Suppose that $u \in W^{3,q}(0, 1)$ for some $q>1$ satisfies for some $x_0\in(0,1)$
\begin{align} \label{eq:0827-1}
\begin{cases}
 \displaystyle{
 \int_0^{x_0} \Bigl[ p  \frac{ |\kappa_u|^{p-2}\kappa_u}{1+(u')^2} \phi''+ (1-3p) \frac{|\kappa_u |^{p} u'}{\sqrt{1+(u')^2}} \phi' \Bigr] \dx =0
} 
\quad \text{for all} \quad \phi \in C^{\infty}_0(0, x_0), \\
 u(0)=0, \quad u''(0)=0, \quad  u(x_0)=h, \quad u'(x_0)=0,
\end{cases}
\end{align}
and 
\begin{equation}
\label{eq:211025-1}
u''(x)<0 \quad \text{for all} \quad x \in (0, x_0).
\end{equation}
Then there is $\lambda_u>0$ such that the arc length parameterization $\gamma_u:[0,\mathbf{s}(x_0)]\to\R^2$ of $\mathrm{graph}(u)$ is given by
\begin{align}\label{eq:planar_curve}
\gamma_u(\mathbf{s}) = R_{-\tfrac{\pi}{2}+\theta_{u}} \Gamma_{\lambda_{u}}(\mathbf{s}), \quad 
\mathbf{s} \in [0, \mathbf{s}(x_0)],
\end{align}
where $\theta_u:=\arctan u'(0) \in (0, \pi/2)$.
Furthermore, for $\mathbf{s}(x_0)$ it follows that 
\begin{equation}
\label{eq:0827-2}
\tan{\bigl( -\varpi_{\lambda_u}(\mathbf{s}(x_0)) \bigr)} = \frac{ h}{x_0}
\end{equation}
and 
\begin{align}\label{eq:ubound-length}
\lambda_u^{1/p}\mathbf{s}(x_0) < L_1.
\end{align}
\end{lemma}
\begin{proof}
Let $u \in W^{3,q}(0, 1)$ satisfy \eqref{eq:0827-1}-\eqref{eq:211025-1} with some $q>1$.
Since $u'(x)$ is monotonically decreasing 
on $[0, x_0]$ and $u'(x_0)=0$, we see that $u'(0) > 0$. 
Thus $\theta_u:=\arctan u'(0)$ belongs to $(0, \tfrac{\pi}{2})$. 
Let $\mathbf{s}$ be the arc length parameter of $(x, u(x))$, i.e., $\mathbf{s}=\mathbf{s}(x)=\int^x_0 \sqrt{1+u'(y)^2} \dy$, 
and $\gamma_u(\mathbf{s})$ denotes the arc length parameterization of $(x, u(x))$. 
Let $\tilde{\kappa}_u(\mathbf{s})$ be the curvature of $\gamma_u(\mathbf{s})$. 
It follows from \eqref{eq:0827-1} and Lemma~\ref{p-EL} that $\tilde{\kappa}_u(\mathbf{s})$ satisfies 
\[
\int^{\mathbf{s}(x_0)}_0  \bigl[ p |\tilde{\kappa}_u|^{p-2} \tilde{\kappa}_u \varphi'' + (p-1) |\tilde{\kappa}_u|^{p} \tilde{\kappa}_u \varphi \bigr]\, \mathrm{d}\mathbf{s} =0
\quad \text{for all}\quad  \varphi \in C^{\infty}_0(0, \mathbf{s}(x_0)). 
\]
Then setting 
\[
\xi_u(\mathbf{s}):= |\tilde{\kappa}_u(\mathbf{s})|^{p-2} \tilde{\kappa}_u(\mathbf{s}), \quad \mathbf{s}\in [0, \mathbf{s}(x_0)],
\]
we infer from Lemma~\ref{lem:regularity-omega} that  $\xi_u \in C^2([0,\mathbf{s}(x_0)])$ and $\xi_u$ satisfies \eqref{eq:0606-2}.
Moreover, \eqref{eq:0827-1} 
implies that $\xi_u (0) = 0$.
By uniqueness of solutions for the initial value problem \eqref{eq:0606-3} and the nondegeneracy of $u''$ (see Proposition~\ref{prop:deg}), it turns out that $\xi_u = \omega_{\lambda_u}$ with 
$$\lambda_u:=-\xi_u'(0).$$
Notice that $\lambda_u>0$ follows from $\xi_u (0) = 0$ and $\xi_u(s)>0$ for $s>0$ and near $0$ by \eqref{eq:211025-1}.
Recalling \eqref{eq:0710-6}, we obtain
\begin{equation}
\label{eq:211024-2}
\tilde{\kappa}_u(\mathbf{s})
=|\xi_u(\mathbf{s})|^{\frac{2-p}{p-1}} \xi_u(\mathbf{s}) 
=|\omega_{\lambda_u}(\mathbf{s})|^{\frac{2-p}{p-1}}\omega_{\lambda_u}(\mathbf{s})
=k_{\lambda_u}(\mathbf{s}). 
\end{equation}
Furthermore, combining \eqref{eq:211024-2} with the fact that both curves $\gamma_u$ and $R_{-\tfrac{\pi}{2}+\theta_u} \Gamma_{\lambda_u}$ coincide up to first order at the endpoint $\mathbf{s}=0$, we also obtain
\begin{align*} 
\gamma_u (\mathbf{s}) 
= R_{-\tfrac{\pi}{2}+\theta_u} \Gamma_{\lambda_u}(\mathbf{s})
=\lambda_u^{-{1}/{p}} R_{-\tfrac{\pi}{2}+\theta_u} \Gamma_1(\lambda_u^{{1}/{p}}\mathbf{s}), \quad 
\mathbf{s} \in [0, \mathbf{s}(x_0)].
\end{align*}
Therefore, \eqref{eq:planar_curve} follows.

Next we show \eqref{eq:ubound-length}.
The condition $u'(x_0)=0$ and $u(x_0)=h$ clearly imply  
\begin{equation}
\label{eq:0901-1}
\gamma'_u(\mathbf{s}(x_0))=(1,0) \quad \text{and} \quad \gamma_u(\mathbf{s}(x_0))=(x_0, h)
\end{equation}
respectively.
Substituting $\mathbf{s}= \mathbf{s}(x_0)$ into \eqref{eq:planar_curve} after differentiating with respect to $\mathbf{s}$, we infer from \eqref{eq:0901-1} that
\[
\Gamma'_1 \big(\lambda_u^{{1}/{p}}\mathbf{s}(x_0) \big)
= R_{\tfrac{\pi}{2}-\theta_u}
\begin{pmatrix}
1 \\
0
\end{pmatrix}.
\]
This together with $\tfrac{\pi}{2}-\theta_u\in(0,\pi/2)$ and Lemma~\ref{scaling}-(v) gives 
$0<\lambda_u^{{1}/{p}}\mathbf{s}(x_0) < L_1$.
Thus we obtain \eqref{eq:ubound-length}.
It remains to show \eqref{eq:0827-2}.
We observe from \eqref{eq:planar_curve} and \eqref{eq:0901-1} that
\[
\Gamma_{\lambda_u}(\mathbf{s}(x_0))
= R_{\tfrac{\pi}{2}-\theta_u}
\begin{pmatrix}
x_0 \\
h
\end{pmatrix}, 
\quad 
\Gamma_{\lambda_u}'(\mathbf{s}(x_0))
= R_{\tfrac{\pi}{2}-\theta_u}
\begin{pmatrix}
1\\
0
\end{pmatrix},
\]
from which it follows immediately that the polar tangential angle function $\varpi_{\lambda_u}(\mathbf{s})$ of $\Gamma_{\lambda_u}(\mathbf{s})$ satisfies \eqref{eq:0827-2}.
The claim follows.
\end{proof}
We are now in a position to prove Theorem~\ref{Theorem:1.2}. 
To this end, by way of \eqref{eq:uc} we define a function $u_c:[0,1] \to \R$ by 
\begin{equation}
u_c(x) :=  \frac{1}{c} \int^{EU_p^{-1}(\frac{c}{2})}_{EU_p^{-1}(\frac{c}{2}-c x)} \frac{t}{(1+t^2)^{\frac{3}{2}-\frac{1}{2p}}}\dt  \label{eq:211024-10}
\end{equation}
for $c \in (0, c_p)$, where $c_p$ is defined by \eqref{def_cp_sec6}.  
Similarly to \eqref{def:U_0^G} we also define 
\begin{equation} \label{def-U_0-sect6}
U_0(x):= \lim_{c \uparrow c_p} u_c(x).  
\end{equation}
We note that $u_c$ uniformly converges to $U_0$ in $[0,1]$ as $c \uparrow c_p$. 
Moreover, since 
\begin{align*}
u_c(x)=\frac{p'}{c}
\Biggl[
\frac{1}{\big( 1+G^{-1}(\tfrac{c}{2}-cx)^2\big)^{\frac{1}{2} -\frac{1}{2p} }}
-\frac{1}{\big( 1+G^{-1}(\tfrac{c}{2})^2\big)^{\frac{1}{2} -\frac{1}{2p} }}
\Biggr],  
\end{align*}
we find 
\begin{align}\label{eq:U_0}
U_0(x)=
\begin{cases}
\dfrac{p'}{c_p} \dfrac{1}{\bigl( 1+G^{-1}(\tfrac{c_p}{2}-c_px)^2\bigr)^{\frac{1}{2} -\frac{1}{2p} }} \quad & \text{if} \quad x \in (0,1), \\
0 & \text{if} \quad x=0,1.
\end{cases}
\end{align}
Set
\begin{equation}\label{eq:h_*}
h_*:=U_0\bigl(\tfrac{1}{2} \bigr)=\dfrac{p'}{c_p}.
\end{equation}
Noting that
\[
c_p=2\int_0^{\infty} \frac{1}{(1+t^2)^{\frac{3}{2} - \frac{1}{2p} } } \dt
=\int_0^{\infty} \frac{1}{(1+ \tau)^{\frac{3}{2} - \frac{1}{2p} } } \frac{1}{\sqrt{\tau}} \,\mathrm{d}\tau = \mathrm{B}\bigl(\tfrac{1}{2}, 1-\tfrac{1}{2p} \bigr),
\]
we have
\begin{equation}\label{value:h_*}
h_*= \dfrac{p'}{\mathrm{B}(\tfrac{1}{2}, 1-\tfrac{1}{2p})}.
\end{equation}

\begin{proof}[Proof of Theorem~\ref{Theorem:1.2}]
We divide the proof into 3 steps. 

\textbf{Step 1.} {\sl We show the existence of minimizers.} 
Suppose that $\psi(\frac{1}{2})<h_*$. 
Then we have $\psi(x) < U_0(x)$ for all $x \in [0,1]$. 
Since $u_c$, defined by \eqref{eq:211024-10}, uniformly converges to $U_0$ as $c \uparrow c_p$,
there is $c_* \in (0, c_p)$ such that $u_{c_*} \geq \psi$ in $[0,1]$. 
Moreover, in view of \eqref{eq:energyuc} we notice that
\[
u_{c_*} \in M_{\rm sym}(\psi), \quad \mathcal{E}_p(u_{c_*})=({c_*})^p <(c_p)^p. 
\]
Thus we observe from Theorem~\ref{thm:sym-existence} that there exists 
a minimizer $u$ of $\mathcal{E}_p$ in $M_{\rm sym}(\psi)$. 
\textbf{Step 2.} {\sl We prove the uniqueness of minimizers.}
Let $u_1$, $u_2 \in M_{\rm sym}(\psi)$ be minimizers of $\mathcal{E}_p$ in $M_{\rm sym}(\psi)$. 
As we mentioned in Section \ref{section:6.1}, $u_1$ and $u_2$ are also critical points of $\mathcal{E}_p$ in $M(\psi)$ in the sense of Definition \ref{def:critical}. 
Thus Remark~\ref{rem:symcone} and Propositions~\ref{prop:deg}, \ref{prop:regu1} and \ref{prop:regu2} allow us 
to apply Lemma~\ref{lem:0827-1} for $u_1$ and $u_2$. 
From Lemma~\ref{lem:0827-1}, we find $\lambda_1, \lambda_2 >0$ such that for $\theta_{u_i}:= \arctan u_i'(0) \in (0, \pi/2)$
\begin{align}\label{eq:1210-4}
\begin{split}
\gamma_{u_1}(\mathbf{s}) 
&= R_{-\tfrac{\pi}{2}+\theta_{u_1}} \Gamma_{\lambda_{1}}(\mathbf{s})
=\lambda_1^{-\frac{1}{p}} R_{-\tfrac{\pi}{2}+\theta_{u_1}} \Gamma_1(\lambda_1^{\frac{1}{p}} \mathbf{s}), \quad 
\mathbf{s} \in [0, \mathbf{s}_{1}(\tfrac{1}{2})], \\
\gamma_{u_2}(\mathbf{s}) 
&= R_{-\tfrac{\pi}{2}+\theta_{u_2}} \Gamma_{\lambda_2}(\mathbf{s})
=\lambda_2^{-\frac{1}{p}}R_{-\tfrac{\pi}{2}+\theta_{u_2}} \Gamma_1(\lambda_2^{\frac{1}{p}} \mathbf{s}) , \quad 
\mathbf{s} \in [0, \mathbf{s}_{2}(\tfrac{1}{2})],
\end{split}
\end{align}
where $\mathbf{s}_{i}$ and $\gamma_{u_i}(\mathbf{s})$ denote the arc length parameter of $(x, u_i(x))$ and the arc length parameterization of $(x, u_i(x))$, respectively ($i=1,2$).
Moreover, $u_i'(\frac{1}{2})=0$ and $u_i(\frac{1}{2})=\psi(\tfrac{1}{2})$ yield   
\begin{align}
\gamma'_{u_1}(\mathbf{s}_1(\tfrac{1}{2}))&=\gamma'_{u_2}(\mathbf{s}_2(\tfrac{1}{2}))=(1,0), \label{eq:1210-2} \\ 
\gamma_{u_1}(\mathbf{s}_1(\tfrac{1}{2}))&=\gamma_{u_2}(\mathbf{s}_2(\tfrac{1}{2}))=(\tfrac{1}{2}, \psi(\tfrac{1}{2})), \label{eq:1210-3}
\end{align}
respectively.
By \eqref{eq:0827-2}, the polar tangential angle function $\varpi_{\lambda_i}$ of $\Gamma_{\lambda_i}$ satisfies
\begin{align}
 -\varpi_{\lambda_1}(\mathbf{s}_1(\tfrac{1}{2}))
=-\varpi_{\lambda_2}(\mathbf{s}_2(\tfrac{1}{2})) 
= \arctan\big(2 \psi(\tfrac{1}{2})\big).
\end{align}
This together with \eqref{eq:0808-2} implies that 
$
-\varpi_1(\lambda_1^{1/p}\mathbf{s}_1(\tfrac{1}{2}))
=-\varpi_1(\lambda_2^{1/p}\mathbf{s}_2(\tfrac{1}{2}))
$, which in combination with the monotonicity of the polar tangential angle (see Lemma~\ref{polar_tangent}) gives 
\begin{align}\label{eq:1210-5}
\lambda_1^{1/p}\mathbf{s}_1(\tfrac{1}{2})=\lambda_2^{1/p}\mathbf{s}_2(\tfrac{1}{2}).
\end{align}
Noting that \eqref{eq:1210-4} and \eqref{eq:1210-2} give
\begin{align*}
R_{-\tfrac{\pi}{2}+\theta_{u_1}} \Gamma_1'(\lambda_1^{{1}/{p}} \mathbf{s}_1(\tfrac{1}{2}))
=\gamma_{u_1}'(\mathbf{s}_1(\tfrac{1}{2})) 
=\gamma_{u_2}'(\mathbf{s}_2(\tfrac{1}{2})) 
=R_{-\tfrac{\pi}{2}+\theta_{u_2}} \Gamma_1'(\lambda_2^{{1}/{p}} \mathbf{s}_2(\tfrac{1}{2})),
\end{align*}
we infer from \eqref{eq:1210-5} that $\theta_{u_1}=\theta_{u_2}$.
Furthermore, combining $\theta_{u_1}=\theta_{u_2}$ with \eqref{eq:1210-4}, \eqref{eq:1210-3} and \eqref{eq:1210-5}, we obtain $\lambda_1=\lambda_2$.
In view of \eqref{eq:1210-4}, $\theta_{u_1}=\theta_{u_2}$ and $\lambda_1=\lambda_2$ imply $\gamma_{u_1}=\gamma_{u_2}$, so that  $u_1=u_2$.
Therefore, we obtain the uniqueness of minimizers.
\smallskip

\textbf{Step 3.} 
{\sl Finally we prove the nonexistence of minimizers of $\mathcal{E}_p$ in $M(\psi)$ with $\psi$ satisfying $\psi(\frac{1}{2}) \geq h_*$.} 
Suppose that there is a minimizer $u \in M(\psi)$ of $\mathcal{E}_p$ in $M(\psi)$. 
Since $u : [0, 1] \to \mathbb{R}$ is continuous and $u(0)=u(1)=0$, $u$ attains its maximum $u_{\rm max}>0$ at $x_{\rm max} \in (0,1)$. 
Without loss of generality we may assume that $x_{\rm max} \in (0, \tfrac{1}{2}]$. 
Since $\psi$ is a symmetric cone, we infer from Remark~\ref{rem:symcone} that $u>\psi$ in $(0, x_{\rm max})$. 
Moreover, recalling that $u$ is also a critical point of $\mathcal{E}_p$ in $M(\psi)$ in the sense of Definition \ref{def:critical}, we see that $u$ satisfies 
\begin{align*} 
\begin{cases}
 \displaystyle{\int^{x_{\rm max}}_0} \Bigl[ p  \dfrac{ |\kappa_{u}|^{p-2}\kappa_{u}}{1+(u')^2} \phi'' 
 + (1-3p) \dfrac{|\kappa_{u} |^{p} u'}{\sqrt{1+(u')^2}} \phi' \Bigr] \dx =0 
\ \  \text{for all} \ \,  \phi \in C^{\infty}_0(0, x_{\rm max}), \\
 u(0)=0, \quad u''(0)=0, \quad  u(x_{\rm max})= u_{\rm max}, \quad u'(x_{\rm max})=0,
\end{cases}
\end{align*}
and $u''(x) < 0$ for all $x \in (0, x_{\rm max})$. 
Thus, using Lemma~\ref{lem:0827-1} with $x_0=x_{\rm max}$, we obtain
\begin{align}\label{eq:falseheight}
2 h_* \leq 2\psi(\tfrac{1}{2}) 
\leq \frac{\psi(\tfrac{1}{2})}{x_{\rm max}} \leq  \frac{u(x_{\rm max})}{x_{\rm max}} 
= \tan \bigl(- \varpi_{\lambda_u}(\mathbf{s}(x_{\rm max})) \bigr). 
\end{align}
Here we used that $u(x_{\rm max}) \geq u(\frac{1}{2}) \geq \psi(\frac{1}{2})$.
On the other hand, by Lemma~\ref{polar_tangent} and \eqref{eq:ubound-length}, 
for the right-hand side of the above we see that 
\begin{align*}
\tan \big( -\varpi_{\lambda_u}(\mathbf{s}(x_{\rm max})) \big) 
&= \tan \big( -\varpi_{1}(\lambda_u^{{1}/{p}}\mathbf{s}(x_{\rm max})) \big) \\
&> \tan \big( -\varpi_{1}( L_1 ) \big) = \frac{Y_1(L_1)}{X_1(L_1)} = 2 h_*.
\end{align*}
This is a contradiction. 
Therefore Theorem \ref{Theorem:1.2} follows.
\end{proof}
\begin{figure}[h!]
    \centering
    \includegraphics[scale=0.6]{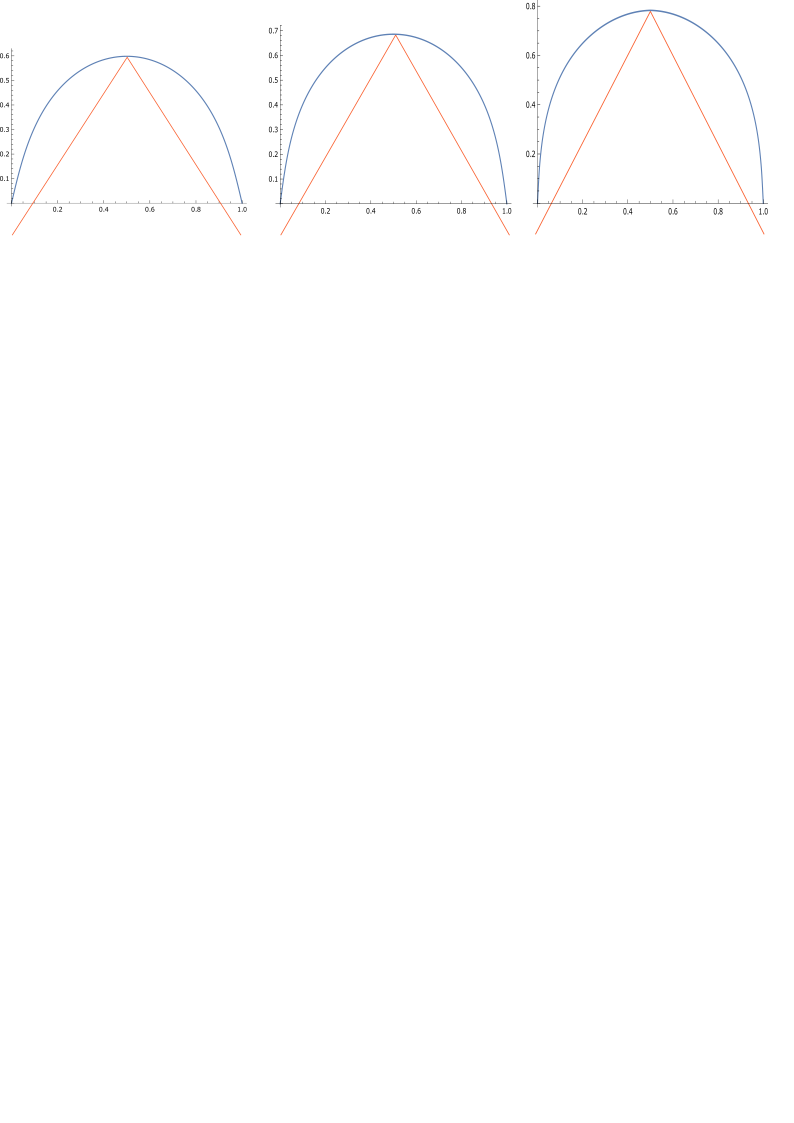}
    \vspace{-14.8cm}
    \caption{Cone obstacles and their unique minimizers, $p=2$.}
    \label{fig:2}
\end{figure}
\begin{figure}[h!]
    \centering
    \includegraphics[scale=0.6]{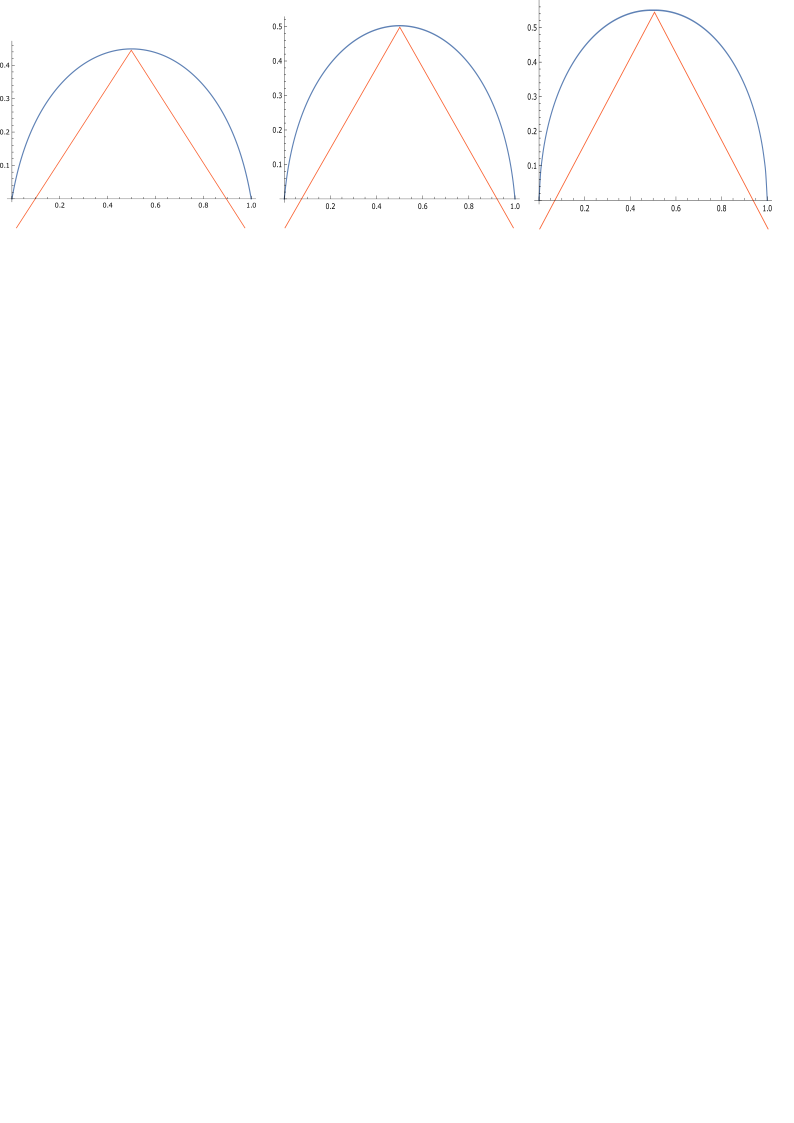}
    \vspace{-14.8cm}
    \caption{Cone obstacles and their unique minimizers, $p=5$.}
    \label{fig:3}
\end{figure}
\begin{remark}
Notice that in the case of $p =2$, \cite[Remark 2.16]{GrunauOkabe} suggests that a sharp nonexistence result in $M_{\rm sym}(\psi)$ can also be obtained for \emph{general symmetric obstacles $\psi$}, and not only for cone obstacles. It could be subject of future research to study generalizations of such result to $p \neq 2$.
\end{remark}

\appendix

\section{Technical proofs}\label{ref:app}
\subsection{Technical proofs in Section \ref{sec:nonexistence}}
\begin{proof}[Proof of Lemma \ref{lem:6.2}]
For the continuity on $(0,\infty)$ we make the substitution $z := \frac{s}{A}$ and infer that
\begin{equation}
    H(A) = A \frac{\int_0^1 \frac{z\dot{G}(Az)}{(1-z)^\frac{1}{p}}  \dz }{ \int_0^1 \frac{\dot{G}(Az)}{(1-z)^\frac{1}{p}} \dz }.
\end{equation}
Since $z \mapsto \frac{1}{(1-z)^\frac{1}{p}}$ is integrable on $(0,1)$ and $A \mapsto \dot{G}(Az)$ is uniformly bounded and depends continuously on $A$ one easily infers from the dominated convergence theorem that $H$ is continuous. For the asymptotics as $A \rightarrow 0$ one can estimate 
\begin{equation}
    0 \leq H(A) = \frac{\int_0^A \frac{s\dot{G}(s)}{(A-s)^\frac{1}{p}}  \ds }{ \int_0^A \frac{\dot{G}(s)}{(A-s)^\frac{1}{p}} \ds } \leq \frac{\int_0^A \frac{A\dot{G}(s)}{(A-s)^\frac{1}{p}}  \ds }{ \int_0^A \frac{\dot{G}(s)}{(A-s)^\frac{1}{p}} \ds } = A, 
\end{equation}
whereupon the claim follows. For the asymptotics as $A \rightarrow \infty$ we define 
\begin{equation}
    I(A) := \int_0^A \frac{\dot{G}(s)}{(A-s)^\frac{1}{p}} \ds, \quad J(A) := \int_0^A \frac{s\dot{G}(s)}{(A-s)^\frac{1}{p}}  \ds. 
\end{equation}
Let $\epsilon > 0$ be as in the statement and choose $\alpha \in (0, 1)$ such that $\frac{1}{\alpha} = 1+ \epsilon$.  Now 
\begin{equation}
  A^\frac{1}{p}  J(A) = \int_{0}^{A}  \frac{s\dot{G}(s)}{(1 - s/A)^\frac{1}{p}} \ds =   \int_{0}^{A^\alpha}  \frac{s\dot{G}(s)}{(1 - s/A)^\frac{1}{p}} \ds +  \int_{A^\alpha}^{A}  \frac{s\dot{G}(s)}{(1 - s/A)^\frac{1}{p}} \ds .
\end{equation}
If $\chi_E$ denotes the characteristic function of a set $E$ the first integral can be studied as follows
\begin{equation}
   \int_0^{A^\alpha} \frac{s\dot{G}(s)}{(1-s/A)^\frac{1}{p}} \; \mathrm{d}s =  \int_0^\infty \frac{\chi_{(0,A^\alpha)}}{(1-s/A)^\frac{1}{p}} s \dot{G}(s) \ds \rightarrow \int_0^\infty s \dot{G}(s) \; \mathrm{d}s \quad (A \rightarrow \infty),
\end{equation}
since $s \dot{G}(s)$ is integrable (as it lies in $o(s^{-1-\epsilon})$) and $\frac{\chi_{(0,A^\alpha)}}{(1- s/A)^\frac{1}{p}}$ is bounded from above by $\frac{1}{(1- A^{\alpha-1})^\frac{1}{p}} \chi_{(0,\infty)} \leq 2\chi_{(0,\infty)} $  (for large enough $A$) and converges pointwise to $\chi_{(0,\infty)}$. The second integral can be estimated by 
\begin{align}
    \int_{A^\alpha}^A \frac{s\dot{G}(s)}{(1-s/A)^\frac{1}{p}} \ds & \leq \left[ \sup_{x > A^\alpha } x \dot{G}(x) \right]  \int_{A^\alpha}^A \frac{1}{(1-s/A)^\frac{1}{p}} \ds  \\ & = \left[ \sup_{x > A^\alpha} x \dot{G}(x) \right]  A \left( 1 - A^{\alpha- 1} \right)^{1-\frac{1}{p}} \frac{p}{p-1}
    \leq \frac{p}{p-1} A \left[ \sup_{x > A^\alpha } x \dot{G}(x) \right].
\end{align}
Rewriting $A = (A^\alpha )^\frac{1}{\alpha}$ we obtain 
\begin{align}
   \int_{A^\alpha}^A \frac{s\dot{G}(s)}{(1-s/A)^\frac{1}{p}} \ds & \leq \frac{p}{p-1}  \left[ \sup_{x > A^\alpha } (A^\alpha)^\frac{1}{\alpha} x \dot{G}(x) \right] \leq \frac{p}{p-1}  \left[ \sup_{x > A^\alpha } x^\frac{1}{\alpha} x \dot{G}(x) \right] \\ & = \frac{p}{p-1} \left[ \sup_{x > A^\alpha } x^{1+\frac{1}{\alpha}} \dot{G}(x) \right] = \frac{p}{p-1} \left[ \sup_{x > A^\alpha } x^{2+\epsilon} \dot{G}(x) \right]. 
\end{align}
By \eqref{eq:asymptG} this tends to zero as $A \rightarrow \infty$. We have proved that 
\begin{equation}
 \lim_{A \rightarrow \infty}    A^\frac{1}{p}J(A) = \int_0^\infty s \dot{G}(s) \ds. 
\end{equation}
With similar techniques (and even less critical growth estimates) one readily checks that 
\begin{equation}
    \lim_{A \rightarrow \infty}    A^\frac{1}{p}I(A) = \int_0^\infty  \dot{G}(s) \ds = \frac{1}{2} c_p(G) .
\end{equation}
Now notice that 
\begin{equation}
    H(A) = \frac{I(A)}{J(A)} = \frac{A^{\frac{1}{p}} I(A)}{A^{\frac{1}{p}}  J(A)} \rightarrow 2\frac{\int_0^\infty s \dot{G}(s) \ds }{  c_p(G)}, 
\end{equation}
whereupon the claim follows. Boundedness of $H$ follows from the continuity, the explained asymptotics and the fact that $(s \mapsto s \dot{G}(s)) \in L^1(0,\infty)$.
\end{proof}

\subsection{Technical proofs in Section \ref{section:4}} \label{appendix6}
We prove \eqref{eq:sym_ineq} for general functional $\mathcal{E}$ which is defined by \eqref{def-E-intro}. 

\begin{proof}[Proof of \eqref{eq:sym_ineq}]
As stated earlier, it suffices to show \eqref{eq:sym_ineq}.
Fix $v\in M(\psi)$ arbitrarily and set 
\[
v_1(x):=  \frac{1}{2}\big( v (x) +v(1-x) \big), \quad v_2(x):=v(x) - v_1(x), \quad x \in [0,1].
\]
Notice that $v_1\in \Msym(\psi)$ and $v_2$ satisfies $v_2(1-x)=-v_2(x)$.
Then inserting $v=v_1$ into \eqref{eq:pre_sym_ineq}, we infer that 
\begin{align}
0\leq D\mathcal{E}(u)(v_1-u) = D\mathcal{E}(u)(v-u) - D\mathcal{E}(u)(v_2).
\end{align}
Therefore, it is sufficient to show that $D\mathcal{E}(u)(v_2)=0$.
By \eqref{eq:FD}, we see that
\begin{align}\label{eq:DEantisymmetric}
& \qquad D\mathcal{E}(u)(v_2) \\
=&p\int_0^1 \Big( \dot{G}(u'(x))^p|u''(x)|^{p-2} u''(x)v_2''(x) +  \dot{G}(u'(x))^{p-1} \ddot{G}(u'(x)) |u''(x)|^{p} v_2'(x) \Big)\dx. \nonumber
\end{align}
Noting that by the oddness of $G$, $u\in \Msym(\psi)$, and $v_2(1-x)=-v_2(x)$ we have
\begin{align*}
&\dot{G}(z) = \dot{G}(-z), \quad u'(x)=-u'(1-x),  \quad v_2'(x)=v_2'(1-x),  \\
&\ddot{G}(z)=-\ddot{G}(-z), \quad  u''(x)=u''(1-x), \quad  v_2''(x)=-v_2''(1-x),
\end{align*}
respectively. These observations and \eqref{eq:DEantisymmetric} immediately imply that $D\mathcal{E}(u)(v_2)=0$. The proof is complete.
\end{proof}

\begin{proof}[Proof of Lemma~\ref{p-EL}] Since $u\in W^{3, q}(0,1)$ for some $q>1$, we see that $\kappa_u \in C([0,1])$.
Hence, by density, \eqref{eq:0709-2} holds for any $\phi \in W^{2,q}_{0}(E)$.
Fix $ \varphi \in C^{\infty}_0(\mathbf{s}(x_1), \mathbf{s}(x_2))$ arbitrarily and set
\begin{align*}
\phi_1(x):= \sqrt{1+u'(x)^2} \varphi(\mathbf{s}(x)), \quad x \in E.
\end{align*}
Then we notice that $\phi_1 \in W^{2,q}_{0}(E)$ since 
\begin{align*}
\phi_1'(x) &=\bigl( 1+u'(x)^2 \bigr)  \varphi'(\mathbf{s}(x)) + \frac{u'(x) u''(x) }{ \sqrt{1+u'(x)^2}} \varphi(\mathbf{s}(x)), \\
\phi_1''(x)&=  \bigl( 1+u'(x)^2 \bigr)^{\frac{3}{2}} \varphi''(\mathbf{s}(x))
+ 3 u'(x) u''(x) \varphi'(\mathbf{s}(x)) \\  
&\quad+\Bigl( \frac{u'(x) u'''(x) }{ \sqrt{1+u'(x)^2}} + \frac{u''(x)^2 }{(1+u'(x)^2 )^{\frac{3}{2} }}\Bigr) \varphi(\mathbf{s}(x)).
\end{align*}
Taking $\phi=\phi_1$ in \eqref{eq:0709-2} we obtain 
\begin{equation}
\label{eq:211015}
\begin{aligned}
0= &\int^{x_2}_{x_1}  \Bigl[p |\kappa_u|^{p-2} \kappa_u 
\bigl[ \varphi''(\mathbf{s}(x)) + 3 \kappa_u u'  \varphi'(\mathbf{s}(x)) + \kappa_u^2 \varphi(\mathbf{s}(x)) \bigr] \\
& \qquad \qquad \,\,\, + (1-3p) |\kappa_u|^{p} \bigl[ u' \varphi'(\mathbf{s}(x)) + \kappa_u (u')^2 \varphi(\mathbf{s}(x)) \bigr] \Bigr]\sqrt{1+(u')^2} \dx \\
& + \int^{x_2}_{x_1}  p \dfrac{|\kappa_u|^{p-2} \kappa_u u' u'''}{\bigl( 1 + (u')^2 \bigr)^{\frac{3}{2}}}\varphi(\mathbf{s}(x)) \dx \\
=& \int^{x_2}_{x_1} \Bigl[p |\kappa_u|^{p-2} \kappa_u \varphi''(\mathbf{s}(x)) + p |\kappa_u|^{p} \kappa_u \varphi(\mathbf{s}(x)) \\
& \qquad + |\kappa_u|^p u' \varphi'(\mathbf{s}(x)) + (1-3p) |\kappa_u|^{p} \kappa_u (u')^2 \varphi(\mathbf{s}(x)) \Bigr]\sqrt{1+(u')^2} \dx \\
& + \int^{x_2}_{x_1}  p \dfrac{|\kappa_u|^{p-2} \kappa_u u' u'''}{\bigl( 1 + (u')^2 \bigr)^{\frac{3}{2}}}\varphi(\mathbf{s}(x)) \dx, 
\end{aligned}
\end{equation}
where we used $\kappa_u=u''/(1+(u')^2)^{3/2}$. 
Since $u \in W^{3,q}(0,1)$, it follows that $\kappa_u \in W^{1,q}(0,1)$ and
\[ 
u''' = \bigl(1+(u')^2 \bigr)^{\frac{3}{2}} \kappa_u' +3 u' \bigl(1+(u')^2 \bigr)^2 \kappa_u^2.  
\]
From this we deduce that
\begin{align*}
& \int^{x_2}_{x_1}  p \dfrac{|\kappa_u|^{p-2} \kappa_u u' u'''}{\bigl( 1 + (u')^2 \bigr)^{\frac{3}{2}}} \varphi(\mathbf{s}(x)) \dx \\
=&\int^{x_2}_{x_1} p \Bigl[ |\kappa_u|^{p-2} \kappa_u \kappa_u' u'+ 3|\kappa_u|^{p}\kappa_u (u')^2 \sqrt{1+(u')^2} \Bigr] \varphi(\mathbf{s}(x)) \dx,
\end{align*}
and then \eqref{eq:211015} is reduced into 
\begin{equation}
\label{eq:211015-2}
\begin{aligned}
0=& \int^{x_2}_{x_1} \Bigl[p |\kappa_u|^{p-2} \kappa_u \varphi''(\mathbf{s}(x)) + p |\kappa_u|^{p} \kappa_u \varphi(\mathbf{s}(x)) \\
& \qquad + |\kappa_u|^p u' \varphi'(\mathbf{s}(x)) + |\kappa_u|^{p} \kappa_u (u')^2 \varphi(\mathbf{s}(x)) \Bigr]\sqrt{1+(u')^2} \dx \\
& + \int^{x_2}_{x_1} p |\kappa_u|^{p-2} \kappa_u \kappa_u' u' \varphi(\mathbf{s}(x)) \dx. 
\end{aligned}
\end{equation}
Integrating by part we have 
\begin{align*}
&\int^{x_2}_{x_1} p |\kappa_u|^{p-2} \kappa_u \kappa_u' u' \varphi(\mathbf{s}(x)) \dx 
 = \int^{x_2}_{x_1} (|\kappa_u|^{p})' u' \varphi(\mathbf{s}(x)) \dx \\
&= - \int^{x_2}_{x_1} |\kappa_u|^{p} \bigl[  u'' \varphi(\mathbf{s}(x)) + u' \varphi'(\mathbf{s}(x)) \sqrt{1+(u')^2} \bigr] \dx \\
&= - \int^{x_2}_{x_1} |\kappa_u|^{p} \bigl[ \kappa_u (1+ (u')^2) \varphi(\mathbf{s}(x)) + u' \varphi'(\mathbf{s}(x)) \bigr] \sqrt{1+(u')^2} \dx.   
\end{align*}
This together with \eqref{eq:211015-2} implies that 
\begin{equation}
\label{eq:211015-3}
\begin{aligned}
\int^{x_2}_{x_1} \Bigl[p |\kappa_u|^{p-2} \kappa_u \varphi''(\mathbf{s}(x)) + (p-1) |\kappa_u|^{p} \kappa_u \varphi(\mathbf{s}(x)) \Bigr]\sqrt{1+(u')^2} \dx = 0. 
\end{aligned}
\end{equation}
Finally, changing the variable $x=x(\mathbf{s})$ in \eqref{eq:211015-3}, we obtain \eqref{eq:0710-4}. 
\end{proof}

\begin{proof}[Proof of Lemma~\ref{lem:regularity-omega}]
First, note that \eqref{eq:0606-2} holds for any $\varphi \in W^{2,2}_0(0,L)$ by the density argument and $\omega \in C([0,L])$.
Fix $\eta \in C^{\infty}_0(0,L)$ arbitrarily and
set
\[
\varphi_1(s) := \int_0^s \eta(t)\; \mathrm{d}t + \Bigl( \int_0^L \eta(s)\; \mathrm{d}s \Bigr) \Bigl(-\frac{3}{L^2}s^2 + \frac{2}{L^3}s^3 \Bigr), \quad s\in [0,L].
\]
Then $\varphi_1$ satisfies $\varphi_1 \in W^{2,2}_0(0,L)$ and $\|\varphi_1 \|_{C^0} \leq C \| \eta \|_{L^1}$, 
$\|\varphi'_1 \|_{L^1} \leq C \| \eta \|_{L^1}$ for some $C>0$. 
By taking $\varphi_1$ as $\varphi$ in \eqref{eq:0606-2}, we have 
\begin{align*}
\Bigl| p\int_0^L \omega \eta' \; \mathrm{d}s \Bigr|
&\leq p\| \omega\|_{L^{\infty}}\| \eta \|_{L^1} \int_0^L \Big| -\frac{6}{L^2}+\frac{12}{L^3}s \Big| \, \mathrm{d}s 
+(p-1)\| \omega \|_{L^{\infty}}^{\frac{p+1}{p-1}} \int_0^L | \varphi| \, \mathrm{d}s \\
& \leq C \| \eta \|_{L^1},
\end{align*}
from which we can deduce that $\omega \in W^{1, \infty}(0,L) $, and 
then it follows from \eqref{eq:0606-2} that
\begin{align*}
\Bigl| p\int_0^L  \omega' \varphi' \, \mathrm{d}s \Bigr|
&\leq (p-1)\| \omega \|_{L^{\infty}}^{\frac{p+1}{p-1}} \int_0^L | \varphi| \, \mathrm{d}s
\leq C \| \varphi \|_{L^1}
\end{align*}
for any $ \varphi \in C^{\infty}_0(0,L)$.
Therefore, we see that $\omega\in W^{2, \infty}(0,L)$ and 
\[
\omega''(s) = -\frac{(p-1)}{p}|\omega(s)|^{\frac{2}{p-1}}\omega(s)  \quad \text{a.e. } s\in(0,L).
\]
The function on the right hand side belongs to $C([0,L])$ since so does $\omega$, 
and hence we obtain $\omega \in C^2([0,L])$.
In addition, using the integration by parts in \eqref{eq:0606-2}, we obtain \eqref{eq:w-equation}.
\end{proof}

\begin{proof}[Proof of \eqref{eq:0718-4}]
From \eqref{eq:0712-3} and $\sin_{2, 2p'}(\cdot) = \sin_{2, 2p'}(\pi_{2,2p'}-\cdot)$, it follows that for $s\in [0,L_{\lambda}]$
\begin{align*}
Y_{\lambda}(2L_{\lambda} -s) &= (p')^{\frac{1}{p'}} \lambda^{-\frac{1}{p}}\sin_{2,2p'}\bigl(\lambda^{\frac{1}{p}} (p')^{-1}(2L_{\lambda}-s) \bigr) \\
&= (p')^{\frac{1}{p'}} \lambda^{-\frac{1}{p}}\sin_{2,2p'}\bigl(\pi_{2,2p'}-\lambda^{\frac{1}{p}} (p')^{-1}s \bigr) \\
&=(p')^{\frac{1}{p'}} \lambda^{-\frac{1}{p}}\sin_{2,2p'}\bigl(\lambda^{\frac{1}{p}} (p')^{-1}s \bigr)
= Y_{\lambda}(s). 
\end{align*}
On the other hand, using \eqref{eq:0712-2}, $\sin_{2, 2p'}(\cdot) = \sin_{2, 2p'}(\pi_{2,2p'}-\cdot)$, and the change of variables $t=2L_{\lambda}-\tau$,  we have
\begin{align*}
X_{\lambda}(2L_{\lambda} -s) 
&=\int_0^{2L_{\lambda} -s} \sin_{2,2p'}^{p'} \bigl( (p')^{-\frac{1}{p'}}\lambda^{\frac{1}{p}} t \bigr)\dt \\
&= \int_{L_{\lambda}}^{s} \sin_{2,2p'}^{p'} \bigl(\pi_{2,2p'}- (p')^{-\frac{1}{p'}}\lambda^{\frac{1}{p}} \tau)\bigr) (-1) \,\mathrm{d}\tau \\ 
&= - \Bigl(
\int_{0}^{s} \sin_{2,2p'}^{p'} \bigl((p')^{-\frac{1}{p'}}\lambda^{\frac{1}{p}} \tau\bigr)\,\mathrm{d}\tau
-\int_{0}^{2L_{\lambda}} \sin_{2,2p'}^{p'} \bigl((p')^{-\frac{1}{p'}}\lambda^{\frac{1}{p}} \tau\bigr)\,\mathrm{d}\tau
\Bigr) \\ 
& = -X_{\lambda}(s) + X_{\lambda}(2L_{\lambda}) 
\end{align*}
for any $s\in[0,L_{\lambda}]$.
\end{proof}


\begin{thebibliography}{99}


\bib{AGM}{article}{
	author = {Arroyo, Josu J.},
	author = {Garay, Oscar Jesus},
 	author = {Menc\'ia, Jos\'e J.},
	journal = {J. Geom. Phys.},
	number = {2-3},
	pages = {339--353},
	title = {Closed generalized elastic curves in {$S^2(1)$}},
	volume = {48},
	year = {2003}
}

\bib{BHV}{article}{
	author = {Blatt, Simon}, 
	author = {Hopper, Christopher},
 	author = {Vorderobermeier, Nicole},
	journal = {Adv. Nonlinear Anal.},
	number = {1},
	pages = {1383--1411},
	title = {A regularized gradient flow for the {$p$}-elastic energy},
	volume = {11},
	year = {2022}
}


\bib{Anna}{article}{
   author={Dall'Acqua, Anna},
   author={Deckelnick, Klaus},
   title={An obstacle problem for elastic graphs},
   journal={SIAM J. Math. Anal.},
   volume={50},
   date={2018},
   number={1},
   pages={119--137},
   issn={0036-1410},
   review={\MR{3742685}},
   doi={10.1137/17M111701X},
}

\bib{DDG08}{article}{
    AUTHOR = {Dall'Acqua, Anna}, 
    author={Deckelnick, Klaus}, 
    author={Grunau, Hans-Christoph},
     TITLE = {Classical solutions to the {D}irichlet problem for {W}illmore
              surfaces of revolution},
   JOURNAL = {Adv. Calc. Var.},
    VOLUME = {1},
      YEAR = {2008},
    NUMBER = {4},
     PAGES = {379--397},
      ISSN = {1864-8258},
   review={\MR{2480063}},
       DOI = {10.1515/ACV.2008.016},
}

\bib{DalMaso}{article}{
   author={Dal Maso, G.},
   author={Fonseca, I.},
   author={Leoni, G.},
   author={Morini, M.},
   title={A higher order model for image restoration: the one-dimensional
   case},
   journal={SIAM J. Math. Anal.},
   volume={40},
   date={2009},
   number={6},
   pages={2351--2391},
   issn={0036-1410},
   review={\MR{2481298}},
   doi={10.1137/070697823},
}

\bib{Novaga}{article}{
    AUTHOR = {Dayrens, Fran\c{c}ois}, 
    author={Masnou, Simon}, 
    author={Novaga, Matteo},
     TITLE = {Existence, regularity and structure of confined elasticae},
   JOURNAL = {ESAIM Control Optim. Calc. Var.},
    VOLUME = {24},
      YEAR = {2018},
    NUMBER = {1},
     PAGES = {25--43},
      ISSN = {1292-8119},
   review={\MR{3764132}},
       DOI = {10.1051/cocv/2016073},
}



\bib{EGL}{article}{
   author={Edmunds, David E.},
   author={Gurka, Petr},
   author={Lang, Jan},
   title={Properties of generalized trigonometric functions},
   journal={J. Approx. Theory},
   volume={164},
   date={2012},
   number={1},
   pages={47--56},
   issn={0021-9045},
   review={\MR{2855769}},
   doi={10.1016/j.jat.2011.09.004},
}



\bib{Euler}{article}{
author = {Eulerus, Leonhardus},
title = {
Methodus inveniendi lineas curvas maximi minimive proprietate gaudentes, sive solutio problematis isoperimetrici lattissimo sensu accepti
},
journal = {Euler Archive. All works},
volume = {65},
   date={1744},
 url= {https://scholarlycommons.pacific.edu/euler-works/65/}, 
}

\bib{EvGar}{book}{
   author={Evans, Lawrence C.},
   author={Gariepy, Ronald F.},
   title={Measure theory and fine properties of functions},
   series={Textbooks in Mathematics},
   edition={Revised edition},
   publisher={CRC Press, Boca Raton, FL},
   date={2015},
   pages={xiv+299},
   isbn={978-1-4822-4238-6},
   review={\MR{3409135}},
}


\bib{GrunauOkabe}{article}{
author={Grunau, Hans-Christoph},
author={Okabe, Shinya},
title={Willmore obstacle problems under Dirichlet boundary conditions},
journal={Ann. Sc. Norm. Super. Pisa Cl. Sci. (5)},
volume={24},
date={2023},
number={3},
pages={1415--1462},
   review={\MR{4675964}},
DOI = {10.2422/2036-2145.202105\_064},
}




\bib{Kinderlehrer}{book}{
   author={Kinderlehrer, David},
   author={Stampacchia, Guido},
   title={An introduction to variational inequalities and their
   applications},
   series={Classics in Applied Mathematics},
   volume={31},
   note={Reprint of the 1980 original},
   publisher={Society for Industrial and Applied Mathematics (SIAM),
   Philadelphia, PA},
   date={2000},
   pages={xx+313},
   isbn={0-89871-466-4},
   review={\MR{1786735}},
   doi={10.1137/1.9780898719451},
}

\bib{KT}{article}{
   author={Kobayashi, Hiroyuki},
   author={Takeuchi, Shingo},
   title={Applications of generalized trigonometric functions with two
   parameters},
   journal={Commun. Pure Appl. Anal.},
   volume={18},
   date={2019},
   number={3},
   pages={1509--1521},
   issn={1534-0392},
   review={\MR{3917717}},
   doi={10.3934/cpaa.2019072},
}

\bib{LangerSinger}{article}{
   author={Langer, Joel},
   author={Singer, David A.},
   title={The total squared curvature of closed curves},
   journal={J. Differential Geom.},
   volume={20},
   date={1984},
   number={1},
   pages={1--22},
   issn={0022-040X},
   review={\MR{772124}},
}

\bib{Lewy}{article}{
   author={Lewy, Hans},
   author={Stampacchia, Guido},
   title={On existence and smoothness of solutions of some non-coercive
   variational inequalities},
   journal={Arch. Rational Mech. Anal.},
   volume={41},
   date={1971},
   pages={241--253},
   issn={0003-9527},
   review={\MR{346313}},
   doi={10.1007/BF00250528},
}	

\bib{LiebLoss}{book}{
   author={Lieb, Elliott H.},
   author={Loss, Michael},
   title={Analysis},
   series={Graduate Studies in Mathematics},
   volume={14},
   edition={2},
   publisher={American Mathematical Society, Providence, RI},
   date={2001},
   pages={xxii+346},
   isbn={0-8218-2783-9},
   review={\MR{1817225}},
   doi={10.1090/gsm/014},
}


 
\bib{LP22}{article}{
	author = {L\'{o}pez, Rafael},
    author = {P\'{a}mpano, \'{A}lvaro}, 
	journal = {Nonlinear Anal.},
	pages = {Paper No. 112661, 22},
	title = {Stationary soap films with vertical potentials},
	volume = {215},
	year = {2022}
}



\bib{Mandel}{article}{
   author={Mandel, Rainer},
   title={Explicit formulas, symmetry and symmetry breaking for Willmore
   surfaces of revolution},
   journal={Ann. Global Anal. Geom.},
   volume={54},
   date={2018},
   number={2},
   pages={187--236},
   issn={0232-704X},
   review={\MR{3850028}},
   doi={10.1007/s10455-018-9598-0},
}


\bib{Moist}{article}{
   author={Miura, Tatsuya},
   title={Polar tangential angles and free elasticae},
   journal={Math. Eng.},
   volume={3},
   date={2021},
   number={4},
   pages={Paper No. 034, 12},
   review={\MR{4147578}},
   doi={10.3934/mine.2021034},
}

\bib{TatsuyaKensuke1}{article}{
author={Miura, Tatsuya},
author={Yoshizawa, Kensuke},
title={Complete classification of planar $p$-elasticae},
journal={to appear in Ann. Mat. Pura Appl.}, 
doi = {10.1007/s10231-024-01445-z},
}

\bib{TatsuyaKensuke2}{article}{
author={Miura, Tatsuya},
author={Yoshizawa, Kensuke}, 
title={Pinned planar $p$-elasticae},
journal={Preprint (arXiv:2209.05721), to appear in Indiana Univ. Math. J.}, 
}

\bib{TatsuyaKensuke3}{article}{
author={Miura, Tatsuya},
author={Yoshizawa, Kensuke}, 
title={Variational stabilization of degenerate $p$-elasticae},
journal={Preprint (arXiv:2310.07451)}, 
date={2023},
}

\bib{TatsuyaKensuke4}{article}{
author={Miura, Tatsuya},
author={Yoshizawa, Kensuke}, 
title={General rigidity principles for stable and minimal elastic curves},
   journal={J. Reine Angew. Math.},
   volume={810},
   date={2024},
   pages={253--281},
   doi={doi:10.1515/crelle-2024-0018},
}

\bib{MariusExistence}{article}{
   author={M\"{u}ller, Marius},
   title={An obstacle problem for elastic curves: existence results},
   journal={Interfaces Free Bound.},
   volume={21},
   date={2019},
   number={1},
   pages={87--129},
   issn={1463-9963},
   review={\MR{3951579}},
   doi={10.4171/IFB/418},
}

\bib{Marius_2020}{article}{
   author={M\"{u}ller, Marius},
   title={On gradient flows with obstacles and Euler's elastica},
   journal={Nonlinear Anal.},
   volume={192},
   date={2020},
   pages={111676, 48},
   issn={0362-546X},
   review={\MR{4031155}},
   doi={10.1016/j.na.2019.111676},
}
		
\bib{MariusFlow}{article}{
   author={M\"{u}ller, Marius},
   title={The Elastic Flow with Obstacles: Small Obstacle Results},
   journal={Appl. Math. Optim.},
   volume={84},
   date={2021},
   number={1, suppl.},
   pages={355--402},
   issn={0095-4616},
   review={\MR{4316789}},
   doi={10.1007/s00245-021-09773-9},
}

\bib{NP20}{article}{
	author = {Novaga, Matteo},
	author = {Pozzi, Paola},
	journal = {SIAM J. Math. Anal.},
	number = {1},
	pages = {682--708},
	title = {A second order gradient flow of {$p$}-elastic planar networks},
	volume = {52},
	year = {2020}
 }


\bib{OPW20}{article}{
	author = {Okabe, Shinya},
	author = {Pozzi, Paola},
	author = {Wheeler, Glen},
	journal = {Math. Ann.},
	number = {1-2},
	pages = {777--828},
	title = {A gradient flow for the {$p$}-elastic energy defined on closed planar curves},
	volume = {378},
	year = {2020}
 }


\bib{OW23}{article}{
	author = {Okabe, Shinya}, author = {Wheeler, Glen},
	journal = {J. Math. Pures Appl. (9)},
	pages = {1--42},
	title = {The {{\(p\)}}-elastic flow for planar closed curves with constant parametrization},
	volume = {173},
	year = {2023}
}


\bib{OY}{article}{
   author={Okabe, Shinya},
   author={Yoshizawa, Kensuke},
   title={A dynamical approach to the variational inequality on modified
   elastic graphs},
   journal={Geom. Flows},
   volume={5},
   date={2020},
   number={1},
   pages={78--101},
   review={\MR{4171115}},
   doi={10.1515/geofl-2020-0100},
}

\bib{Ponce}{book}{
   author={Ponce, Augusto C.},
   title={Elliptic PDEs, measures and capacities},
   series={EMS Tracts in Mathematics},
   volume={23},
   note={From the Poisson equations to nonlinear Thomas-Fermi problems},
   publisher={European Mathematical Society (EMS), Z\"{u}rich},
   date={2016},
   pages={x+453},
   isbn={978-3-03719-140-8},
   review={\MR{3675703}},
   doi={10.4171/140},
}

\bib{Poz22}{article}{
	author = {Pozzetta, Marco},
	journal = {Nonlinear Anal.},
	pages = {Paper No. 112581, 53},
	title = {Convergence of elastic flows of curves into manifolds},
	volume = {214},
	year = {2022}
}

\bib{Talenti}{article}{
    AUTHOR = {Talenti, Giorgio},
     TITLE = {Elliptic equations and rearrangements},
   JOURNAL = {Ann. Scuola Norm. Sup. Pisa Cl. Sci. (4)},
  FJOURNAL = {Annali della Scuola Normale Superiore di Pisa. Classe di
              Scienze. Serie IV},
    VOLUME = {3},
      YEAR = {1976},
    NUMBER = {4},
     PAGES = {697--718},
      ISSN = {0391-173X},
   MRCLASS = {35J35 (35P15)},
  MRNUMBER = {601601},
MRREVIEWER = {Vladimir A. Kondratiev},
       URL = {http://www.numdam.org/item?id=ASNSP_1976_4_3_4_697_0},
}

\bib{SW20}{article}{
    author = {Shioji, Naoki},
    AUTHOR = {Watanabe, Kohtaro},
     TITLE = {Total {$p$}-powered curvature of closed curves and flat-core
              closed {$p$}-curves in {${\rm S}^2(G)$}},
   JOURNAL = {Comm. Anal. Geom.},
    VOLUME = {28},
      YEAR = {2020},
    NUMBER = {6},
     PAGES = {1451--1487},
      ISSN = {1019-8385},
    review={\MR{4184824}},
       DOI = {10.4310/CAG.2020.v28.n6.a6},
       URL = {https://doi.org/10.4310/CAG.2020.v28.n6.a6},
}



\bib{Watanabe}{article}{
   author={Watanabe, Kohtaro},
   title={Planar $p$-elastic curves and related generalized complete
   elliptic integrals},
   journal={Kodai Math. J.},
   volume={37},
   date={2014},
   number={2},
   pages={453--474},
   issn={0386-5991},
   review={\MR{3229087}},
   doi={10.2996/kmj/1404393898},
}

\bib{Y2021}{article}{
AUTHOR = {Yoshizawa, Kensuke},
     TITLE = {A remark on elastic graphs with the symmetric cone obstacle},
   JOURNAL = {SIAM J. Math. Anal.},
    VOLUME = {53},
      YEAR = {2021},
    NUMBER = {2},
     PAGES = {1857--1885},
      ISSN = {0036-1410},
   MRCLASS = {49J40 (34B15 49J05 53A04 53E99)},
  review={\MR{4238008}},
MRREVIEWER = {Graziano Crasta},
       DOI = {10.1137/19M1307901},
       URL = {https://doi.org/10.1137/19M1307901},
}
\bib{KensukeNavier}{article}{
   author={Yoshizawa, Kensuke},
   title={The critical points of the elastic energy among curves pinned at
   endpoints},
   journal={Discrete Contin. Dyn. Syst.},
   volume={42},
   date={2022},
   number={1},
   pages={403--423},
   issn={1078-0947},
   review={\MR{4349792}},
   doi={10.3934/dcds.2021122},
}

\end{thebibliography}
\end{document}